\documentclass[11pt,reqno]{amsart}

\usepackage[T1]{fontenc}
\usepackage{amsmath, amssymb, amsthm, amsfonts}
\usepackage{mathtools}
\usepackage[dvipsnames]{xcolor}

\definecolor{DartmouthGreen}{HTML}{00693E}
\definecolor{DartmouthDarkGreen}{HTML}{004B23}
\definecolor{DartmouthLightGreen}{HTML}{4BAF7A}

\usepackage[
  colorlinks=true,
  hyperindex,
  linkcolor=DartmouthDarkGreen,
  citecolor=DartmouthGreen,
  urlcolor=DartmouthGreen,
  pagebackref=true
]{hyperref}

\usepackage{fullpage}
\usepackage{tikz,tikz-cd}
\usetikzlibrary{graphs, graphs.standard,positioning}

\usepackage{colortbl, array,booktabs,tabularx}
\usepackage[shortlabels]{enumitem}

\usepackage{calrsfs}
\DeclareMathAlphabet{\pazocal}{OMS}{zplm}{m}{n}
\usepackage{newcent,fouriernc}

\newtheorem{theorem}{Theorem}[section]
\newtheorem{lemma}[theorem]{Lemma}
\newtheorem{theoremalpha}{Theorem}
\newtheorem{corollaryalpha}[theoremalpha]{Corollary}

\newtheorem{corollary}[theorem]{Corollary}
\newtheorem{proposition}[theorem]{Proposition}
\newtheorem{conjecture}[theorem]{Conjecture}

\theoremstyle{definition}

\newtheorem{definition}[theorem]{Definition}

\newtheorem{example}[theorem]{Example}

\newtheorem{remark}[theorem]{Remark}

\newcommand{\dd}{\mathbf{d}}

\newcommand{\mm}{\mathbf{m}}

\newcommand{\vv}{\mathbf{v}}

\newcommand{\pp}{\mathbf{p}}
\newcommand{\qq}{\mathbf{q}}

\newcommand{\one}{\mathbf{1}}

\renewcommand{\AA}{\mathbb{A}}
\newcommand{\FF}{\mathbb{F}}

\newcommand{\GG}{\mathbb{G}}
\newcommand{\KK}{\mathbb{K}}
\newcommand{\LL}{\mathbb{L}}

\newcommand{\PP}{\mathbb{P}}
\newcommand{\QQ}{\mathbb{Q}}
\newcommand{\RR}{\mathbb{R}}
\newcommand{\ZZ}{\mathbb{Z}}
\newcommand{\CC}{\mathbb{C}}
\newcommand{\TT}{\mathbb{T}}
\newcommand{\VV}{\mathbb{V}}

\newcommand{\cA}{\mathcal{A}}

\newcommand{\cH}{\mathcal{H}}

\newcommand{\cI}{\mathcal{I}}
\newcommand{\cL}{\mathcal{L}}
\newcommand{\cN}{\mathcal{N}}

\newcommand{\cO}{\mathcal{O}}
\newcommand{\cP}{\mathcal{P}}
\newcommand{\cR}{\mathcal{R}}

\newcommand{\cV}{\mathcal{V}}

\newcommand{\pB}{\pazocal{B}}
\newcommand{\pC}{\pazocal{C}}
\newcommand{\pD}{\pazocal{D}}

\newcommand{\pY}{\pazocal{Y}}
\newcommand{\pS}{\pazocal{S}}

\newcommand{\pU}{\pazocal{U}}

\newcommand{\pZ}{\pazocal{Z}}

\newcommand{\sD}{\mathsf{D}}
\newcommand{\sE}{\mathsf{E}}
\newcommand{\sL}{\mathsf{L}}
\newcommand{\sI}{\mathsf{I}}

\newcommand{\sS}{\mathsf{S}}

\newcommand{\sP}{\mathsf{P}}

\newcommand{\se}{\mathsf{e}}

\newcommand{\mfm}{\mathfrak{m}}

\renewcommand{\tilde}{\widetilde}

\DeclareMathOperator{\Sym}{Sym}
\DeclareMathOperator{\Div}{Div}

\DeclareMathOperator{\mult}{mult}
\DeclareMathOperator{\Bl}{Bl}
\DeclareMathOperator{\Pic}{Pic}
\DeclareMathOperator{\val}{val}
\DeclareMathOperator{\an}{an}
\DeclareMathOperator{\Spec}{Spec}
\DeclareMathOperator{\inn}{in}
\DeclareMathOperator{\gr}{gr}
\DeclareMathOperator{\bSpec}{\mathbf{Spec}}
\DeclareMathOperator{\supp}{supp}
\DeclareMathOperator{\Conf}{Conf}
\DeclareMathOperator{\red}{red}
\DeclareMathOperator{\Span}{span}
\DeclareMathOperator{\pr}{pr}
\DeclareMathOperator{\Rees}{Rees}
\DeclareMathOperator{\ord}{ord}
\DeclareMathOperator{\ev}{ev}
\renewcommand{\Div}{\mathrm{Div}}
\renewcommand{\div}{\mathrm{div}}
\newcommand{\SI}{\mathsf{Exp}}
\DeclareMathOperator{\len}{len}
\DeclareMathOperator{\Sp}{sp}
\DeclareMathOperator{\img}{img}
\DeclareMathOperator{\codim}{codim}
\DeclareMathOperator{\NE}{NE}
\DeclareMathOperator{\Hilb}{Hilb}
\DeclareMathOperator{\mom}{mom}

\newcommand{\Gap}{\varepsilon_{\mathrm{gap}}}

\newcommand{\Bad}{\Sigma^{\mathrm{bad}}}
\newcommand{\Tay}{\mathrm{Tay}_{\one}}

\usepackage{graphicx}
\usepackage{pgfplots}
\pgfplotsset{compat=1.17}
\usetikzlibrary{arrows.meta,calc,decorations.pathreplacing}
 
\definecolor{FigRed}{HTML}{A4343A}
\definecolor{FigGold}{HTML}{C28E0E}
\definecolor{FigBlue}{HTML}{2F6E9E}
\definecolor{FigGray}{HTML}{7A7A7A}
 
\tikzset{
  fig/.style={line cap=round, line join=round, >={Stealth[length=2.2mm]}},
  pt/.style={circle, fill=#1, inner sep=0pt, minimum size=3.2pt},
  pt/.default=black,
  smallpt/.style={circle, fill=#1, inner sep=0pt, minimum size=2.2pt},
  smallpt/.default=black!35,
  vtx/.style={circle, draw=black, fill=white, line width=0.6pt, inner sep=0pt, minimum size=5pt},
  lbl/.style={font=\footnotesize},
  tinylbl/.style={font=\scriptsize},
}

\numberwithin{equation}{section}

\title{A Tropical Viewpoint on the Nagata--Biran--Szemberg Conjecture}

\author{Juliette Bruce}
\address{Department of Mathematics, Dartmouth College, Hanover, NH}
\email{\href{mailto:juliette.bruce@dartmouth.edu}{juliette.bruce@dartmouth.edu}}
\urladdr{\url{https://www.juliettebruce.xyz}}


\begin{document}

\begin{abstract}

	We give a tropical perspective on the recent arguments announced by OpenAI for Nagata's conjecture and the Nagata--Biran--Szemberg conjecture. By degenerating to both the normal cone and a semistable model whose dual graph is a cycle, we replace sections on a surface by Laurent polynomials in two variables. In place of theta functions, tropical methods bound the exponents of these Laurent polynomials, and a lift of the tropical Abel--Jacobi class of the normal bundle controls the shape of the resulting exponent diagram. Combining these support bounds with interpolation on a torus proves the Nagata--Biran--Szemberg conjecture with its predicted threshold whenever the corresponding linear series contains an integral curve whose only singularity is an ordinary node that is not a basepoint.
 \end{abstract}

\maketitle
\setcounter{tocdepth}{1}
\tableofcontents

%%%%%%%%%%%%%%%%%%%%%%%%%%%%%%%%%%%%%%%%%%%%%%%%%%%%%%%%%%%%%%%%
\section{Introduction}
%%%%%%%%%%%%%%%%%%%%%%%%%%%%%%%%%%%%%%%%%%%%%%%%%%%%%%%%%%%%%%%%

In constructing counterexamples to Hilbert's 14th problem, Nagata considered the following question: Let $p_1,\ldots,p_r \in \PP^2$ be a finite set of $r\geq10$ very general points in the projective plane. By very general points we mean that  the points lie outside a specified countable union of proper Zariski closed subsets. If $C \subset \PP^2$ is a plane curve of degree $d\geq1$, is it the case that
\[
\sum_{i=1}^{r} \mult_{p_i}\left(C\right)< d \sqrt{r}?
\]
Here $\mult_{p}(C)$ denotes the multiplicity of the curve $C$ at $p$: if $f$ is a local equation for $C$ in coordinates centered at $p$, it is the least total degree of a nonzero term of $f$. Nagata proved the strict inequality when $r\geq 10$ is a perfect square \cite{nagata59}, and the general statement has become known as Nagata's conjecture. This celebrated conjecture has attracted significant attention \cite{nagata60,harbourne85,xu94,cilibertoMiranda98, biran99,harbourne01,roe01,szembergTutaj02,cilibertoMiranda06,harbourneRoe09,syzdekSzemberg10,cilibertoMiranda11,kalinin17} and has shown connections to various areas of mathematics: symplectic packing problems \cite{mcduffPolterovich94,biran97,biran99,maleyMastrangeli00,dioneeRoth25}, Cox rings, Mori dream spaces, and the birational geometry of $\overline{M}_{0,n}$ \cite{huKeel00,mukai04,castravetTevelev06,mckernan10,gonzalezEtAl12,castravetTevelev15}, Seshadri constants \cite{demailly92,einKuchleLazarsfeld95,bauer99,harbourne03,bauerEtAl09}, interpolation and the Segre--Harbourne--Gimigliano--Hirschowitz conjecture \cite{segre61,harbourne86,gimigliano87,hirschowitz89,alexanderHirschowitz00,fernex05,DKMS16}, symbolic powers of ideals of points \cite{harbourne00,harbourneRoe04,bocciHarbourne10}, and numerous generalizations \cite{cilibertoEtAl13,dumnickiEtAl14,nivoche21}. See \cite{biran01,strycharzSzembergSzemberg04,gonzalezLaface22} for surveys of these areas.

Nagata's conjecture is a special case of a broader question about local positivity on surfaces, known as the Nagata--Biran--Szemberg conjecture. Let $X$ be a smooth projective connected surface defined over $\KK$. Throughout we shall assume that $\KK$ is an algebraically closed uncountable field of characteristic zero. Let $H$ be an ample divisor on $X$, and set $h=\langle H\cdot H\rangle$. For an integer $r\geq1$ let $\Conf_r(X)$ denote the configuration space of $r$ distinct ordered points on $X$. For a point $\pp=(p_1,\ldots,p_r) \in \Conf_{r}(X)$ let
\[
\begin{tikzcd}[row sep = .2em, column sep = 3em]
X_{\pp}\coloneqq \Bl_{\pp}(X) \arrow[r, "\mu_{\pp}"] & X
\end{tikzcd}
\]
be the blow-up of $X$ at the points $p_1,\ldots,p_r$. Let $E_i\subset X_{\pp}$ denote the exceptional divisor lying over the point $p_i$. Each $E_{i}$ is a $(-1)$-curve on $X_{\pp}$, with $E_{i}$ and $E_{j}$ being disjoint for $i\neq j$. The (total) exceptional divisor $E$ on $X_{\pp}$ is then $E=E_1+\cdots+E_r$. When confusion seems unlikely, we will write $H$ both for an ample line bundle on $X$ and for the pullback $\mu^*_{\pp}(H)$. The following multipoint generalization of Demailly's Seshadri constants \cite[\S~6]{demailly92} was introduced in \cite{Kuchle96}; see also \cite[Def.~1.9]{bauerEtAl09}.

\begin{definition}
    The \emph{(multipoint) Seshadri constant} of the ample line bundle $H$ at $\pp\in \Conf_{r}(X)$ is
    \[
    \varepsilon\left(X,H;\pp\right) \coloneqq \sup\left\{ \lambda \in \RR_{\geq0} \;\; \big| \;\; H -\lambda E \text{ is nef on $X_{\pp}$}\right\}.
    \]
\end{definition}

Seshadri constants are closely connected to Nagata's conjecture because if $\lambda>0$ then the $\RR$-divisor $H-\lambda E$ being nef on $X_{\pp}$ is equivalent to
\[
\sum_{i=1}^{r}\mult_{p_i}(C) \leq \frac{1}{\lambda} \langle H\cdot C\rangle
\]
for every curve $C\subset X$; see Lemma~\ref{lem:nagata-to-seshadri} and \cite[p.~417]{strycharzSzembergSzemberg04}. For $\PP^2$, with $H$ the class of a line, this shows that Nagata's conjecture (in its non-strict form) is equivalent to saying that the Seshadri constant of $(\PP^2,H)$ is $\sqrt{1/r}$ for $r\geq10$ very general points. 

The value $\sqrt{1/r}$ predicted by Nagata's conjecture has a nice meaning: it is the largest the Seshadri constant can be. Indeed, on a surface, if $H-\lambda E$ is nef, then $\left\langle (H-\lambda E)^2 \right\rangle = h-r\lambda^2 \geq0$, implying $\lambda \leq \sqrt{h/r}$. Taking the supremum gives
\[
\varepsilon(X,H;\pp) \leq  \sqrt{\frac{h}{r}}.
\]
 The Nagata--Biran--Szemberg conjecture, also called the Nagata--Biran conjecture, predicts that this upper bound is attained for sufficiently many very general points on a polarized surface $(X,H)$ \cite{biran99,biran99-2}. We will consider the following version of the Nagata--Biran--Szemberg conjecture that gives an explicit threshold for when maximality occurs; see \cite[Rem.~5.1.24]{lazarsfeld04-I} and \cite[Conj.~2.1]{strycharzSzembergSzemberg04}.

\begin{conjecture}[Nagata--Biran--Szemberg Conjecture]\label{con:general-nagata}
    Let $X$ be a smooth projective connected surface, and let $H$ be an ample divisor on $X$ with $h=\langle H^2\rangle$. Let $\nu\geq1$ be an integer such that $|\nu H|$ contains a smooth connected curve of positive genus. For every integer $r\geq r_0\coloneqq \nu^2h$,
    \[
    \varepsilon\left(X,H;\pp\right)=\sqrt{\frac{h}{r}}
    \]
    for a very general point $\pp\in \Conf_{r}(X)$.
\end{conjecture}

While until recently the conjecture was largely open, asymptotic versions were known. In particular, when $H$ is very ample Harbourne gives lower bounds on $\varepsilon(X,H;\pp)$ whose ratio with $\sqrt{h/r}$ tends to one as  $r\to \infty$ \cite[Thm.~I.1]{harbourne03}. If $X$ does not admit a fibration to a curve, Syzdek and Szemberg give the lower bound of $\sqrt{(r-1)/r}\sqrt{h/r}$ for $r\geq2$ very general points \cite[Thm.~4.1]{syzdekSzemberg10}. 

In their October 6, 2026 release of mathematical work, OpenAI announced resolutions of Nagata's conjecture \cite{openAI-P2} and a qualitative version of the Nagata--Biran--Szemberg conjecture without an explicit threshold\footnote{While I am not certain, it seems highly likely that one could extract explicit thresholds directly from the arguments in \cite{openAI-NB}.}  \cite{openAI-NB}. (Both works are dated September 23, 2026.) The argument in \cite{openAI-P2} used to prove Nagata's conjecture for $\PP^2$ passes from interpolation on the normal bundle to a lattice point count by degenerating theta bases. In \cite{openAI-NB} asymptotic maximality of multipoint Seshadri constants on a polarized surface is proved using weighted Taylor expansions and finite torus subgroups. We give a modest reinterpretation of the ideas behind the first approach using tropical geometry. We replace theta bases by support bounds for residual sections on a rational component of a semistable curve whose dual graph is a cycle. This carries the plane argument's interpolation mechanism to higher-genus curves, which yields the predicted threshold.  While finishing this note, Tevelev~\cite{tevelev26} posted a lovely article giving a complex-geometric account of Nagata's conjecture using theta functions and interpolation on the normal ruled surface of a smooth plane cubic.

%%%%%%%%%%%%%%%%%%%%%%%%%%%%%%%%%%%%%%%%%%%%%%%%%%%%%%%%%%%%%%%%
\subsection{Main Results}
%%%%%%%%%%%%%%%%%%%%%%%%%%%%%%%%%%%%%%%%%%%%%%%%%%%%%%%%%%%%%%%%

The main theorem proves the Nagata--Biran--Szemberg conjecture, with its predicted threshold, whenever $|\nu H|$ contains an integral curve whose only singular point is one ordinary node that is not a basepoint of $|\nu H|$.

\begin{theoremalpha}\label{main-thm:technical-nagata}
    Let $X$ be a smooth projective connected surface defined over $\KK$ and let $H$ be an ample divisor on $X$ with $h\coloneqq\langle H^2\rangle$. Let $\nu\geq1$ be an integer such that $|\nu H|$ contains an integral curve $B_0$ whose only singularity is one ordinary node $o$, such that $o$ is not a basepoint of $|\nu H|$.
    For every integer $r \geq \nu^2h$ there exists  $\Bad \subset \Conf_{r}(X)$ that is a countable union of proper closed subsets such that, for every $\pp\in \Conf_{r}(X)\setminus \Bad$,
        \[
    \varepsilon\left(X,H;\pp\right)=\sqrt{\frac{h}{r}}.
    \]
\end{theoremalpha}

The smoothing constructed in Section~\ref{subsec:nodal-curve}, together with the genus formula for normalization, shows that any linear series $|\nu H|$ that satisfies the hypotheses of Theorem~\ref{main-thm:technical-nagata} contains a smooth curve of positive genus. Thus, Theorem~\ref{main-thm:technical-nagata} proves the Nagata--Biran--Szemberg conjecture under these assumptions. However, there are linear series that contain a smooth curve of positive genus, but no  integral curves with a single node, e.g., the complete linear series of a smooth theta divisor on a principally polarized abelian surface. 

Nevertheless, Lemma~\ref{lem:2H-nodal-curve-exists} shows that if $H$ is very ample and $(X,H)\neq(\PP^2,\cO_{\PP^2}(1))$  then $|2H|$ always contains an integral curve whose only singularity is one ordinary node. This allows us to prove that the Nagata--Biran--Szemberg conjecture holds for all polarized surfaces with a worse threshold. 

\begin{theoremalpha}\label{main-thm:general-nagata}
    Let $X$ be a smooth projective connected surface defined over $\KK$ and let $H$ be an ample divisor on $X$ with $h\coloneqq\langle H^2\rangle$. Let  $k\geq1$ be an integer such that $kH$ is very ample and $(X, kH)\neq(\PP^2,\cO_{\PP^2}(1))$.  For every integer $r \geq 4k^2h$ there exists  $\Bad \subset \Conf_{r}(X)$ that is a countable union of proper closed subsets such that for every $\pp\in \Conf_{r}(X)\setminus \Bad$:
	\[
    \varepsilon\left(X,H;\pp\right)=\sqrt{\frac{h}{r}}.
    \]
\end{theoremalpha}

Without the explicit bound $r\geq 4k^2h$, Theorem~\ref{main-thm:general-nagata} is the main theorem of \cite{openAI-NB}; what is new here is the explicit threshold. When $k=1$, excluding  $(\PP^2,\cO_{\PP^2}(1))$ is necessary. In this case $h=1$ and the threshold would be $r\geq 4$, but five points in $\PP^2$ lie on a conic and $5>2\sqrt{5}$. Similar arguments show that the bound is false for  $r=5,6,7,8$, but it is true for $r=4$. See \cite{nagata60}, \cite[Ex.~2.4]{strycharzSzembergSzemberg04}, and \cite[\S~1]{harbourne01}.

We record a few particularly nice cases of Theorem~\ref{main-thm:technical-nagata} below. 

\begin{corollaryalpha}\label{main-cor:examples}
    In each of the following examples Conjecture~\ref{con:general-nagata} is true with the given $r_0$:
    \begin{enumerate}
        \item $X=\PP^2$, $H=\cO_{\PP^2}(1)$, and $r_0=9$.
        \item $X=\PP^1\times \PP^1$, $H=\cO_{\PP^1\times\PP^1}(1,1)$, and $r_0=8$.
        \item $X \subset \PP^3$ a smooth surface of degree $d\geq3$, $H$ a hyperplane section of $X$, and $r_0=d$.
    \end{enumerate}
\end{corollaryalpha}

In each case $\nu$ is the smallest value for which Conjecture~\ref{con:general-nagata} applies, and so these are precisely the thresholds the conjecture predicts. Part (3) relies on the fact that when $d\geq3$ such surfaces have a hyperplane section that is an integral curve with only one ordinary node \cite[Prop.~3.5]{katz73}.  Part (2) of Corollary~\ref{main-cor:examples} with $r=10$ was recently proven over $\CC$ in \cite{lafaceUgaglia26} by using dihedral orbits and the reflection approach of Dionne and Roth \cite{dioneeRoth25}. 

Returning to the opening question, part (1) of Corollary~\ref{main-cor:examples} is the non-strict form of Nagata's conjecture, announced in \cite{openAI-P2} and explained in \cite{tevelev26}. A short traditional argument gives the strict form. 

\begin{theoremalpha}\label{main-thm:nagata-I}
    Let $r\geq 10$. There exists $\Bad \subset \Conf_{r}(\PP^2)$ that is a countable union of proper Zariski closed subsets such that for every $\pp\in \Conf_{r}(\PP^2)\setminus \Bad$ and every plane curve $C\subset \PP^2$ of degree $d\geq1$:
    \[
    \sum_{i=1}^{r}\mult_{p_i}(C) < d\sqrt{r}.
    \]
\end{theoremalpha}

We deduce Theorem~\ref{main-thm:nagata-I} with the non-strict inequality from Theorem~\ref{main-thm:technical-nagata} with $\nu=3$ and a nodal cubic. When $r$ is not a perfect square, $d\sqrt{r}$ is irrational while the left-hand side is an integer, so the inequality is strict. When $r$ is a perfect square we invoke Nagata's original theorem \cite{nagata59} together with a symmetrization argument (Lemma~\ref{lem:sym}) to prove the stronger inequality. 

%%%%%%%%%%%%%%%%%%%%%%%%%%%%%%%%%%%%%%%%%%%%%%%%%%%%%%%%%%%%%%%%
\subsection{Approach of the Proof}
%%%%%%%%%%%%%%%%%%%%%%%%%%%%%%%%%%%%%%%%%%%%%%%%%%%%%%%%%%%%%%%%

From one thousand feet, the proof of Theorem~\ref{main-thm:technical-nagata} consists of four reductions. First, we pass from a statement about nefness to a vanishing statement. Second, we pass from a question about vanishing on our surface to a vanishing question on the normal bundle of a curve. Third, we simplify the sections we are considering by deforming the underlying curve in a semistable model allowing us to replace the sections with Laurent polynomials in two variables. Finally, we reduce the resulting interpolation problem on a torus to a combinatorial problem. In the end, we are left to check  that one triangle fits inside another triangle. 

This is the approach given in \cite{openAI-P2}, and explained beautifully in \cite{tevelev26}, to prove Nagata's conjecture for $\PP^2$. The main new input here is in the third reduction: the use of semistable degenerations to reduce the problem to one about Laurent polynomials on a torus, and then the use of tropical methods to bound the exponents  of these polynomials and adjust the shape of the resulting subset of $\ZZ^2$. This replaces the theta-function calculations in \cite{openAI-P2,tevelev26} and works equally well for curves of higher genus.

\subsubsection{First Reduction: From Nefness to Vanishing} To prove the theorem we must show that $H-\lambda E$ is nef for all $0\leq \lambda \leq \sqrt{h/r}$ for a very general set of points. Checking this directly would involve every curve on $X_{\pp}$. However, if $r>\nu^2h$ a Riemann--Roch argument reduces this problem to proving that, for every pair of positive integers $(d,m)$ 
\[
H^0\left(X_{\pp}, dH-mE\right) = 0 \quad \quad \text{whenever \quad \quad }\nu < \frac{d}{m} < \sqrt{\frac{r}{h}}
\] 
for a general point $\pp\in \Conf_{r}(X)$. The idea is that a curve meeting $H-\sqrt{h/r}E$ negatively would force $dH-mE$ to have a section with $d/m$ arbitrarily close to $\sqrt{r/h}$; see Lemma~\ref{lem:vanishing-to-nefness}. Phrased differently, Conjecture~\ref{con:general-nagata} is equivalent to $H-\sqrt{h/r}E$ not being big \cite[p.~418]{strycharzSzembergSzemberg04}. We handle the remaining case $r=\nu^2h$ separately by Lemma~\ref{lem:endpoints}. 

The reduction has two advantages.  First, there are only countably many pairs $(d,m)$. For each pair it is enough to prove the vanishing on a nonempty open subset of $\Conf_{r}(X)$; we then define $\Bad$ to be the union of the corresponding exceptional sets. Second, $H^0\left(X_{\pp}, dH-mE\right)$ is the kernel of the jet evaluation map
\[
\begin{tikzcd}[row sep = .2em, column sep  =4em]
H^0\left(X, \cO_{X}(dH)\right) \arrow[r, "\ev"] & H^0\left(Z_{\pp}(m^r), \cO_{X}(dH)|_{Z_{\pp}(m^r)}\right)
\end{tikzcd}
\]
where $Z_{\pp}(m^r)$ is the union of fat points defined by $\cI_{p_i}^m$. The required vanishing is equivalent to this map being injective for a general point $\pp$. By a rough dimension count, the source has dimension $\approx hd^2/2$ and the target has dimension $\approx r m^2/2$, which suggests the inequality $h(d/m)^2 < r$. Since injectivity is open in $\pp$, it is enough to construct one point for which this map is injective. 

\subsubsection{Second Reduction: From the Surface to the Normal Bundle} A smoothing of $B_0$ gives a smooth geometrically integral curve $B\in |\nu H|$ of positive genus over $K=\KK((s))$. We would like to use this curve to simplify the vanishing statements. Restricting sections to $B$ would lose all the sections that vanish along $B$. To retain this information, we also record their first nonzero term in the normal direction to $B$. 

This is made precise using a variant of the classical deformation to the normal cone \cite[\S~5]{fulton98}. In general, let $S$ be a smooth projective geometrically connected surface, let $B\subset S$ be a smooth curve, and let $L$ be a line bundle on $S$. We will apply this when $S=X_{K}$ and $L=\cO_{X_K}(dH)$.  Set $V=H^0(S,L)$, and let $\pi:Y\to B$ be the total space of the normal bundle of $B$. For $j\geq0$, set
\[
V_{B,j} \coloneqq \img\left(
\begin{tikzcd}[row sep = .2em, column sep = 3.5em]
H^0\left(S,L(-jB)\right) \arrow[r, "\rho_j"] & H^0\left(B, L|_{B} \otimes N_{B/S}^{-j}\right)
\end{tikzcd}\right).
\]
These are the coefficient spaces of the associated graded for the filtration of $V$ by order of vanishing along $B$.  If $\tau$ is the tautological section of $\pi^*N_{B/S}$ then the initial linear series is
\[
V_{B}^{\inn}\coloneqq \bigoplus_{j\geq0} \pi^*(V_{B,j})\tau^j \subset H^0\left(Y,\pi^*(L|_{B})\right).
\]
This is the image of the associated graded  under the restriction map, so $\dim V_{B}^{\inn} = \dim V$. Suppressing pullbacks and trivializations, the sections of $V_{B}^{\inn}$ have the form  $\sum_{j\geq0} g_j u^j$ where $g_j\in V_{B,j}$ and $u$ is the fiber coordinate. 

Deformation to the normal cone builds a family $\pS\to \AA^1$ whose fiber over $a\neq0$ is $S$ and over $a=0$ is $Y$. Further, this gives a line bundle restricting to $L$ on a nonzero fiber and to $\pi^*(L|_B)$ on the zero fiber. The Rees construction gives a vector bundle on $\AA^1$ whose fibers are the spaces of sections $V$ and $V_{B}^{\inn}$ for $a\neq0$ and $a=0$ respectively. Their jet evaluations form a map of vector bundles on the relative configuration spaces. Since injectivity on fibers is open, injectivity of the jet map at one point $\qq\in \Conf_r(Y)$ implies injectivity of the jet map for a general set of $r$ points on $S$. 

Thus, Theorem~\ref{thm:initial-series-to-vanishing}, applied after extending to $\overline{K}$, reduces the desired vanishing to finding a configuration $\qq\in \Conf_{r}(Y)(K)$ such that 
\[
V_{B}^{\inn}\left(-Z_{\qq}(m^r)\right) \coloneqq \ker\left(
\begin{tikzcd}[cramped, row sep = .2em, column sep = 3.5em]
V_{B}^{\inn}  \arrow[r] & H^0\left(Z_{\qq}(m^r), \left(\pi^*(L|_{B})\right)|_{Z_{\qq}(m^r)}\right)
\end{tikzcd}\right) =0.
\]
Here $Z_{\qq}(m^r)$ is the union of fat points on $Y$ defined by $\cI^{m}_{q_i}$.  

\subsubsection{Third Reduction: From the Normal Bundle to a Torus} We next simplify the curve so that its sections can be described explicitly.  After a ramified base change and a sequence of blow-ups resolving the singularity, the smoothing of $B_0$ has a strongly semistable regular model $\pB$ whose special fiber is a necklace: the normalization of $B_0$ is joined at the two points over the node by a chain of rational curves. We study residual sections on one rational component of the special fiber. 

Let $K=\KK((s))$ with $\val(s)=1/e$ for an integer $e\geq1$. Let $B$ be a smooth projective geometrically connected curve over $K$ with a strongly semistable regular model $\pB$ whose dual graph $\Gamma$ is a cycle. Every edge of $\Gamma$ has length $1/e$; write $\gamma$ for the total length of $\Gamma$ (see \cite{baker08}). Fix a rational component $C\cong \PP^1_{\KK}$ and its corresponding vertex $x_{\bullet}$. The curve $C$ meets the rest of the special fiber in two points; call them $0$ and $\infty$. Choose a coordinate $v$ on $C$ with its zero at $0$ and its pole at $\infty$. Orient $\Gamma$ so that the edge at $0$ leaves $x_{\bullet}$ in the positive direction. 

If $\cL$ is a line bundle on $\pB$ with generic fiber $L$ and degree zero on $C$, fix an identification of the generic fiber with $L$ and a nowhere vanishing section of $\cL|_{C}$. The initial term of a nonzero section $g\in H^0(B,L)$ is the residual section
\[
\inn_{x_{\bullet}}\left(g\right) \coloneqq \left(s^{-\ord_{C}(g)}g\right)\big|_{C},
\]
which is a Laurent polynomial in $v$. Here $\ord_{C}(g)$ is the order of $g$ along $C$. The support of this Laurent polynomial is controlled by the tropical 
moment. If $\sD$ is a divisor on $\Gamma$ then its tropical moment is
\[
\mom_{x_{\bullet}}(\sD)=\frac{-1}{\gamma}\sum_{y\in \Gamma}\sD(y)\len_{x_\bullet}(y)
\]
where $\len_{x_{\bullet}}(y)\in [0,\gamma)$ is the positively oriented distance from $y$ to $x_{\bullet}$. We extend this to line bundles on $\pB$ by setting $\mom_{x_{\bullet}}(\cL)\coloneqq \mom_{x_{\bullet}}(\dd(\cL))$ where $\dd$ is the multidegree map of Baker and Jensen \cite[\S~2]{bakerJensen16}.

\begin{theoremalpha}\label{main-thm:support-int}
	Let $\cL$ be a line bundle on $\pB$ with generic fiber $L$ and  $\deg(\cL|_{C})=0$. 
	\begin{enumerate}
	\item For every nonzero section $g\in H^0(B,L)$,
	\[
	\supp \inn_{x_{\bullet}}(g) \subset \left[\mom_{x_{\bullet}}(\cL), \;\; \mom_{x_{\bullet}}(\cL)+\deg(L)\right]\cap \ZZ.
	\]
	\item For every other extension $\cL'$ with degree zero on $C$,
	\[
	\mom_{x_{\bullet}}(\cL) \equiv \mom_{x_{\bullet}}(\cL') \mod\ZZ.
	\] 
	\item If $\Gamma$ has at least three vertices, then for every $c\in \ZZ$ there exists an extension $\cL'$ of $L$ with $\deg(\cL'|_{C})=0$ such that
	\[
	\mom_{x_{\bullet}}(\cL') = \mom_{x_{\bullet}}(\cL)+c.
	\]
	\end{enumerate}
\end{theoremalpha}

Note that the class $\mom_{x_{\bullet}}(\cL) \mod\ZZ$ can be identified with $\gamma^{-1}$ times the image of $\dd(\cL)-\deg(L)x_{\bullet}$ under the tropical Abel--Jacobi map $\Pic^0(\Gamma) \to \RR/\gamma\ZZ$ \cite{bakerNorine07,mikhalkinZharkov08}.

Part (1) bounds the exponents of all the initial terms of nonzero sections of $L$ at once, placing them in an interval of length $\deg(L)$. Parts (2) and (3) explain how we can move this interval by changing the extension of $L$. We will use this freedom to change the extension of the normal bundle and shear the support diagram so that there are not too many lattice points in a horizontal row. 

Returning to the initial linear series, choose extensions $\cL_{B}$ and $\cN$ of $L|_{B}$ and $N_{B/X_{K}}$ whose restrictions to $C$ have degree zero. Compatible trivializations identify the complement of the zero section in the total space of $\cN|_{C\setminus\{0,\infty\}}$ with a torus $\TT$ with coordinates $u,v$. Here $u$ is the fiber coordinate and $v$ is the coordinate on $C$. We extend the definition of initial terms to $V^{\inn}_{B}$ by normalizing by a common power of $s$ and restricting to the total space of $\cN|_{C}$. On $\TT$, the resulting initial term is a Laurent polynomial in $u,v$. We record the possible exponents of these Laurent polynomials via the support diagram 
\[
\SI(V)\coloneqq \bigcup_{j\geq0}\bigcup_{0\neq g_j \in V_{B,j}} \left\{j\right\} \times \supp \inn_{x_{\bullet}}(g_j) \subset \ZZ^2.
\]
Here the initial term of $g_j$ is computed using the extension $\cL_B\otimes \cN^{-j}$. For a finite set $ \sS\subset \ZZ^2$ let $W_{\sS}$ be the $\KK$-vector space spanned by Laurent polynomials of the form $u^jv^k$ for $(j,k)\in \sS$. Thus, all initial terms of nonzero elements of $V^{\inn}_{B}$ lie in $W_{\SI(V)}$. 

Distinct points $\zeta_1,\ldots,\zeta_r\in \TT$ lift to distinct points $q_1,\ldots,q_r \in Y(K)$. We show that if a nonzero element of $V_B^{\inn}$ vanishes to order at least $m$ at every $q_i$, then its nonzero initial term vanishes to order at least $m$ at every $\zeta_i$; see Lemma~\ref{lem:jet-specialization}. Consequently, if $\SI(V)\subset \sS$ then injectivity of the evaluation map on the torus
\[
\begin{tikzcd}
W_{\sS} \arrow[r] & \displaystyle \bigoplus_{i=1}^r \cO_{\TT,\zeta_i}/\mfm_{\zeta_i}^m
\end{tikzcd}
\]
implies that $V_B^{\inn}(-Z_{\qq}(m^r))=0$ at the lifted points. We are free to choose the configuration of points on the torus we wish to study.  We take the points on the torus to be $(1,\xi_1),\ldots,(1,\xi_r)$ where $\overline{\xi}=(\xi_1,\ldots, \xi_r)\in \Conf_{r}(\KK^{\times})$. Thus, it suffices to show that 
\[
W_{\sS}\cap I_{Z_{\overline{\xi}}(m^r)}=0
\]
where $I_{Z_{\overline{\xi}}(m^r)} \subset \KK[u^{\pm1},v^{\pm1}]$ defines the union of the fat points of order $m$ at the chosen points on the torus. This reduction is the content of Theorem~\ref{thm:tropical-jets-criteria}.

\subsubsection{Fourth Reduction: From Torus Interpolation to Triangle Containment} 

We choose points on the line $\ell=\{u=1\}\subset \TT$ because their limit as they slide along $\ell$ and collide at $(1,1)$ can be described explicitly. The limiting subscheme is defined by $J_{\sE(m^r)}$ where
\[
J_{\sE(m^r)} \coloneqq \left\langle (u-1)^p(v-1)^q \;\; \big| \;\; (p,q)\in \ZZ^2_{\geq0}\setminus \sE(m^r)\right\rangle \quad \quad \text{and} \quad \quad \sE(m^r) \coloneqq \left\{ (p,q)\in \ZZ^2_{\geq0} \;\; \big|\;\; 0\leq q < r(m-p)\right\}.
\]
For example, when $m=1$ the collision of $r$ points at $(1,1)$ gives the length-$r$ subscheme defined by $\langle u-1, (v-1)^r\rangle$. The evaluation maps from $W_{\sS}$ to the coordinate rings  of the members of this flat family form a morphism of vector bundles. Injectivity on fibers is open; thus, if 
\[
\begin{tikzcd}
W_{\sS} \arrow[r] & \displaystyle \KK[u^{\pm1},v^{\pm1}]/J_{\sE(m^r)}
\end{tikzcd}
\]
is injective, then evaluation is injective for general distinct points on $\ell$.

For a finite set $\sS\subset \ZZ^2$ and an integer $p\geq0$ let $N_{\sS}(p)$ denote the number of rows $\sS\cap(\ZZ\times \{k\})$ containing at least $p+1$ points of $\sS$. Corollary~\ref{cor:row-counts} shows that $W_{\sS}\cap I_{Z_{\overline{\xi}}(m^r)}=0$ for $(1,\xi_1),\ldots,(1,\xi_r)$ general distinct points on $\ell$, provided that
\[
N_{\sS}(p) \leq r\left(m-p\right)_{+}\quad\quad \text{for every integer $p\geq0$.}
\]
Here $(m-p)_{+}\coloneqq \max\{m-p,0\}$. Combining this with the previous two degenerations gives the following result for arbitrary multiplicities.

\begin{theoremalpha}\label{main-thm:torus-interpolation}
	With notation and hypotheses as above, fix  multiplicities $\mm=(m_1,m_2,\ldots,m_r)\in \ZZ^{r}_{\geq1}$. Let $\sS \subset \ZZ^2$ be a finite set containing $\SI(V)$ such that for every integer $p\geq0$:
	\[
	N_{\sS}(p) \leq \sum_{i=1}^{r}(m_i-p)_{+}.
	\]
	Then $H^0\left(S_{\overline{K}}, L_{\overline{K}}\otimes \cI_{Z_{\pp}(\mm)}\right)=0$  for a general point $\pp \in  \Conf_{r}(S_{\overline{K}})$.
\end{theoremalpha}

The inequality above compares two staircase regions in $\ZZ^2$. Restricting to the case when $m_1=\cdots =m_r=m$, the sum on the right-hand side describes $\sE(m^r)$ defined above since $\sE(m^r)$ is a staircase where the $p$-th column has exactly $r(m-p)_+$ points. The left-hand side can be visualized as follows. Replace each nonempty row of $\sS$ containing $p$ points with the consecutive lattice points $0,1,\ldots,(p-1)$, and then stack these rows from bottom to top in order of length. The resulting staircase, which we call the lex staircase $\sL(\sS)$, has $N_{\sS}(p)$ points in column $p$. The inequality in Theorem~\ref{main-thm:torus-interpolation} determines when $\sL(\sS) \subset \sE(m^r)$. Interpolation one row at a time then shows that $\sL(\sS)$ indexes the set of partial derivatives at $(1,1)$ that detect every nonzero element of $W_{\sS}$. Thus, the inequalities ensure that all the functionals are supplied by the limit scheme.

\subsubsection{Fitting One Triangle Inside Another}

We apply Theorem~\ref{main-thm:torus-interpolation} with $S=X_K$ and $L=dH$. Set
\[ 
b=\deg\left(N_{B/X_K}\right)= \nu^2h, \quad \kappa = \nu hd/m, \quad \alpha = \mom_{x_{\bullet}}(\cN), \quad \quad \psi = \mom_{x_{\bullet}}(\cL_B),
\]
where $\cN$ and $\cL_B$ extend $N_{B/X_K}$ and $L|_{B}$ respectively and have degree zero on $C$. The support bounds from Theorem~\ref{main-thm:support-int} place $\SI(V)$ inside the lattice points of $(0,\psi)+m\sP_{\kappa,\alpha}$ where
\[
\sP_{\kappa,\alpha} \coloneqq\left\{
(x,y)\in \RR^{2}\;\; \mid \;\;
\begin{gathered}
 x\geq0,  \quad 
 -\alpha x \leq y \leq \kappa-(\alpha+b)x
 \end{gathered}
\right\}.
\]
This is a triangle with area $\kappa^2/(2b) = h(d/m)^2/2$, recovering the leading term of the previous dimension count. Changing the tropical moment $\alpha$ changes its shape by a shear, without changing its area. 

\begin{figure}
\begingroup
\providecolor{DartmouthGreen}{HTML}{00693E}
\providecolor{DartmouthDarkGreen}{HTML}{004B23}
\providecolor{FigRed}{HTML}{A4343A}
\providecolor{FigGold}{HTML}{C28E0E}
\providecolor{FigGray}{HTML}{7A7A7A}
\begin{tikzpicture}[
 x=1cm,y=1cm,line cap=round,line join=round,
 >={Stealth[length=1.7mm,width=1.1mm]},
 every node/.style={inner sep=1.5pt},
 title/.style={font=\footnotesize},small/.style={font=\scriptsize},
 grid/.style={draw=FigGray!24,line width=.25pt},
 edge/.style={draw=DartmouthDarkGreen,line width=.8pt},
 dot/.style={circle,fill=black,draw=white,line width=.3pt,inner sep=0pt,minimum size=3.5pt},
 axis/.style={draw=FigGray!80,line width=.45pt,->},
 ambient/.style={draw=FigGray!35!black,line width=1.05pt,
 preaction={draw=white,line width=1.9pt}}]
\foreach \panelshift in {.7,7.6} {
 \begin{scope}[shift={(\panelshift cm,1.35cm)},x=.7cm,y=.17cm]
  \foreach \k in {-7,...,19} \draw[grid] (-.13,\k)--(2.52,\k);
  \foreach \j in {0,1,2} {
   \draw[grid] (\j,-7.5)--(\j,19.35);
   \node[small,anchor=north] at (\j,-7.9) {$\j$};
  }
  \foreach \k in {-6,0,6,12,18}
   \node[small,anchor=east] at (-.2,\k) {$\k$};
  \node[small] at (-.33,20) {$k$};
  \node[small] at (2.7,-8.45) {$j$};
 \end{scope}
}
% No shear: the three red rows contain three points each.
\begin{scope}[shift={(.7cm,1.35cm)},x=.7cm,y=.17cm]
 \fill[DartmouthGreen!8] (0,-7/3)--(0,56/3)--(7/3,-7/3)--cycle;
 \foreach \k in {-2,-1,0}
  \fill[FigRed!22] (-.13,\k-.28) rectangle (2.52,\k+.28);
 \draw[edge] (0,-7/3)--(0,56/3)--(7/3,-7/3)--cycle;
 \foreach \j/\lo/\hi in {0/-2/18,1/-2/9,2/-2/0}
  \foreach \k in {\lo,...,\hi} \node[dot] at (\j,\k) {};
\end{scope}
% Shear (j,k)->(j,k-2j): eleven gold rows have two points each.
\begin{scope}[shift={(7.6cm,1.35cm)},x=.7cm,y=.17cm]
 \fill[DartmouthGreen!8] (0,-7/3)--(0,56/3)--(7/3,-7)--cycle;
 \foreach \k in {-4,-2,-1,0,1,2,3,4,5,6,7}
  \fill[FigGold!26] (-.13,\k-.28) rectangle (2.52,\k+.28);
 \draw[edge] (0,-7/3)--(0,56/3)--(7/3,-7)--cycle;
 \foreach \j/\lo/\hi in {0/-2/18,1/-4/7,2/-6/-4}
  \foreach \k in {\lo,...,\hi} \node[dot] at (\j,\k) {};
\end{scope}
\foreach \x in {2.98,9.88} {
}
% Each cell's lower-left corner is the exponent (p,q).
\foreach \panelshift in {3.65,10.55} {
 \begin{scope}[shift={(\panelshift cm,.1cm)},x=.46cm,y=.17cm]
  \foreach \p/\h in {0/26,1/13}
   \foreach \q in {1,...,\h}
    \filldraw[fill=white,draw=FigGray!50,line width=.3pt]
     (\p,\q-1) rectangle (\p+1,\q);
 \end{scope}
}
% Failure: 21+12+3=36 cells, with the third column outside the target.
\begin{scope}[shift={(3.65cm,.1cm)},x=.46cm,y=.17cm]
 \foreach \p/\h in {0/21,1/12}
  \foreach \q in {1,...,\h}
   \filldraw[fill=DartmouthGreen!26,draw=white,line width=.4pt]
    (\p,\q-1) rectangle (\p+1,\q);
 \foreach \q in {1,2,3}
  \filldraw[fill=FigRed!43,draw=white,line width=.4pt]
   (2,\q-1) rectangle (3,\q);
 \draw[edge,line width=.65pt]
  (0,0)--(0,21)--(1,21)--(1,12)--(2,12)--(2,0)--cycle;
 \draw[FigRed,line width=.8pt] (2,0)--(3,0)--(3,3)--(2,3)--cycle;
 \node[small,DartmouthDarkGreen,anchor=west] at (1.13,21) {$21$};
 \node[small,DartmouthDarkGreen,anchor=west] at (2.13,11.8) {$12$};
 \node[small,FigRed,anchor=west] at (3.13,3) {$3$};
\end{scope}
% Success: 25+11=36 cells, three of the target's 39 cells unused.
\begin{scope}[shift={(10.55cm,.1cm)},x=.46cm,y=.17cm]
 \foreach \p/\h in {0/25,1/11}
  \foreach \q in {1,...,\h}
   \filldraw[fill=DartmouthGreen!26,draw=white,line width=.4pt]
    (\p,\q-1) rectangle (\p+1,\q);
 \draw[edge,line width=.65pt]
  (0,0)--(0,25)--(1,25)--(1,11)--(2,11)--(2,0)--cycle;
 \node[small,DartmouthDarkGreen,anchor=west] at (1.13,25) {$25$};
 \node[small,DartmouthDarkGreen,anchor=west] at (2.13,11) {$11$};
\end{scope}
% Ambient boundary drawn last for unambiguous containment.
\foreach \panelshift in {3.65,10.55} {
 \begin{scope}[shift={(\panelshift cm,.1cm)},x=.46cm,y=.17cm]
  \draw[axis] (-.16,0)--(4.15,0) node[right,small] {$p$};
  \draw[axis] (0,-.35)--(0,27.7) node[above,small] {$q$};
  \foreach \p in {0,1,2}
   \node[small,anchor=north] at (\p+.5,-.35) {$\p$};
  \foreach \q in {13,26}
   \node[small,anchor=east] at (-.12,\q) {$\q$};
  \draw[ambient] (0,0)--(0,26)--(1,26)--(1,13)--(2,13)--(2,0)--cycle;
  %\draw[FigGray!80!black,line width=.45pt] (1,23)--(1.65,23.7)--(2,23.7);
  \node[small,anchor=west] at (2.08,23.7) {$\mathsf E(2^{13})$};
 \end{scope}
}
\draw[FigGray!30,line width=.4pt] (-.1,-1.55)--(13.2,-1.55);
\end{tikzpicture}
\endgroup
 \caption{The support triangles for $(\PP^2, H)$ with $\nu=3$, $r=13$, $(d,m)=(7,2)$, $\kappa=21/2$, and $\psi=-7/3$. Left: for $\alpha=0$ the lex staircase has heights $21,12,3$ and is not contained in $\sE(2^{13})$. Right: increasing the shear to $\alpha=2$, the lex staircase has heights $25$ and $11$ and the containment holds.}
\label{fig:intro-example}
\end{figure}
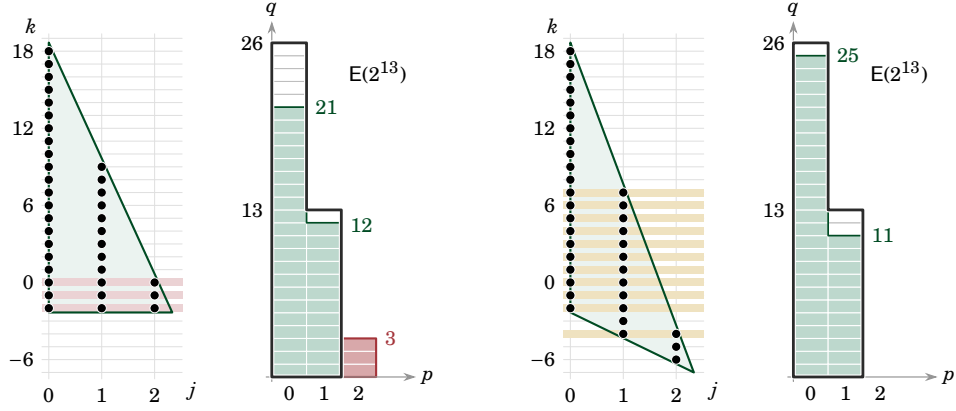
For $\alpha>0$, stacking the rows of $\sP_{\kappa,\alpha}$ by length gives a right triangle with legs of length $\kappa/(\alpha+b)$ and $\kappa(\alpha+b)/b$. Meanwhile the scaled staircase $\frac{1}{m}\sE(m^r)$ approaches a right triangle with legs of length $1$ and $r$. The first triangle fits into the second  when
\[
\kappa-b<\alpha < \frac{br}{\kappa}-b.
\]
Since $\kappa>b$, this interval is nonempty when $\kappa^2/b=h(d/m)^2<r$, which is precisely the upper bound in the crucial interval constraining $d/m$. Proposition~\ref{prop:find-the-moment} shows that, by choosing the ramification index $e$ and the component $C$ appropriately, we can realize any non-integer rational number $\alpha$  in this interval as $\mom_{x_{\bullet}}(\cN)$ for some extension $\cN$. The strict inequality provides the wiggle room needed to absorb error in our lattice point counts after a large enough scaling. This is enough to prove the needed uniform vanishing since the $k$-th power of a nonzero section of $H^0(X_{\pp},dH-mE)$ would be a nonzero section of $H^0(X_{\pp},kdH-kmE)$.

%%%%%%%%%%%%%%%%%%%%%%%%%%%%%%%%%%%%%%%%%%%%%%%%%%%%%%%%%%%%%%%%
\subsubsection{Summary of the Proof of Theorem~\ref{main-thm:technical-nagata}}
%%%%%%%%%%%%%%%%%%%%%%%%%%%%%%%%%%%%%%%%%%%%%%%%%%%%%%%%%%%%%%%%

We  can summarize the proof of Theorem~\ref{main-thm:technical-nagata} as a chain of vanishing statements on successively simpler spaces. Fix positive integers $d,m$ such that $\nu < d/m < \sqrt{r/h}$. The vanishing conditions are:
\begin{itemize}
    \item \textbf{(V0):} $H^0(X_{\pp}, \cO_{X_{\pp}}(d\mu_{\pp}^*(H)-mE))=0$ for a general point $\pp \in \Conf_r(X)$.
    \item \textbf{(V1):} $V_{B}^{\inn}(-Z_{\qq}(m^r)) = 0$ for some $\qq\in \Conf_r(Y)(K)$, where $Y$ is the total space of the normal bundle $N_{B/X_{K}}$ and $V=H^0(X_{K},\cO_{X_K}(dH))$.
     \item \textbf{(V2):} $W_{\sS} \cap I_{Z_{\overline{\xi}}(m^r)}=0$ for $r$ distinct points $((1,\xi_1),\ldots,(1,\xi_r))$ on the line $\ell=\{u=1\}\subset \TT$.
    \item \textbf{(V3):} $N_{\sS}(p) \leq r(m-p)_{+}$ for every integer $p\geq0$.
\end{itemize}
Here $V_{B}^{\inn}$ is the initial linear series, $\sS\subset \ZZ^2$ is a finite set containing $\SI(V)$, and $W_{\sS} \subset \KK[u^{\pm1},v^{\pm1}]$ is the $\KK$-span of the corresponding Laurent monomials. The number $N_{\sS}(p)$ counts rows of $\sS$ containing at least $p+1$ lattice points.  Corollary~\ref{cor:row-counts}, Theorem~\ref{thm:tropical-jets-criteria}, and Theorem~\ref{thm:initial-series-to-vanishing}, together with Lemma~\ref{lem:field-descent}, give the following implications:
\begin{align*}
\text{\textbf{(V3)}} \;\;\implies \;\;\text{\textbf{(V2)}}\;\;\implies \;\; \text{\textbf{(V1)}} \;\;\implies \;\; \text{\textbf{(V0)}}.
\end{align*}

Section~\ref{sec:specialization-normal} develops the normal cone degeneration and initial linear series. Section~\ref{sec:inital-terms} defines initial terms, bounds their support, and proves compatibility with specialization of jets. Section~\ref{sec:torus} proves the torus interpolation criteria. Section~\ref{sec:tropical-criteria} combines these ingredients into the general vanishing criterion. Finally, Section~\ref{sec:AVC-Bring-It-HOME} chooses the shear, carries out the triangle comparison needed for \textbf{(V3)}, and deduces the main theorems.

%%%%%%%%%%%%%%%%%%%%%%%%%%%%%%%%%%%%%%%%%%%%%%%%%%%%%%%%%%%%%%%%
\subsection{Relation to Prior Work}
%%%%%%%%%%%%%%%%%%%%%%%%%%%%%%%%%%%%%%%%%%%%%%%%%%%%%%%%%%%%%%%%

The arguments bring together ideas from tropical geometry, classical deformation arguments, and interpolation. We describe the role of each ingredient and its relation to existing work. 

\subsubsection{Tropical Geometry} Specializing line bundles and their sections to a semistable model to extract combinatorics lies at the heart of tropical geometry \cite{baker08,bakerNorine07}. We use the compatibility between multidegrees, principal divisors, and specialization in Lemma~\ref{lem:bakerJensen-Special}; see \cite[\S~2]{bakerJensen16}. Changing an extension by a vertical twist changes its multidegree by a principal divisor. On a cycle, the tropical moment is a real lift of the normalized Abel--Jacobi class \cite{mikhalkinZharkov08}, and the integer ambiguity is the freedom used in Proposition~\ref{prop:find-the-moment}. 

We also retain the residual section on one rational component to compute the initial terms of sections. This is related to the additional componentwise data in metrized complexes \cite{aminiBaker15}. The relation between the order of these residual sections at the nodes and the slopes on the graph is a local form of the slope formula \cite[Thm.~5.15]{bakerPayneRabinoff13}; see Lemma~\ref{lem:support-bounds}. Kalinin also uses tropical geometry to obtain Nagata-type bounds \cite{kalinin17}. His approach relates prescribed multiplicities to regions in the Newton polygon of the curve. 

Allowing semistable models with more general dual graphs might strengthen or extend Theorem~\ref{main-thm:technical-nagata}. For a rational component meeting the rest of the special fiber in two nodes, one idea might be to replace the normalized arc length that appears in the definition of the tropical moment by a harmonic function of some form in the spirit of \cite{bakerFaber06}. In order to apply this to the relevant interpolation problems, one would need to be able to control this ``harmonic tropical moment'' in order to fine-tune the shear parameter. 

\subsubsection{Deformation to the Normal Cone} Deformation to the normal cone is a standard construction \cite[\S~5.1]{fulton98}. Here we keep track of a filtered vector space of sections using the Rees construction to build a flat family. Ciliberto and Miranda degenerate the plane and study linear systems on the resulting components, reducing interpolation questions to lower-dimensional calculations \cite{cilibertoMiranda98,cilibertoMiranda00,cilibertoMiranda06,cilibertoMiranda11}; see also \cite{eckl11}. Petrakiev specializes points to an elliptic curve and then uses deformation to the normal cone to study interpolation on elliptic ruled surfaces \cite{petrakiev09,petrakiev14}.

Our first degeneration similarly replaces a surface problem by data along a curve, but retains this data as an initial linear series on the total space of the normal bundle. The filtration in Section~\ref{sec:specialization-normal} is built from successive restrictions to the curve $B$ and division by its local defining equations, which feels at the very least spiritually akin to the residual--trace arguments in the m\'{e}thode d'Horace and its refinements \cite{hirschowitz85,alexnderHirschowitz95,alexanderHirschowitz00}.

\subsubsection{Interpolation}
The analysis of fat points in Section~\ref{sec:torus} combines collisions with polynomial interpolation. The collision of points along a line in Lemma~\ref{lem:coliding-colinear} is related to the study of limits of linear systems under specialization in \cite{evain97,evain98,evain99,evain07}. The conversion of vanishing at fat points to interpolation on the exponents has appeared in several places  \cite[Prop.~13]{dumnicki07}, \cite[\S~8]{castravet18}, and \cite{gonzalezKaru16, perkinson00}. For a finite exponent set $\sS\subset \ZZ^2$, the lex staircase $\sL(\sS)$ gives a monomial interpolation basis. This connects with the substantial literature on the vanishing ideals and lex initial ideals of finite sets of points \cite{mollerBuchberger82,cerliencoMureddu95,felszeghyRathRonya06,gundlach21}. Lemma~\ref{lem:key-interpolation} supplies the two-variable row argument directly, and is the only interpolation fact we need.

\subsubsection{Other Work} Three further connections feel worth mentioning, although our arguments do not rely on them.  Eisenbud and Harris studied  limit linear series on curves of compact type by concentrating their degrees on individual components and comparing vanishing sequences \cite{eisenbudHarris86}. Our degenerations are not of compact type. Extensions of the theory of  limit linear series include \cite{aminiBaker15,osserman19}, while \cite{aminiGierczak26} develops a combinatorial theory in terms of slopes on metric graphs. In this note we retain only the residual sections on one fixed component, and their compatibility with the two adjacent slopes. 

Jensen and Payne's theory of tropical independence \cite{jensenPayne14,jensenPayne16,jensenPayne17} gives another perspective on using tropical methods to test linear independence. Their theory provides powerful tools to test whether a collection of sections is linearly independent by looking at their tropicalizations. However, it does not seem to apply directly to the case at hand, as we need the images of sections in the space of jets at a point to be linearly independent. This depends not only on the tropicalization of the sections, but also on their initial terms. It would be interesting to know whether a jet version of tropical independence could streamline Sections~\ref{sec:torus} and~\ref{sec:tropical-criteria}.

Finally, the polygons that appear in Section~\ref{subsec:polygon} seem closely connected to Newton--Okounkov bodies \cite{lazarsfeldMustata09,okounkov13}, which is perhaps unsurprising given their relationship to Seshadri constants \cite{DKMS16-2,kuronyaLozovanu18}. The triangles $\sP_{\kappa,\alpha}$ appearing in the proof of Theorem~\ref{main-thm:technical-nagata} are shears of the Newton--Okounkov body for $(d/m)H$ with respect to the flag $\{p\}\subset B_{K}\subset X_{K}$ for any point $p\in B$; see Remark~\ref{rem:NO}. However, I do not see how to leverage this into a more direct proof. One possible route would be to use the theory of Newton--Okounkov bodies on graphs developed by  Katz and Urbinati \cite{katzUrbinati19}.
 
%%%%%%%%%%%%%%%%%%%%%%%%%%%%%%%%%%%%%%%%%%%%%%%%%%%%%%%%%%%%%%%%
\subsection{Relation to OpenAI's Releases}
%%%%%%%%%%%%%%%%%%%%%%%%%%%%%%%%%%%%%%%%%%%%%%%%%%%%%%%%%%%%%%%%

The main ideas of this note come from \cite{openAI-P2, openAI-NB}. In particular, the normal bundle and interpolation constructions used here follow the approach developed there, and our aim is only to offer a tropical perspective on one step of that approach. The first release announces a proof of Nagata's conjecture for $\mathbb{P}^2$, while the second announces a proof of the maximality of multipoint Seshadri constants on arbitrary polarized surfaces for sufficiently many very general points. A precise bound on the number of points needed is not given, though it seems likely that one could be extracted from its arguments. OpenAI also released work concerning higher-dimensional \cite{openAI-higherD} and positive-characteristic \cite{openAI-posChar} analogs of these questions. Neither is used in the approach here. 

%%%%%%%%%%%%%%%%%%%%%%%%%%%%%%%%%%%%%%%%%%%%%%%%%%%%%%%%%%%%%%%%
\subsection*{Conventions} Throughout $\KK$ denotes an uncountable algebraically closed field of characteristic zero. Unless otherwise stated, a curve on a smooth surface means a nonzero effective divisor, possibly reducible or nonreduced. For $a \in \RR$ set  $a_{+}\coloneqq \max\{a,0\}$. If $S$ is a projective surface we write $\langle D \cdot D'\rangle$ for the intersection product of  divisors $D$ and $D'$ on $S$. We write $\langle D^2\rangle$ for $\langle D\cdot D\rangle$. A statement holds for a general point if it holds on a nonempty Zariski open subset, and holds for a very general point if it holds outside a countable union of proper Zariski closed subsets. 

%%%%%%%%%%%%%%%%%%%%%%%%%%%%%%%%%%%%%%%%%%%%%%%%%%%%%%%%%%%%%%%%
\subsection*{Acknowledgments} I am deeply grateful to Renzo Cavalieri, Daoji Huang, and Diane Maclagan for the many wonderful conversations on related topics, which greatly shaped my thinking on these ideas and placed me in a position to quickly write this article. I also benefited from attending the MFO workshop \emph{Asymptotic Invariants of Linear Systems and Homogeneous Ideals with Symmetries} in early September, and I thank the organizers Krishna Hanumanthu, Alexandra Seceleanu, and Tomasz Szemberg. I would like to thank Asher Auel for encouraging me to write these ideas down, and  Dave Jensen, Amy Q Li, and Jenia Tevelev for helpful feedback on an earlier draft. 

%%%%%%%%%%%%%%%%%%%%%%%%%%%%%%%%%%%%%%%%%%%%%%%%%%%%%%%%%%%%%%%%
\subsection*{Apologia} In writing this article, I erred on the side of making the exposition self-contained and giving detailed proofs. As a result, it is likely much longer than experts need, and many of the results could be deduced from existing results elsewhere. I have tried  to acknowledge this prior work, but I am also sure there are things I overlooked or am unaware of. I hope the reader will look past both of these faults, and offer grace as we all figure out how to navigate these changing times.  

%%%%%%%%%%%%%%%%%%%%%%%%%%%%%%%%%%%%%%%%%%%%%%%%%%%%%%%%%%%%%%%%
\subsection*{AI Acknowledgments} Given that this project originated from OpenAI's large release on October 6, 2026, it feels like this article should be regarded as fruit of the poisonous tree with regard to AI usage. In attempting to understand the manuscripts released by OpenAI I had numerous conversations with OpenAI's ChatGPT (6 Astra / 6.1 Sol) and Anthropic's Claude (Opus 5.5). These models also produced the figures in the paper, reviewed earlier drafts, and proofread the final version of this article. I am responsible for the mathematics and errors.

%%%%%%%%%%%%%%%%%%%%%%%%%%%%%%%%%%%%%%%%%%%%%%%%%%%%%%%%%%%%%%%%
%%%%%%%%%%%%%%%%%%%%%%%%%%%%%%%%%%%%%%%%%%%%%%%%%%%%%%%%%%%%%%%%
\section{Specialization to the Normal Bundle}\label{sec:specialization-normal}
%%%%%%%%%%%%%%%%%%%%%%%%%%%%%%%%%%%%%%%%%%%%%%%%%%%%%%%%%%%%%%%%
%%%%%%%%%%%%%%%%%%%%%%%%%%%%%%%%%%%%%%%%%%%%%%%%%%%%%%%%%%%%%%%%

Let $\KK\subset \LL$ be algebraically closed fields. We allow this extension as later we will apply the arguments both over our original field $\KK$ and  an algebraic closure of $\KK((s))$.  Let $S$ be a smooth projective integral surface defined over $\LL$, and let $B\subset S$ be a smooth integral curve. Write $\iota:B\hookrightarrow S$ for the inclusion and $N_{B/S}=\cO_{S}(B)|_{B}$ for the normal bundle. Let
\[
\begin{tikzcd}[row sep = .2em, column sep =3em]
    Y\coloneqq \bSpec_{B}\left(\Sym N^{\vee}_{B/S}\right) \arrow[r, "\pi"] & B    
\end{tikzcd}
\]
be the total space of the normal bundle of $B$. Fix a line bundle $L$ on $S$, and write $L_{B}=\iota^*L$ for the restriction of $L$ to $B$. Set $V=H^0(S,L)$. Our goal is to compare sections of $L$ vanishing to prescribed orders at points on $S$ with sections of $\pi^*L_{B}$ vanishing at points on $Y$. 

The sections on $Y$ that we use form a finite-dimensional graded vector space obtained from the first nonzero terms of sections of $L$ in the direction normal to $B$. Theorem~\ref{thm:initial-series-to-vanishing} shows that, to prove that no nonzero section of $L$ vanishes to prescribed orders at general points of $S$, it is enough to exhibit a single configuration of points on $Y$ at which no nonzero element of the initial linear series satisfies the same vanishing conditions. The proof puts $(S,L)$ and $(Y, \pi^*L_B)$ into one family by deforming to the normal cone, and then uses that injectivity of jet evaluation is an open condition. 

\subsection{The Initial Linear Series} We will define a decreasing filtration $F^{\bullet}V$ on $V$ by order of vanishing along $B$. For an integer $j\geq0$ define
\[
F^{j}V \coloneqq H^{0}\left(S, L(-jB)\right) \cong H^0\left(S,L\otimes \cI_{B}^{j}\right)
\]
where $\cI_{B} \subset \cO_{S}$ is the ideal sheaf of $B$. Since $B$ is Cartier, the identification $N_{B/S}^{\vee}\cong \cO_{S}(-B)|_{B}$ gives an exact sequence
\[
\begin{tikzcd}[row sep = .2em, column sep = 3em]
0 \arrow[r] & L\left(-(j+1)B\right) \arrow[r] & L\left(-jB\right) \arrow[r] &\iota_* \left(L_{B} \otimes N_{B/S}^{-j}\right) \arrow[r] & 0
\end{tikzcd}.
\]
Taking global sections we see that we have maps
\[
\begin{tikzcd}[row sep = .2em, column sep = 3em]
    F^{j}V \arrow[r, "\rho_j"] & H^0\left(B, L_B\otimes N^{-j}_{B/S}\right)
\end{tikzcd}
\]
such that $\ker(\rho_{j}) = F^{j+1}V$. Note the restriction map $\rho_j$ may not be surjective.

Let $\tau \in H^0(Y, \pi^{*}N_{B/S})$ be the tautological section. For $g_{j}\in H^0(B, L_{B}\otimes N_{B/S}^{-j})$, the contraction 
\[
\begin{tikzcd}[row sep = .2em, column sep = 4em]
\left(\pi^*L_{B} \otimes \pi^*N_{B/S}^{-j}\right) \otimes \left( \pi^*N_{B/S}\right)^j \arrow[r, "\ev_j", "\sim" {below}] & \pi^*L_B
\end{tikzcd}
\]
gives a section $\ev_{j}(g_j)=\pi^*(g_j)\tau^j$ of $\pi^*L_B$. 

\begin{lemma}\label{lem:proj-iso}
    The projection $\pi$ induces a canonical isomorphism 
    \[
    \begin{tikzcd}[row sep = .1em, column sep = 3em]
    \displaystyle \bigoplus_{j\geq0} H^0\left(B, L_{B}\otimes N^{-j}_{B/S}\right) \arrow[r, "\ev"] & H^{0}\left(Y, \pi^*L_B\right)\\
    \sum_{j} g_j\arrow[r, mapsto] & \sum_{j} \pi^*(g_j)\tau^j
    \end{tikzcd}.
    \]
\end{lemma}

\begin{proof}
	By \cite[Ex.~II.5.1, II.5.17]{hartshorne77} the projection $\pi$ induces an isomorphism 
	\[
	\pi_*\pi^*L_{B} \cong L_{B} \otimes \Sym(N_{B/S}^{\vee}) \cong \bigoplus_{j\geq0} L_{B}\otimes N^{-j}_{B/S}.
	\]
	Since $B$ is Noetherian global sections commute with direct sums \cite[Prop.~III.2.9]{hartshorne77}, and the result follows. 
\end{proof}

The isomorphism has a simple local description. On an open set of $B$ where $N_{B/S}$ and $L_{B}$ are trivialized by $s_N$ and $s_L$, write $\tau=u\pi^*s_{N}$ and $g_j = a_j(s_L\otimes s_{N}^{-j})$. Then $u$ is the fiber coordinate over this open set and 
\[
\ev_j(g_j) = \pi^*(a_js_L)u^j.
\]
The associated graded of the filtration $F^{\bullet}$ is $\gr_{F}V \coloneqq \oplus_{j\geq0}F^{j}V/F^{j+1}V$. 

\begin{lemma}\label{lem:asoc-graded}
    The restriction maps $\rho_{j}$ induce a graded injection of vector spaces 
    \[
\begin{tikzcd}[row sep = .2em, column sep = 3em]
    \gr_{F}V \arrow[r, "\rho"]& H^0\left(Y, \pi^*L_B\right).
\end{tikzcd}
\]
    Further, $F^{j}V=0$ for $j\gg0$, and so $\dim_{\LL} \gr_{F}V = \dim_{\LL}V$. 
\end{lemma}

\begin{proof}
	Injectivity follows from the fact, noted above, that $\ker(\rho_{j}) = F^{j+1}V$ together with Lemma~\ref{lem:proj-iso}. Since $S$ is integral a section of $L$ that vanishes to every order along $B$ must be zero. Thus, $\cap_{j} F^{j}V=0$, which since $V$ is finite dimensional and $F^{\bullet}$ is decreasing, implies that $F^{j}V=0$ for $j\gg0$.
\end{proof}

\begin{definition}
    The \emph{initial linear series of $V$ along $B$} is the image of $\rho$, which we denote by
    \[
    V^{\inn}_{B}\coloneqq \rho\left(\gr_{F}V\right) \subset H^0\left(Y, \pi^*L_{B}\right). 
    \]
    For an integer $j\geq0$ the \emph{degree-$j$ coefficient space} of $V^{\inn}_{B}$ is
    \[
    V_{B,j}\coloneqq \rho_j\left(F^{j}V\right) \subset H^0\left(B,L_B\otimes N_{B/S}^{-j}\right).
    \] 
\end{definition}

As a warning, $V_{B,j}$ is a space of sections on $B$, not $Y$. It is its image $\ev_{j}(V_{B,j})$, consisting of sections of $\pi^*L_B$ that are homogeneous of degree $j$ in the fiber direction, that is the degree $j$ piece of $V_{B}^{\inn}$. In particular, we have:
\[
V^{\inn}_{B} = \bigoplus_{j\geq0} \ev_{j}\left(V_{B,j}\right) =\left\{ \sum_{j}\pi^*(g_j)\tau^j \;\; \bigg| \;\; g_{j} \in V_{B,j}\right\}. 
\]

\begin{example}\label{exm:ini-lin-series-1}
	Let $B\subset \PP^2$ be the plane conic given by $Q=xz-y^2 = 0$ and let $L=\cO_{\PP^2}(3)$. Restriction to $B\cong \PP^1$ maps the ten-dimensional vector space $V=H^0(\PP^2,\cO_{\PP^2}(3))$ onto $H^0(\PP^1, \cO_{\PP^1}(6))$ with kernel $F^1V=Q\cdot H^0(\PP^2, \cO_{\PP^2}(1))$. Moreover, $F^{j}V=0$ for all $j\geq2$ since a nonzero cubic cannot be divisible by $Q^2$. Thus, $\gr_{F}V = (V/F^{1}V)\oplus F^1V$ and $V^{\inn}_{B}$ is the image of this under the evaluation map. The coefficient spaces are then
	\begin{align*}
	V_{B,0} &= \rho_{0}\left(V\right) \subset H^0\left(B,L_B\right)\cong H^0\left(\PP^1, \cO_{\PP^1}(6)\right) \\
	V_{B,1} &= \rho_{1}\left(F^1V\right) \subset H^0\left(B,L_B\otimes N_{B/\PP^2}^{-1}\right)\cong H^0\left(\PP^1, \cO_{\PP^1}(2)\right).
	\end{align*}
	The above discussion shows that $V_{B,0}$ is all of $H^0\left(B,L_B\right)$. To describe $V_{B,1}$, work on the chart $x\neq0$ letting $s=y/x$ and $t=z/x-s^2$. In this chart $B=\{t=0\}$, and the class of $t$ trivializes the conormal bundle on this chart. Let $u$ be the corresponding fiber coordinate on this chart of $Y$. Using $x^3$ to trivialize $L$, a cubic $f$ that vanishes along $B$ can be written as
	\[
	\frac{f}{x^3} =\frac{(ax+by+cz)Q}{x^3}=(a+bs+cs^2+ct)t.
	\] 
	Dividing by $t$ and restricting to $B$ gives $\rho_1(f)=a+bs+cs^2$. Since $a,b,c\in \LL$ are arbitrary every quadratic polynomial occurs. Hence $V_{B,1}=H^0\left(\PP^1, \cO_{\PP^1}(2)\right)$. Thus, suppressing the chosen frames and pullbacks, we have that 
	\[
	V_{B}^{\inn} =\Span_{\LL}\langle 1,s,s^2,s^3,s^4,s^5,s^6\rangle \oplus\Span_{\LL}\langle 1,s,s^2\rangle u.
	\]
	Thus, every element of $V^{\inn}_{B}$ can be written uniquely as $g_0+ug_1$ where $g_0$ and $g_1$ are polynomials in $s$ with $\deg(g_0)\leq 6$ and $\deg(g_1)\leq 2$. The first summand above records the restriction of our section to $B$; the second records the first normal coefficient of cubics vanishing along $B$. 
	
	Cancellation can raise the order of vanishing along $B$. For example, the monomials $x^2z$ and $xy^2$ both have initial term $s^2$, whereas $x^2z-xy^2=xQ$ has initial term $u$. Thus, taking only the initial terms of the monomials does not recover the entire initial linear series. 
\end{example}

We are interested in knowing when elements of $V_{B}^{\inn}$ vanish at a collection of points on $Y$. Towards this we introduce the following notation: if $A$ is a line bundle on $Y$ and $W \subset H^0(Y,A)$ then for any finite subscheme $Z \subset Y$ we set
\[
W\left(-Z\right) \coloneqq \ker\left(
\begin{tikzcd}[cramped, row sep = .2em, column sep = 3.5em]
W \arrow[r] & H^0\left(Z, A|_{Z}\right)
\end{tikzcd}\right).
\]
We will be most interested in considering this when $W=V^{\inn}_{B}$. If $\qq\coloneqq(q_1,\ldots,q_n) \in \Conf_n(Y)$ and $\mm=(m_1,\ldots,m_n)\in \ZZ^n_{\geq0}$ let
\[
Z_{\qq}(\mm) = \bigcup_{i=1}^{n}\VV\left(\cI_{q_i}^{m_i}\right),
\]
where $\cI_{q_i}\subset \cO_{Y}$ is the ideal sheaf of $q_i$. Thus, $Z_{\qq}(\mm)$ is the finite subscheme of $Y$ where $q_i$ has multiplicity $m_i$. We use the convention that $\cI_{q_i}^0=\cO_Y$, so that a point with multiplicity zero contributes the empty subscheme and imposes no conditions. For $m_i>0$ the \emph{jets of order less than $m_i$} of a section at a point $q_i$ are given by its restriction to the fat point $\VV(\cI_{q_i}^{m_i})$. After choosing local coordinates at $q_i$ and trivializing the line bundle, this restriction records the terms of total degree less than $m_i$ in its Taylor expansion at $q_i$. In particular, the jet is zero if and only if the section vanishes to order at least $m_i$ at $q_i$.  We call the restriction map 
\[
\begin{tikzcd}[row sep = .2em, column sep = 3.5em]
V_B^{\inn} \arrow[r] & H^0\left(Z_{\qq}(\mm), (\pi^*L_B)|_{Z_{\qq}(\mm)}\right)
\end{tikzcd}
\]
the \emph{jet evaluation map}. It records these jets at all the points $q_i$ simultaneously. Its kernel consists precisely of those sections vanishing to order at least $m_i$ at the point $q_i$ for every $i$. Thus, this map is injective if and only if no nonzero section of $V_B^{\inn}$ satisfies these vanishing conditions. The following theorem establishes the implication $\textbf{(V1)} \implies  \textbf{(V0)}$ outlined in the introduction.

\begin{theorem}\label{thm:initial-series-to-vanishing}
    Fix an integer $n\geq1$ and  $\mm=(m_1,\ldots,m_n)\in \ZZ^n_{\geq0}$. If there exist distinct points $q_1,\ldots,q_n\in Y$ such that  
    \[
    V_{B}^{\inn}\left(-Z_{\qq}(\mm)\right) = 0
    \]
    then $H^0\left(S, L\otimes \cI_{Z_{\pp}(\mm)}\right)=0$ for a general point $\pp\in \Conf_{n}(S)$. 
\end{theorem}

We postpone the proof of this theorem until the end of this section. 

\begin{example}
	Returning to Example~\ref{exm:ini-lin-series-1} let $B\subset \PP^2$ be the plane conic given by $Q=xz-y^2 = 0$ and let $L=\cO_{\PP^2}(3)$. Working in the same chart on $Y$ every point can be expressed as a pair $(s_i,u_i)\in \LL^2$ with $s_i$ the coordinate of its image in $B$. Choose four points $\qq=(q_1,q_2,q_3,q_4)$ in this chart of $Y$ lying over four distinct points on $B$. That is $q_i=(s_i,u_i)$ with $s_i$ distinct. We claim that $V_{B}^{\inn}(-Z_{\qq}(2,2,2,2))=0$. Our description of the coefficient spaces in Example~\ref{exm:ini-lin-series-1} means we may write an element $F\in V_{B}^{\inn}$ as  $F=g_0+ug_1$ where $g_0$ and $g_1$ are polynomials in $s$ with $\deg(g_0)\leq 6$ and $\deg(g_1)\leq 2$. If $F$ vanishes to order at least $2$ at $\qq$ then 
	\[
	F(q_i)=0, \quad \quad \partial_{u}F(q_i)=g_1(s_i) = 0, \quad \quad \text{and} \quad \quad \partial_{s}F(q_i)=0
	\]
	for all $i$. Since the $s_i$ are distinct and $g_1$ has degree at most 2 the middle equality implies that $g_1=0$. Thus, $F=g_0$, and the first and third equalities give $g_0(s_i)=g'_0(s_i)=0$ for all $i$. A nonzero polynomial of degree at most $6$ cannot have four distinct double roots, and so  $F=g_0=0$. 
	
	The specialization in Theorem~\ref{thm:initial-series-to-vanishing} implies that $H^0(\PP^2,\cO_{\PP^2}(3)\otimes\cI_{Z_{\pp}(2^4)})=0$ for a general point $\pp \in \Conf_{4}(\PP^2)$. Thus, we have given a slightly complicated proof of the fact that no nonzero plane cubic has multiplicity at least two at four general points.
\end{example}
	
%%%%%%%%%%%%%%%%%%%%%%%%%%%%%%%%%%%%%%%%%%%%%%%%%%%%%%%%%%%%%%%%
\subsection{Deformation to the Normal Cone}
%%%%%%%%%%%%%%%%%%%%%%%%%%%%%%%%%%%%%%%%%%%%%%%%%%%%%%%%%%%%%%%%

We now place $(S,L)$ and $(Y,\pi^{*}L_B)$ into a family over $\AA^1$ using the method of deformation to the normal cone (see \cite[\S~5]{fulton98}). We use the open form of this deformation. Consider the extended Rees algebra 
\[
\cA\coloneqq \cO_{S}[z,\cI_Bz^{-1}]  = \bigoplus_{k\in\ZZ} \cI_B^{-k}z^{k} \subset \cO_{S}[z,z^{-1}]
\]
with the convention that $\cI_{B}^{k}=\cO_{S}$ when $k\leq 0$. Let $\pS \coloneqq \bSpec_{S}(\cA)$ with $\pr:\pS \to S$ the natural projection and $f:\pS\to \AA^{1}_{\LL}$ the map given by the inclusion $\LL[z]\hookrightarrow \cA$. Set $\cL\coloneqq \pr^*L$. 

\begin{lemma}
    The morphism $f:\pS\to \AA^1_{\LL}$ is smooth of relative dimension two. For $a\in \AA^1_{\LL}$ the fiber $\pS_{a}\coloneqq f^{-1}(a)$, together with the restriction of $\cL$ to $\pS_a$, is isomorphic to $(S,L)$ if $a\neq0$ and to $(Y,\pi^*L_B)$ for $a=0$.
\end{lemma}

\begin{proof}
	Inverting $z$ in the extended Rees algebra, we see that $\cA[\frac{1}{z}]\cong \cO_{S}[z,z^{-1}]$. Thus, $f^{-1}(\GG_m) \cong S\times \GG_m$, and on this set $f$ is given by projection on the second factor. Since $S$ is a smooth surface this projection is smooth of relative dimension $2$. We can therefore focus our attention on the fiber over $z=0$. On $S\setminus B$ we have that $\cI_{B}|_{S\setminus B} \cong \cO_{S\setminus B}$, and so $\cA|_{S\setminus B} \cong \cO_{S\setminus B}[z,z^{-1}]$. Therefore, $\pr^{-1}(S\setminus B) \cong (S \setminus B)\times \GG_m$ is contained in $f^{-1}(\GG_m)$. Thus $f^{-1}(0)$ lies over $B$. It remains to check smoothness near $B$. 

	Let $p\in B$. Since both $S$ and $B$ are smooth, we may choose an affine open neighborhood $U=\Spec(A) \subset S$ of $p$, an element $x\in A$, and a local equation $y\in A$ for $B\cap U$ such that $(x,y):U\to \AA^2_{\LL}$  is \'{e}tale. On the open set $U$ we have that $\cI_{B}|_{U} \cong \langle y \rangle$. Thus, on $U$ the extended Rees algebra can be identified with 
	\[
	A[z,u]/\langle uz-y\rangle \cong A\otimes_{\LL[x,y]}\LL[x,z,u]
	\]
	where $u=y/z$ and $\LL[x,z,u]$ is an $\LL[x,y]$-algebra via $x\mapsto x$ and $y\mapsto zu$. Since $A$ is \'{e}tale over $\LL[x,y]$ and base change preserves \'{e}tale-ness, $A[z,u]/\langle uz-y\rangle$ is \'{e}tale over $\LL[x,z,u]$. Since $\LL[x,z,u]$ is smooth of relative dimension two over $\LL[z]$ the same is true for $f$ over $U$. 
	
	We now shift to computing the fibers. The above description shows that $\pS_{a} \cong S\times \{a\}$ for $a\neq0$. Modulo $z$ the extended Rees algebra is 
	\[
	\bigoplus_{j\geq0} \cI_{B}^{j}/\cI^{j+1}_{B} \cong \Sym(N_{B/S}^{\vee}),
	\]
	which identifies $\pS_0$ with $Y$. On $\pS_{a}$ with $a\neq0$ the projection $\pr$ is the identity on $S$, so $\cL$ restricts to $L$. On $\pS_0$ the map $\pr$ factors as $Y\to B \to S$ and so $\cL|_{Y} = \pi^*(L|_{B}) = \pi^*L_{B}$. 
\end{proof}

Since $f:\pS\to \AA^1_{\LL}$ is not proper we cannot directly utilize the fact that the dimension of global sections of a line bundle is upper semicontinuous in a proper flat family \cite[\href{https://stacks.math.columbia.edu/tag/0BDN}{Tag 0BDN}]{stacks-project}. Instead using the Rees construction we build a finite free module whose fibers are $V$ and $V^{\inn}_{B}$, and obtain  the required openness from its evaluation map to a vector bundle of jets. Consider the extended Rees module
\[
\Rees_{F}(V)\coloneqq \sum_{k\geq0}\left(F^{k}V\right)\cdot z^{-k} \LL[z]  \subset V[z,z^{-1}]
\]
which we view as a module over $\LL[z]$ \cite[\S~6.5]{eisenbud95}. Since $z^{-k}\cI_{B}^k\subset \cA$ the elements of $\Rees_{F}(V)$ are regular sections of $\cL$ on $\pS$. 

\begin{lemma}\label{rees-family}
The Rees module $\Rees_{F}(V)$ has the following properties.
\begin{enumerate}
    \item The Rees module $\Rees_{F}(V)$ is a free $\LL[z]$-module of rank $\dim_{\LL}V$.
    \item The fiber of $\Rees_{F}(V)$ at $a\neq0$, viewed as a space of sections on $\pS_{a}\cong S$, is $V$. 
    \item The fiber of $\Rees_{F}(V)$ at $a=0$ is canonically $\gr_{F}(V)$; viewed as a space of sections on $\pS_{0}\cong Y$, it is $V^{\inn}_{B}$. 
\end{enumerate} 
\end{lemma}

\begin{proof}
	Choose a basis $g_1,\ldots,g_N$ for $V$ that is adapted to the filtration $F^{\bullet}$ and let $b_i = \max\{k\; | \; g_i \in F^kV\}$. Then $z^{-b_i}g_i$ span  $\Rees_{F}(V)$ as an $\LL[z]$-module. These elements are linearly independent over $\LL[z]$ since they are linearly independent over $\LL[z,z^{-1}]$. Part (2) and the first part of (3) follow from the standard identifications for the Rees construction \cite[\S~6.5]{eisenbud95}. 
	
	It remains to identify the fiber over $0$ with sections on $\pS_0 \cong Y$. Choose a local equation $y=0$ for $B$ and a trivialization $l_0$ of $L$ on $S$, and use the same notation for their pullbacks to $\pS$. A section $g \in F^jV$ can be written locally as $g=y^jal_0$. On the chart of the deformation where $y=zu$ we have that $z^{-j}g = u^{j}al_0$. Restricting to $z=0$ gives $u^{j}a|_{B}(l_0|_{B})$, which is $\ev_j(\rho_j(g))$. Thus, under the identification $\Rees_{F}(V)/z\Rees_{F}(V) \cong \gr_{F}(V)$, the restriction to the special fiber is the injective map defining $V_{B}^{\inn}$ as claimed. 
\end{proof}

\subsection{The Universal Family of Fat Points}\label{ref:fat-points-family}
Fix an integer $n\geq1$ and consider the relative configuration space of $n$ distinct ordered points 
\[
\pC\coloneqq \Conf_{n}(\pS/\AA^1_{\LL}) \coloneqq \underbrace{\left(\pS\times_{\AA^1_{\LL}} \cdots \times_{\AA^1_{\LL}}\pS \right)}_{\text{$n$ factors}} \setminus \Delta
\]
where $\Delta$ is the union of the pairwise diagonals. Let $\psi: \pC\to \AA^1_{\LL}$ be the natural map. Its fiber over a point $a\in \AA^{1}_{\LL}$ is $\Conf_{n}(S)$ when $a\neq0$ and is $\Conf_{n}(Y)$ when $a=0$. Form the fiber product
\[
\begin{tikzcd}[row sep = 3em, column sep = 3em]
\pS_{\pC}\coloneqq \pS\times_{\AA^1_{\LL}}\pC \arrow[d] \arrow[r, "\omega"] & \pC \arrow[d, "\psi"]\\
\pS \arrow[r, "f"] & \AA^1_{\LL}
\end{tikzcd}.
\]
The map $\omega$ has $n$ disjoint sections $\sigma_1,\ldots,\sigma_n:\pC\to \pS_{\pC}$ where at a point $c=(a;p_1,\ldots,p_n) \in \pC$  the section is given by $\sigma_i(c)=(p_i,c)$. Let $\cI_{i}\subset \cO_{\pS_{\pC}}$ be the ideal sheaf defining the image of $\sigma_i$. For $\mm=(m_1,\ldots,m_n)\in \ZZ^{n}_{\geq0}$, define the \emph{universal family of fat points with multiplicity $\mm$} by
\[
\pZ=\pZ(\mm) \coloneqq \bigcup_{i=1}^{n}\VV\left(\cI_{i}^{m_i}\right) \subset \pS_{\pC}.
\]
Its fiber over a point $(a;p_1,\ldots,p_n)$ is $Z_{\pp}(\mm) \subset \pS_{a}$. Thus, the fibers are fat-point schemes on $S$ for $a\neq0$ and on $Y$ when $a=0$. Let $\cL_{\pC}$ denote the pullback of $\cL$ to $\pS_{\pC}$. 

\begin{lemma}\label{lem:uni-fat-bundle}
    The morphism $\omega|_{\pZ}:\pZ\to \pC$ is finite and locally free of degree $\sum_{i} \binom{m_i+1}{2}$. Hence $\cP\coloneqq (\omega|_{\pZ})_*(\cL_{\pC}|_{\pZ})$ is a vector bundle of rank $\sum_{i}\binom{m_i+1}{2}$ on $\pC$. 
\end{lemma}

\begin{proof}
	Since the sections $\sigma_1,\ldots,\sigma_n$ are disjoint, $\pZ$ is a disjoint union $\pZ_1 \sqcup \cdots \sqcup \pZ_n$ where $\pZ_i \subset \pS_{\pC}$ is the subscheme defined by $\cI_i^{m_i}$. Hence $(\omega|_{\pZ})_*\cO_{\pZ} \cong \bigoplus_{i=1}^{n} (\omega|_{\pZ_i})_*\cO_{\pZ_i}$. Each $\pZ_i$ is finite over $\pC$, since it is a finite thickening of the image of $\sigma_i$. If $m_i=0$ then $\pZ_i$ is empty and we may omit this term going forward. The structure sheaf of $\pZ_i$ is filtered
	\[
	0 \subset \cI_{i}^{m_i-1}/\cI_{i}^{m_i} \subset \cdots \subset  \cI_{i}^{2}/\cI_{i}^{m_i} \subset  \cI_{i}/\cI_{i}^{m_i} \subset  \cO_{\pS_{\pC}}/\cI_{i}^{m_i}=\cO_{\pZ_i}.
	\]
	The successive quotients of this filtration are $\cI^a_{i}/\cI^{a+1}_{i}$ for $0\leq a<m_i$. Since $\omega$ is smooth of relative dimension two, each section $\sigma_i$ is a regular immersion of codimension two \cite[\href{https://stacks.math.columbia.edu/tag/067R}{Tag 067R}]{stacks-project}. Thus, its conormal bundle $\cI_i/\cI_i^2$ is locally free of rank two \cite[\href{https://stacks.math.columbia.edu/tag/063M}{Tag 063M}]{stacks-project}. Further, $\cI^a_{i}/\cI^{a+1}_{i} \cong \Sym^{a}(\cI_i/\cI^2_i)$ is locally free of rank $a+1$ on $\pC$. Since $\omega|_{\pZ_i}$ is finite, and finite pushforward is exact, pushing forward the above filtration gives a filtration of $(\omega|_{\pZ_i})_*\cO_{\pZ_i}$ whose successive quotients are locally free of ranks $1,\ldots,m_i$. Hence $\omega|_{\pZ}:\pZ\to \pC$ is finite and locally free of degree $\sum_{i} \binom{m_i+1}{2}$. 
	
	It remains to check that $\cP$ is a vector bundle of the claimed rank on $\pC$. Tensor the above filtration for $\cO_{\pZ_i}$ by $\cL_{\pC}|_{\pZ_i}$, and then push it forward via $\omega|_{\pZ_i}$. Since tensoring with a line bundle and finite pushforward are exact, this results in a filtration of $(\omega|_{\pZ_i})_*(\cL_{\pC}|_{\pZ_i})$ whose successive quotients are $\sigma_i^{*}\cL_{\pC}\otimes \Sym^{a}(\cI_i/\cI_i^2)$ for $0\leq a <m_i$. Since $\sigma_i^{*}\cL_{\pC}$ is a line bundle, each quotient is locally free of rank $a+1$, and so  $(\omega|_{\pZ_i})_*(\cL_{\pC}|_{\pZ_i})$ is locally free of rank $\binom{m_i+1}{2}$. Finally, since
	\[
	\cP \cong \bigoplus_{i=1}^n (\omega|_{\pZ_i})_*(\cL_{\pC}|_{\pZ_i}),
	\]
	we get that $\cP$ is a vector bundle on $\pC$ of the claimed rank. 
\end{proof}

\begin{proof}[Proof of Theorem~\ref{thm:initial-series-to-vanishing}]
	Let $\cR$ be the vector bundle on $\AA^1_{\LL}$ associated to the finite free module $\Rees_F(V)$. Restriction to the universal fat point scheme gives a morphism of vector bundles
	\[
	\begin{tikzcd}[row sep = .2em, column sep = 3.5em]
	\psi^*\cR \arrow[r] & \cP
	\end{tikzcd}
	\]
	on $\pC$. By Lemmas~\ref{rees-family} and~\ref{lem:uni-fat-bundle} the fiber of this morphism at $(0;\qq)$ is the evaluation map on $V_{B}^{\inn}$ and its fiber at $(a; \pp)$ for $a\neq0$ is the evaluation map on $V$. 
	
	The locus $U\subset \pC$ where this map is injective on fibers is defined by the nonvanishing of maximal minors, and so is open. The hypotheses imply that $(0;\qq) \in U$. The morphism $\psi$ is smooth and hence open. Thus, $\psi(U) \subset \AA^1_{\LL}$ is an open subset containing $0$; therefore it contains $a\in \AA^1_{\LL}\setminus\{0\}$. Thus $U \cap \pC_{a}$ is a nonempty open subset of $\pC_{a}\cong \Conf_{n}(S)$. At every point in this nonempty open subset the evaluation map is injective and so
	\[
	H^0\left(S,L\otimes \cI_{Z_{\pp}(\mm)}\right)=0.
	\]
	
\end{proof}

\begin{remark}
All of the constructions in this section commute with field extensions. Indeed global sections of a line bundle on a projective surface commute with extension of the ground field \cite[\href{https://stacks.math.columbia.edu/tag/02KH}{Tag 02KH}]{stacks-project}. Tensoring with a field is exact, and so the restriction maps, their kernels and images, the filtration $F^{\bullet}$, and its associated graded commute with extension of scalars. The universal fat point scheme and its evaluation maps do so as well. In particular, the kernel of an evaluation map after base change is the extension of the original kernel. The filtration, initial linear series, and Rees construction also make sense over a field that is not algebraically closed. Later we shall use them over $K=\KK((s))$ and apply the theorem after base changing to $\overline{K}$. 
\end{remark}

%%%%%%%%%%%%%%%%%%%%%%%%%%%%%%%%%%%%%%%%%%%%%%%%%%%%%%%%%%%%%%%
%%%%%%%%%%%%%%%%%%%%%%%%%%%%%%%%%%%%%%%%%%%%%%%%%%%%%%%%%%%%%%%
\section{Initial Terms on a Necklace}\label{sec:inital-terms}
%%%%%%%%%%%%%%%%%%%%%%%%%%%%%%%%%%%%%%%%%%%%%%%%%%%%%%%%%%%%%%%
%%%%%%%%%%%%%%%%%%%%%%%%%%%%%%%%%%%%%%%%%%%%%%%%%%%%%%%%%%%%%%%

In this section we study sections of a line bundle on a curve that degenerates to a \emph{necklace}: a strongly semistable regular model whose dual graph is a cycle. After extending the line bundle to the model, we normalize a section along a fixed rational component $C$ of the special fiber. With suitable choices, the restriction of the normalized section to $C$ is a Laurent polynomial in a coordinate on $C$, which we call the initial term of the section. The main result confines the exponents of this initial term to an interval of length $\deg(L)$. The position of this interval is controlled by a single real number, which we call the tropical moment. Among extensions of $L$ that have degree zero on $C$, the fractional part of the moment depends only on $L$. When the dual graph has at least three vertices, the integer part may be changed freely by a vertical twist. We then apply the same construction to the initial linear series on the normal bundle, obtaining Laurent polynomials in two variables. 

%%%%%%%%%%%%%%%%%%%%%%%%%%%%%%%%%%%%%%%%%%%%%%%%%%%%%%%%%%%%%%%%
\subsection{Semistable Models and Specialization}
%%%%%%%%%%%%%%%%%%%%%%%%%%%%%%%%%%%%%%%%%%%%%%%%%%%%%%%%%%%%%%%%

Fix an integer $e\geq1$. Let $K=\KK((s))$ and $\val:K\to \RR\cup \{\infty\}$ be the valuation where $\val(s) =e^{-1}$. The valuation group of $\val$ is $\frac{1}{e}\ZZ$ and the valuation ring is $R=\KK[[s]]$. In later sections we shall choose the integer $e$ carefully so that it is adapted to the situation at hand; however, in this section we keep things general. 

Let $B$ be a smooth projective geometrically connected curve over $K$ with a strongly semistable regular model $\pB\to \Spec(R)$. Thus, $\pB$ is regular, and the morphism $\pB \to \Spec(R)$ is proper and flat with generic fiber $B$, and its special fiber $\pB_{0}$ is a reduced curve over $\KK$ with smooth components and ordinary nodal singularities. At a node $p\in \pB$ of the special fiber, the completed local ring of $\pB$ is
\[
\widehat{\cO_{\pB,p}} \cong \KK[[z_1,z_2,s]]/\langle z_1z_2-s\rangle;
\]
see \cite{baker08}. Let $\Gamma$ be the metric dual graph of $\pB_0$. It has one vertex $x$ for each irreducible component $C_x \subset \pB_0$, one edge between $x$ and $y$ for each point in the intersection $C_x\cap C_y$, and every edge has length $1/e$. Points of $\Gamma$ may lie at vertices or in the interior of edges. 

Throughout this section we assume that $\Gamma$ is a cycle of total length $\gamma$ with at least two vertices. We fix a distinguished vertex $x_{\bullet} \in \Gamma$ whose component $C\coloneqq C_{x_{\bullet}}$ is isomorphic to $\PP^1_{\KK}$. We label the two points on $C$ meeting the rest of $\pB_0$ by $0$ and $\infty$, and orient $\Gamma$ so that the edge corresponding to the point $0$ is positively oriented leaving the vertex $x_{\bullet}$. Set $C^{\circ}\coloneqq C\setminus\{0,\infty\}$. 

For a prime divisor $Z$ on $B$, the \emph{horizontal closure} of $Z$ is the closure of $Z$ in $\pB$, which, at the risk of causing confusion, we denote by $\overline{Z}$. We extend this linearly to all divisors on $B$.  A divisor on $\pB$ of the form $\overline{Z}$ is called \emph{horizontal}. In contrast, we call a divisor of the form $\sum_{y\in V(\Gamma)}a_y C_y$ for $a_y\in \ZZ$ \emph{vertical}. Since $\pB$ is regular these Weil divisors are Cartier. 

Baker \cite{baker08}  and Baker--Jensen \cite{bakerJensen16} showed there are well-defined maps from divisors on $B$ to divisors on the graph $\Gamma$ (the specialization map) and from line bundles on $\pB$ to divisors on $\Gamma$ (the multidegree map). 

\begin{definition}
\begin{enumerate}
    \item The \emph{specialization} of a divisor $Z$ on $B$ is the divisor $\Sp(Z)$ on $\Gamma$ given by
    \[
    \Sp(Z)\coloneqq \sum_{y\in V(\Gamma)} \left\langle \overline{Z}\cdot C_y\right\rangle  y.
    \]
    \item  The \emph{multidegree} of a line bundle $\cL$ on $\pB$ is the divisor $\dd(\cL)$ on $\Gamma$ given by
    \[
    \dd(\cL) \coloneqq \sum_{y\in V(\Gamma)} \deg\left(\cL\big|_{C_y}\right)y.
    \]
    \end{enumerate}
\end{definition}

Both maps preserve degree relative to $B$ in that $\deg(\Sp(Z)) = \deg(Z)$ and $\deg(\dd(\cL))=\deg(\cL|_{B})$. Moreover, if $Z$ is effective then $\Sp(Z)$ is effective.

Let $\cL$ be a line bundle on $\pB$ with generic fiber $L$. A nonzero section $g\in H^0(B,L)$ is a rational section of $\cL$, and so we may write down its divisor of zeros and poles on $\pB$ where it has the form:
\begin{equation}\label{eq:form-of-divisor}
\div_{\pB}(g) = \overline{Z} + \sum_{y\in V(\Gamma)} m_{y}C_y \quad \quad \text{where} \quad \quad Z= \div_{B}(g)
\end{equation}
for some $m_y \in \ZZ$. Since $g$ is a global section of $L$ the divisor $\div_{B}(g)$ is effective. Define $\phi_g:\Gamma\to \RR$ by setting $\phi_{g}(y) =m_y/e$ for a vertex $y\in \Gamma$ and interpolating linearly along the edges. To such a function we associate a divisor $\div_{\Gamma}(\phi_g)$ on $\Gamma$ where 
\[
\div_{\Gamma}(\phi_g)(y) \coloneqq -\sum_{\text{$\tau$ leaving $y$}} \partial_{\tau}(\phi_g)
\] 
where $\partial_{\tau}(\phi_g)$ is the outgoing slope of $\phi_g$ in the direction $\tau$. The sum is zero except at finitely many points of $\Gamma$. The same construction works to define a divisor $\div_{\Gamma}(\phi)$ for any continuous piecewise linear function $\phi:\Gamma\to \RR$ with integral slopes.

\begin{lemma}\label{lem:bakerJensen-Special}
    The function $\phi_g$ is continuous and piecewise linear on $\Gamma$ with integral slopes and
    \[
    \Sp(Z)=\dd(\cL)+\div_{\Gamma}(\phi_g).
    \]
\end{lemma}

\begin{proof}
	 It is clear from the definition that $\phi_g$ is continuous and piecewise linear. Since each edge has length $1/e$ a direct computation shows that the slope of $\phi_g$ along an edge from $y$ to $z$ is $m_z-m_y$ which is an integer. The remainder is similar to the calculation in \cite[\S~2]{baker08}. It is enough to check the formula coefficient by coefficient. Since $\cL\cong \cO_{\pB}(\div_{\pB}(g))$, intersecting \eqref{eq:form-of-divisor} with $C_{z}$  gives
	 \[
	 \deg\left(\cL|_{C_z}\right) = \langle \overline{Z}\cdot C_{z}\rangle + \sum_{y\in V(\Gamma)} m_{y} \langle C_y \cdot C_{z} \rangle. 
	 \]
	  Since $\sum_{y\in V(\Gamma)} C_y$ is principal we have $\langle C_z^2\rangle = -\sum_{y\neq z}  \langle C_y \cdot C_{z} \rangle$. Combining this with the slope calculation above, we get 
	  \[
	 \deg\left(\cL|_{C_z}\right) = \langle \overline{Z}\cdot C_{z}\rangle + \sum_{y\neq z}\left( m_{y}-m_{z}\right) \langle C_y \cdot C_{z} \rangle =  \langle \overline{Z}\cdot C_{z}\rangle - \div_{\Gamma}(\phi_g)(z).
	 \]
\end{proof}
	 
%%%%%%%%%%%%%%%%%%%%%%%%%%%%%%%%%%%%%%%%%%%%%%%%%%%%%%%%%%%%%%%
\subsection{Initial Terms and Slopes}
%%%%%%%%%%%%%%%%%%%%%%%%%%%%%%%%%%%%%%%%%%%%%%%%%%%%%%%%%%%%%%%

Continue to let $\cL$ be a line bundle on $\pB$ with generic fiber $L$, and let $g\in H^0(B,L)$ be nonzero. We define the initial term of $g$ with respect to the fixed vertex $x_{\bullet}\in \Gamma$ as follows. 

\begin{definition}
    The \emph{initial term at $x_{\bullet}$} of a section $g$, as above, is
    \[
    \inn_{x_{\bullet}}(g)\coloneqq \left(s^{-\ord_{C}(g)}g\right)\big|_{C}.
    \]
\end{definition}

The section $s^{-\ord_{C}(g)}g$ has order zero along $C$, so its restriction to $C$ is a nonzero rational section of $\cL|_{C}$. Its poles can only occur at $0$ and $\infty$. Indeed, away from these nodes the only vertical component is $C$, and $s^{-\ord_{C}(g)}g$ has no horizontal poles. Since $\pB$ is normal this implies that the normalized section, and hence $\inn_{x_{\bullet}}(g)$, is regular on $C^{\circ}$. 

The restriction $\cL|_{C^{\circ}}$ is trivial, and so choosing a trivialization and a coordinate $v$ on $C$ with its zero at $0$ and its pole at $\infty$ identifies $\inn_{x_{\bullet}}(g)$ with a Laurent polynomial in $\KK[v^{\pm1}]$. An arbitrary change of trivialization of $\cL|_{C^{\circ}}$ can multiply this Laurent polynomial by $cv^a$, changing the exponents of the initial term. We avoid this ambiguity by assuming that $\deg(\cL|_{C})=0$ and using a nowhere vanishing global section of $\cL|_{C}\cong \cO_{C}$ as the trivialization. Since $C\cong \PP^1$ a change of this global section multiplies the Laurent polynomial by a nonzero scalar. Likewise a different choice of coordinate -- with the same zero and pole -- amounts to replacing $v$ by $cv$ for some $ c\in \KK^{\times}$. So the coefficients might be rescaled, but the exponents are unchanged. 

Thus, under the assumption that $\deg(\cL|_{C})=0$, and with the conventions above, the \emph{support} of $\inn_{x_{\bullet}}(g)$
\[
\supp\inn_{x_{\bullet}}(g) \coloneqq \left\{j \in \ZZ \;\; \mid \;\; \text{the coefficient of $v^{j}$ is nonzero}\right\}
\]
is independent of these choices. We can bound the support of $\inn_{x_{\bullet}}(g)$  in terms of basic divisorial data. Given a continuous piecewise linear function $\phi:\Gamma\to \RR$ with integer slopes, let $k_{+}(\phi)$ be the one-sided slope of the function $\phi$ leaving $x_{\bullet}$ in the positive direction, and let $k_{-}(\phi)$ be the one-sided slope along the positively oriented edge arriving at $x_{\bullet}$. For a section $g$ set  $k_{\pm}(g)=k_{\pm}(\phi_g)$. Further, let $i_0(g)$ and $i_\infty(g)$ be the local intersection numbers of $\overline{\div_{B}(g)}$ with $C$ at the nodes $0$ and $\infty$ respectively. 

\begin{lemma}\label{lem:support-bounds}
    With notation as above, assume $\deg(\cL|_C)=0$. Then
    \[
    \min \supp \inn_{x_{\bullet}}(g) = k_+(g) + i_0(g) \quad \quad \text{and} \quad \quad \max \supp \inn_{x_{\bullet}}(g) = k_-(g) - i_\infty(g).
    \]
\end{lemma}

\begin{proof}
	Choose local coordinates $z_1,z_2\in \widehat{\cO}_{\pB,0}$ so that the node at $0$ is given by $z_1z_2=s$, $C=\VV(z_1)$ and $C_{y}=\VV(z_2)$. Let $l_0$ trivialize $\cL$ around the node $0$. Since $\overline{\div_{B}(g)}$ is effective with no vertical components, we can write $g$ as $g = z_1^{\ord_{C}(g)}z_2^{\ord_{C_y}(g)}hl_0$ where $h$ is the local equation for $\overline{\div_{B}(g)}$ not divisible by $z_1$ or $z_2$. Thus, we have $s^{-\ord_{C}(g)}g=z_{2}^{\ord_{C_y}(g)-\ord_{C}(g)}hl_0$. Restricting this to $C$ gives
	\[
	\ord_{0}\left(\inn_{x_{\bullet}}(g)\right) = \ord_{C_y}(g)-\ord_{C}(g)+\ord_{z_2}\left(h(0,z_2)\right) = k_+(g) +i_0(g).
	\]
	The same argument at the node labeled $\infty$ gives
	\[
	\ord_{\infty}\left(\inn_{x_{\bullet}}(g)\right) =-k_{-}(g)+i_{\infty}(g),
	\]
	with the sign difference arising since the outgoing slope at this node is the negative of the positively oriented incoming slope. Using a nowhere vanishing global section of $\cL|_{C}$ we identify $\inn_{x_{\bullet}}(g)$ with a Laurent polynomial
	\[
	\sum_{i=k_{\min}}^{k_{\max}} c_{i}v^{i} \quad \quad \text{with}\quad  c_{k_{\min}},c_{k_{\max}}\neq0.
	\]
	Near $0$, this is $v^{k_{\min}}$ times a regular function whose value at $0$ is $c_{k_{\min}}\neq0$. So its order at $0$ is $k_{\min}$. Similarly, using $v^{-1}$ as the local coordinate at $\infty$ the above sum is $(v^{-1})^{-k_{\max}}$ times a regular function whose value at $\infty$ is $c_{k_{\max}}$. Since the trivializing section is nowhere vanishing, these are also just the orders of $\inn_{x_{\bullet}}(g)$ at $0$ and $\infty$ respectively. Thus, its smallest exponent is its order at $0$ and its largest exponent is minus its order at $\infty$. 
\end{proof}

\begin{remark}
    After passing to a completed algebraic closure of $K$, the component $C$ corresponds to a type-2 point of the Berkovich analytification of $B$ \cite{berkovich90}. After trivializing the line bundle, the residue-order calculation above is essentially the slope formula \cite[Thm.~5.15]{bakerPayneRabinoff13}. Two distinctions between our conventions and those used in \cite{bakerPayneRabinoff13} are needed. First, the function $\phi_g$ is defined by linear interpolation between vertices, and so it need not agree with the true tropicalization of $g$ determined by the analytic valuation function. Second, the specialization map $\Sp$ need not agree with the map induced by the retraction of closed points to the skeleton of $B^{\an}$; see \cite[Rem.~2.10]{baker08}.
 \end{remark}
    
%%%%%%%%%%%%%%%%%%%%%%%%%%%%%%%%%%%%%%%%%%%%%%%%%%%%%%%%%%%%%%%
\subsection{Tropical Moments and Support Bounds}\label{subsec:tropical-moment}  
%%%%%%%%%%%%%%%%%%%%%%%%%%%%%%%%%%%%%%%%%%%%%%%%%%%%%%%%%%%%%%%

For a point $y\in \Gamma\setminus\{x_{\bullet}\}$ define $\len_{x_{\bullet}}(y)\in (0,\gamma)$ to be the length of the positively oriented arc from $y$ to $x_{\bullet}$. Set $\len_{x_{\bullet}}(x_{\bullet})=0$. 

\begin{definition}
    The \emph{tropical moment} at $x_{\bullet}$ of a divisor $\sD$ on $\Gamma$ is
    \[
    \mom_{x_{\bullet}}(\sD)\coloneqq -\frac{1}{\gamma}\sum_{y\in \Gamma}\sD(y)\len_{x_{\bullet}}(y).
    \]
    If $\cL$ is a line bundle on $\pB$ we set $\mom_{x_{\bullet}}(\cL)\coloneqq \mom_{x_{\bullet}}(\dd(\cL))$.
\end{definition}

The tropical moment is additive in $\sD$. Since $\len_{x_{\bullet}}(y)\geq0$ for all $y\in \Gamma$ if $\sD$ is an effective divisor then $\mom_{x_{\bullet}}(\sD)\leq0$. The key computation is the moment of a principal divisor. 

\begin{lemma}\label{lem:moment-principal}
	If $\phi$ is a continuous piecewise linear function on $\Gamma$ with integral slopes then
	 \[
	 \mom_{x_{\bullet}}(\div_{\Gamma}(\phi))=-k_{+}(\phi).
	 \]
\end{lemma}

\begin{proof}
	Using the orientation we have fixed on $\Gamma$ we may parametrize $\Gamma$ by $t\in [0,\gamma]$ with both endpoints corresponding to $x_{\bullet}$. Let $p(t)$ denote the pullback of $\phi$, which is a piecewise linear function on $[0,\gamma]$. Choose a subdivision $0=t_0 < t_1 < \cdots <t_n = \gamma$ of the interval $[0,\gamma]$ on which $p(t)$ is linear. Let $a_i$ be its slope on $(t_i,t_{i+1})$. Without loss of generality let $y_i\in \Gamma$ denote the point corresponding to $t_i$. When $1\leq i\leq n-1$ the coefficient of $\div_{\Gamma}(\phi)$ at $y_i$ is $a_{i-1}-a_i$ and $\len_{x_\bullet}(y_i)=\gamma-t_i$. Since $\len_{x_\bullet}(x_\bullet)=0$ the $x_\bullet$ term contributes nothing to the tropical moment. Therefore
	\begin{align*}
	\mom_{x_\bullet}\left(\div_{\Gamma}(\phi)\right) &= \frac{-1}{\gamma}\sum_{y\in \Gamma}\div_{\Gamma}(\phi)(y)\len_{x_\bullet}(y) = \frac{-1}{\gamma}\sum_{i=1}^{n-1}\left(a_{i-1}-a_i\right)\left(\gamma-t_i\right)\\
	&=\frac{-1}{\gamma}\left(\sum_{i=0}^{n-2} a_i\left(\gamma-t_{i+1}\right)  - \sum_{i=1}^{n-1}a_i\left(\gamma-t_i\right)\right).\\
	\intertext{We may reindex both sums to $0\leq i \leq n-1$ by adding $a_{n-1}(\gamma-t_n)$ to the first sum, which is zero since $t_n=\gamma$, and adding $a_{0}(\gamma-t_{0})-a_0\gamma$ to the second sum, which is zero since $t_{0}=0$:} 
	&= \frac{-1}{\gamma}\left(a_{0}\gamma + \sum_{i=0}^{n-1}a_i\left(\gamma-t_{i+1}\right)  - \sum_{i=0}^{n-1}a_i\left(\gamma-t_i\right)\right) = \frac{-1}{\gamma}\left(a_0\gamma  - \sum_{i=0}^{n-1}a_i(t_{i+1}-t_{i})\right)\\
	&= - a_0 + \frac{p(\gamma) - p(0)}{\gamma} = -a_0.
	\end{align*}
	This finishes the proof since $a_0 = k_{+}(\phi)$.
\end{proof}

The following lemma bounds the slopes at $x_{\bullet}$ of a continuous piecewise linear function on $\Gamma$ with integral slopes in terms of the tropical moment. 

\begin{lemma}\label{lem:tropical-moments-bounds}
    Let $\sD$ be a divisor on $\Gamma$ with $\sD(x_{\bullet})=0$. If $\phi:\Gamma\to \RR$ is a continuous piecewise linear function with integral slopes and $\div_{\Gamma}(\phi)+\sD\geq0$ then
    \[
    \mom_{x_{\bullet}}(\sD) \leq k_{+}(\phi) \leq k_{-}(\phi) \leq \mom_{x_{\bullet}}(\sD)+\deg(\sD).
    \]
\end{lemma}

\begin{proof}
	Applying the additivity of tropical moments and Lemma~\ref{lem:moment-principal}  to $\div_{\Gamma}(\phi)+\sD$ gives
	\[
	\mom_{x_{\bullet}}\left(\div_{\Gamma}(\phi)+\sD\right) = \mom_{x_{\bullet}}\left(\div_{\Gamma}(\phi)\right)+\mom_{x_{\bullet}}\left(\sD\right) = -k_{+}(\phi)+\mom_{x_{\bullet}}(\sD).
	\]
	By assumption $\div_{\Gamma}(\phi)+\sD$ is effective, and so $\mom_{x_{\bullet}}\left(\div_{\Gamma}(\phi)+\sD\right)\leq 0$ giving the left-hand inequality. For the middle inequality, notice that since $\sD(x_{\bullet})=0$ the inequality $(\div_{\Gamma}(\phi)+\sD)(x_{\bullet})\geq0$ implies $\div_{\Gamma}(\phi)(x_{\bullet}) = k_{-}(\phi)-k_{+}(\phi)\geq0$. Finally, for the right-hand inequality since $\deg(\div_{\Gamma}(\phi))=0$ we have 
	\[
	\deg(\sD)- k_{-}(\phi)+k_{+}(\phi) = \sum_{y \neq x_{\bullet}} \left( \sD(y) + \div_{\Gamma}(\phi)(y)\right).
	\]
	Since $\sD+\div_{\Gamma}(\phi)$ is effective and $0\leq \len_{x_{\bullet}}(y)<\gamma$ we have
	\begin{align*}
	k_{+}(\phi) - \mom_{x_{\bullet}}(\sD) = - \mom_{x_\bullet}\left(\sD+\div_{\Gamma}(\phi)\right) &= \frac{1}{\gamma}\sum_{y\neq x_{\bullet}}\left(\sD(y) + \div_{\Gamma}(\phi)(y)\right) \len_{x_\bullet}(y) \\&\leq \sum_{y\neq x_{\bullet}}\left(\sD(y) + \div_{\Gamma}(\phi)(y)\right) = \deg(\sD)-k_{-}(\phi)+k_{+}(\phi).
	\end{align*}
\end{proof}

\begin{corollary}\label{cor:support-interval}
    Let $\cL$ be a line bundle on $\pB$ with generic fiber $L$. If $\deg(\cL|_{C})=0$ then for every nonzero section $g\in H^0(B,L)$:
    \[
    \supp \inn_{x_{\bullet}}(g) \subset \left[\mom_{x_{\bullet}}(\cL),\; \mom_{x_{\bullet}}(\cL)+\deg(L)\right]\cap \ZZ.
    \]
\end{corollary}

\begin{proof}
	Since $\deg(\cL|_C)=0$ Lemma~\ref{lem:support-bounds} implies that 
	\[
	\min \supp \inn_{x_{\bullet}}(g) = k_+(g) + i_0(g) \quad \quad \text{and} \quad \quad \max \supp \inn_{x_{\bullet}}(g) = k_-(g) - i_\infty(g).
	\]
	We now apply Lemma~\ref{lem:tropical-moments-bounds} with $\sD=\dd(\cL)$, which we may do since $ \dd(\cL)(x_{\bullet})=\deg(\cL|_{C})=0$ and  Lemma~\ref{lem:bakerJensen-Special} implies $\div_{\Gamma}(\phi_g)+\dd(\cL)=\Sp(\div_{B}(g))\geq0$. Doing so we get that
	\begin{align*}
	&\mom_{x_{\bullet}}(\cL)+i_{0}(g) \leq \min \supp \inn_{x_{\bullet}}(g) = k_+(g) + i_0(g) \\
	 &\max \supp \inn_{x_{\bullet}}(g) = k_-(g) - i_\infty(g) \leq \mom_{x_{\bullet}}(\cL)  + \deg(\dd(\cL)) -i_{\infty}(g).
	\end{align*}
	The claim now follows from the fact that $\deg(\dd(\cL))=\deg(L)$ and that $i_{0}(g),i_{\infty}(g)\geq0$.
\end{proof}

We call the intervals in Lemma~\ref{lem:tropical-moments-bounds} and Corollary~\ref{cor:support-interval} the support intervals of $\sD$ and $\cL$ respectively. When $\deg(L)\geq0$ the support interval of $\cL$ has length $\deg(L)$, and so it contains  $\deg(L)+1$ integers if $\mom_{x_{\bullet}}(\cL)\in \ZZ$, and $\deg(L)$ integers otherwise. When $\deg(L)<0$,  $H^0(B,L)=0$ and the interval is empty.

\begin{remark}
    By Lemma~\ref{lem:moment-principal} the tropical moment of a principal divisor $\div_{\Gamma}(\phi)$ is the integer $-k_{+}(\phi)$. It therefore induces a homomorphism 
   \[
   \begin{tikzcd}[row sep = .2em, column sep = 3.5em]
   \Pic\left(\Gamma\right) \arrow[r, "\widehat{\mom}_{x_{\bullet}}"] & \RR/\ZZ.
   \end{tikzcd}
   \]
   As we will show in a moment, up to scaling by $\frac{1}{\gamma}$ this is the tropical Abel--Jacobi map. Let $\vartheta_{x_{\bullet}}(y)\coloneqq \gamma - \len_{x_{\bullet}}(y)$ for $y\neq x_{\bullet}$ and $\vartheta_{x_{\bullet}}(x_{\bullet})=0$. If $\sD$ is a divisor of degree $0$ then 
    \[
    \mom_{x_{\bullet}}(\sD) = \frac{-1}{\gamma} \sum_{y\neq x_{\bullet}} \sD(y)\len_{x_{\bullet}}(y) = -\sum_{y\neq x_{\bullet}}\sD(y) + \frac{1}{\gamma}\sum_{y \in \Gamma} \sD(y)\vartheta_{x_{\bullet}}(y) = \sD(x_{\bullet})+\frac{1}{\gamma}\sum_{y \in \Gamma} \sD(y)\vartheta_{x_{\bullet}}(y)
    \]
    with the last equality using the fact that $\deg(\sD)=0$. Hence
    \[
     \mom_{x_{\bullet}}(\sD) \equiv \frac{1}{\gamma}\sum_{y \in \Gamma} \sD(y)\vartheta_{x_{\bullet}}(y) \mod \ZZ.
    \]
     When $\Gamma$ is a cycle, its Jacobian is $\Pic^{0}(\Gamma) \cong \RR/\gamma\ZZ$, and the tropical Abel--Jacobi map based at $x_{\bullet}$ is the homomorphism $\Pic(\Gamma)\to \Pic^0(\Gamma)$  given on the class of a point by $[y]\mapsto \vartheta_{x_{\bullet}}(y) \mod
     \gamma\ZZ$ \cite{bakerNorine07,mikhalkinZharkov08}.  Thus, the tropical moment modulo $\ZZ$ is the image of $\sD$ under the tropical Abel--Jacobi map composed with the map $\RR/\gamma\ZZ\to\RR/\ZZ$ given by $[a]\mapsto[a/\gamma]$. In other words, the tropical moment is a real lift, defined on divisors, of this normalized Abel--Jacobi map, and it is this lift rather than its class modulo $\ZZ$ that controls the endpoints of the support interval in Corollary~\ref{cor:support-interval}. More generally, if $\sD$ has arbitrary degree the above calculation shows that $\mom_{x_{\bullet}}(\sD)\mod \ZZ$ is, up to a factor of $\gamma^{-1}$, the image of $\sD-\deg(\sD)x_{\bullet}$ under the tropical Abel--Jacobi map.
\end{remark}

\begin{remark}
Given the role of theta functions in \cite{openAI-P2,tevelev26}, it would be interesting to understand the relationship between our support intervals, the tropical theta functions of Foster, Rabinoff, Shokrieh, and Soto \cite{FRSS18}, and the arguments in the original OpenAI release. 
\end{remark}
%%%%%%%%%%%%%%%%%%%%%%%%%%%%%%%%%%%%%%%%%%%%%%%%%%%%%%%%%%%%%%%
\subsection{Vertical Twists and Support Intervals}
%%%%%%%%%%%%%%%%f%%%%%%%%%%%%%%%%%%%%%%%%%%%%%%%%%%%%%%%%%%%%%%%

Corollary~\ref{cor:support-interval} bounds the support of initial terms of sections of a line bundle $L$ using an extension $\cL$ whose degree on $C$ is zero. We now explain how to find such an extension and how a twist by a vertical divisor, which we call a vertical twist, changes its tropical moment. 
 
 Every line bundle $L$ on $B$ has an extension to $\pB$. Indeed, choose a divisor $D$ on $B$ with $L\cong \cO_{B}(D)$. Since $\pB$ is regular, the horizontal closure $\overline{D}$ on $\pB$ of $D$ is Cartier and $\cO_{\pB}(\overline{D})$ extends $L$. If $\cV=\sum_{y\in V(\Gamma)}a_yC_y$ is a vertical divisor on $\pB$ then $\cO_{\pB}(\cV)|_{B}$ is trivial, so $\cL(\cV)$ has the same generic fiber as $\cL$. Conversely, all extensions arise by such twists. 
 
\begin{lemma}\label{lem:extensions-exists}
	 If $\cL$ and $\cL'$ are line bundles on $\pB$ with generic fibers isomorphic to  $L$ then  $\cL\cong \cL'(\cV)$ for some vertical divisor $\cV$. \end{lemma}

\begin{proof}
	The generic fiber of  $\cL'^{-1} \otimes \cL$ is $\cO_{B}$ and so there is a rational section $\tau$ of $\cL'^{-1}\otimes \cL$ that is a regular nowhere-vanishing section on $B$. Hence $\cV =  \div_{\pB}(\tau)$ is a vertical divisor and $\cL \cong \cL'(\cV)$. 
\end{proof}	

We can define a continuous piecewise linear function $\phi_{\cV}:\Gamma\to \RR$ just as we did for rational sections. At each vertex $y\in \Gamma$ let $\phi_{\cV}(y)=a_y/e$ and then interpolate linearly along each edge. The same argument behind the proof of Lemma~\ref{lem:bakerJensen-Special} shows that
\[
\dd(\cL(\cV)) = \dd(\cL) - \div_{\Gamma}(\phi_\cV).
\]

\begin{lemma}\label{lem:moving-interval-1} Let $\cL$ be a line bundle on $\pB$ with generic fiber $L$. 
\begin{enumerate}
	\item If $\cV$ is a vertical divisor with $k_{+}(\phi_{\cV})=k_{-}(\phi_{\cV})=k$ then 
	\[
	\deg(\cL(\cV)|_{C})=\deg(\cL|_{C}) \quad \quad \text{and} \quad \quad \mom_{x_{\bullet}}(\cL(\cV)) = \mom_{x_{\bullet}}(\cL)+k.
	\]
	\item Among extensions of $L$ with degree zero on $C$ the tropical moment is well-defined modulo $\ZZ$.
\end{enumerate}
\end{lemma}

\begin{proof}
	For part (1) consider the divisor $\dd(\cL(\cV)) - \dd(\cL)$. The coefficient of $x_{\bullet}$ in this divisor is the difference $\deg(\cL(\cV)|_{C}) - \deg(\cL|_{C})$. The identity preceding the lemma gives 
	\[
	\dd(\cL(\cV)) - \dd(\cL) = -\div_{\Gamma}(\phi_{\cV}),
	\]
	and so the coefficient on $x_{\bullet}$ is also equal to $k_{+}(\phi_\cV)-k_{-}(\phi_{\cV})$. Thus, if $k_{+}(\phi_{\cV})=k_{-}(\phi_{\cV})=k$ the coefficient on $x_{\bullet}$ is equal to zero and so $\deg(\cL(\cV)|_{C}) - \deg(\cL|_{C})=0$. Computing the tropical moment using that it is linear we get
	\[
	 \mom_{x_{\bullet}}\left(\dd(\cL(\cV))\right) - \mom_{x_{\bullet}}\left(\dd(\cL)\right) = \mom_{x_{\bullet}}\left(\dd(\cL(\cV)) - \dd(\cL)\right) = -\mom_{x_{\bullet}}\left(\div_{\Gamma}(\phi_{\cV})\right) = k_{+}(\phi_{\cV}) = k,
	\]
	which completes the proof of (1). For part (2) let $\cL$ and $\cL'$ be two extensions of $L$ with degree zero on $C$. By the preceding Lemma~\ref{lem:extensions-exists}, $\cL\cong \cL'(\cV)$ for some vertical divisor $\cV$. The degree calculation above shows that $k_{+}(\phi_{\cV}) = k_{-}(\phi_{\cV})$. Since these are integers part (1) shows that the tropical moments of $\cL$ and $\cL'$ differ by an integer, proving part (2).
\end{proof}

\begin{lemma}\label{lem:moving-interval-2} Assume that $\Gamma$ has at least three vertices. Let $L$ be a line bundle on $B$.
\begin{enumerate}
	\item For every integer $c\in \ZZ$ there exists an extension $\cL$ of $L$ such that $\deg(\cL|_{C})=c$. 
	\item Fix an extension $\cL_0$ of $L$ with degree zero on $C$, and set $\eta = \mom_{x_{\bullet}}(\cL_0)$. The set of tropical moments of all extensions of $L$ with degree zero on $C$ is exactly $\eta+\ZZ$.
\end{enumerate}
\end{lemma}

\begin{proof}
	Let $C_{+}$ and $C_{-}$ be the components of the special fiber intersecting $C$  at $0$ and $\infty$ respectively. Since $\#V(\Gamma)>2$ these two components are distinct and each meets $C$ exactly once. Let $\cL'$ be any extension of $L$ and set $k=\deg(\cL'|_{C})$. Since $C_+$ meets $C$ transversely at one point, twisting by $C_+$ increases the degree on $C$ by one. Thus, setting $\cL=\cL'((c-k)C_{+})$ we have $\deg(\cL|_{C})=c$ and since $C_+$ is a vertical divisor this twist does not change the generic fiber. This proves part (1). 
	
	Turning to part (2) for $k\in \ZZ$ set $\cV_k \coloneqq kC_{+}-kC_{-}$. The degree of $\cV_{k}$ on $C$ is $0$. Moreover, the coefficients of $C$, $C_{+}$, and $C_{-}$ appearing in $\cV_k$ are $0$, $k$, and $-k$, respectively, so $k_{\pm}(\phi_{\cV_k}) = k$. Part (1) of Lemma~\ref{lem:moving-interval-1} shows that $\cL_0(\cV_k)$ has degree zero on $C$ and has tropical moment $\eta+k$. Together with part (2) of Lemma~\ref{lem:moving-interval-1} this proves the claim. 
\end{proof}

\begin{remark}
	The assumption that $\Gamma$ has at least three vertices is necessary in Lemma~\ref{lem:moving-interval-2}. If $\Gamma$ has only two vertices the other component $C'$ meets  $C$ twice. Since $\langle C^2\rangle =  -2$ and $\langle C \cdot C' \rangle = 2$ twisting by $aC+bC'$ changes the degree on $C$ by $2(b-a)$. Thus, vertical twists cannot produce line bundles of every degree. Moreover, a degree-preserving twist has $a=b$, which means both slopes at $x_{\bullet}$ are zero implying the tropical moment does not change. Hence neither part of Lemma~\ref{lem:moving-interval-2} holds in general for a cycle with two vertices. 
	
	Luckily, we can add vertices after a finite ramified extension of $K$. For example, after the base change $s=t^2$, each node has local equation $z_1z_2=t^2$. Resolving this singularity inserts a rational component, subdividing the corresponding edge of $\Gamma$. After normalizing the valuation so that $\val(t)=1/(2e)$, we find that each new edge has length $1/(2e)$, so the total length of the dual graph is unchanged. The resulting model is again strongly semistable and regular, with generic fiber the base change of $B$; see \cite{baker08}. In our application this subdivision is built into Proposition~\ref{prop:necklace}.
\end{remark}

%%%%%%%%%%%%%%%%%%%%%%%%%%%%%%%%%%%%%%%%%%%%%%%%%%%%%%%%%%%%%%%
\subsection{Initial Terms on the Normal Bundle}
%%%%%%%%%%%%%%%%%%%%%%%%%%%%%%%%%%%%%%%%%%%%%%%%%%%%%%%%%%%%%%%
Let us return to the setting of Section~\ref{sec:specialization-normal}, now over $K$. Let $S$ be a smooth projective geometrically connected surface over $K$ containing a smooth curve $B\subset S$. Write $N=N_{B/S}$ for the normal bundle of $B$. Fix a line bundle $L$ on $S$, and set $L_{B}=L|_B$ and $V=H^0(S, L)$. Recall the initial linear series of $V$ is
\[
V^{\inn}_{B} \cong \bigoplus_{j\geq0}\ev_j(V_{B,j}) \quad \quad \text{where} \quad \quad V_{B,j}\subset H^0\left(B,L_{B}\otimes N^{-j}\right)
\]
with $V_{B,j}=0$ for $j\gg0$. In this section we extend the initial term construction to $V_B^{\inn}$, obtaining Laurent polynomials in a coordinate $v$ on $C$ and a fiber coordinate $u$ normal to $B$. 

Fix extensions $\cN$ and $\cL_{B}$ on $\pB$ of $N$ and $L_B$ respectively. Further, fix identifications $\cN|_{B}\cong N$ and $\cL_{B}|_{B} \cong L_B$. Assume that $\deg(\cN|_{C}) = \deg(\cL_{B}|_{C}) = 0$. Such extensions exist when $\Gamma$ has at least three vertices by Lemma~\ref{lem:moving-interval-2}. This holds in our applications, where the extensions will be explicitly constructed. Note that the degree-zero conditions concern only the chosen component $C$ of the special fiber. We impose no conditions on the degree of $\cN$ or $\cL_B$, or their restrictions to the other components of the special fiber.

Set $\cA_{j}\coloneqq \cL_{B}\otimes \cN^{-j}$. An element  $g\in V_{B,j}$ is then a rational section of $\cA_{j}$. Considering the total spaces of $\cN$ and $\cN|_{C}$ we have the following diagram 
\[
\begin{tikzcd}[row sep = 3em, column sep = 4.5em]
 \pY_C \coloneqq \bSpec_{C}\left(\Sym (\cN|_C)^{\vee}\right) \arrow[r, "\pi_C"] \arrow[d,hook] & C \arrow[d, hook] \\
\pY\coloneqq \bSpec_{\pB}\left(\Sym \cN^{\vee}\right) \arrow[r, "\pi"] & \pB 
\end{tikzcd}.
\]
Let $\tau_{\cN}$ be the tautological section of $\pi^*\cN$. On the generic fiber of $\pY$, every element  $F\in V_{B}^{\inn}$ has a unique expression
\[
F = \sum_{j\geq0} \pi^*(g_j)\tau_{\cN}^{j} \quad\quad \text{where} \quad \quad g_{j} \in V_{B,j}
\]
with only finitely many $g_j$ nonzero. We use this expression to regard $F$ as a rational section of $\pi^*\cL_B$ on $\pY$ and to define its initial term. 

\begin{definition}
    The \emph{initial term at $x_{\bullet}$} of a nonzero element $F\in V^{\inn}_{B}$ is
    \[
    \inn_{x_{\bullet}}(F)\coloneqq \left(s^{-\ord_{C}(F)}F\right)\big|_{\pY_{C}},
    \]
    where $\ord_{C}(F) \coloneqq \min\{\ord_{C}(g_j) \; | \; g_j\neq0\}$.
\end{definition}

\begin{remark}
The initial term of an element of $V^{\inn}_{B}$ depends on the extensions $\cL_{B}$ and $\cN$ and the identifications of their generic fibers with $L_{B}$ and $N$. Changing any of these can change the orders of the coefficients along $C$, and hence change the initial term. As such we keep this data fixed when making computations.
\end{remark}

Each normalized coefficient $s^{-\ord_{C}(F)}g_j$ has nonnegative order along $C$ and has no horizontal poles. Hence $s^{-\ord_{C}(F)}g_j$ has no poles along prime divisors in an open set containing $C^{\circ}$. Since $\pB$ is normal this implies that $s^{-\ord_{C}(F)}g_j$ is regular on an open set containing $C^{\circ}$. It follows that $s^{-\ord_{C}(F)}F$ is regular on the preimage of this open set in $\pY$. Hence $\inn_{x_{\bullet}}(F)$ is regular on $\pi_{C}^{-1}(C^{\circ})$. The order along $C$ of at least one coefficient is exactly $\ord_{C}(F)$, and the restriction to $C$ of such normalized coefficients is nonzero. Since distinct powers of the fiber coordinate are linearly independent, this means $\inn_{x_{\bullet}}(F)\neq0$. 

\begin{remark}
We normalize each coefficient by the same power of $s$. A coefficient whose order along $C$ is greater than $\ord_{C}(F)$ does not contribute to the initial term. The point of using this common power is that we are multiplying the entire section $F$ by a nonzero scalar $s^{-\ord_{C}(F)}\in K^{\times}$, and so its order of vanishing at every point of the generic fiber is unchanged. Together with the regularity of $s^{-\ord_{C}(F)}F$ near $\pi_C^{-1}(C^{\circ})$ this allows us to specialize jets as in Lemma~\ref{lem:jet-specialization}.
\end{remark}

Since $\cN|_{C}$ and $\cL_{B}|_{C}$ have degree zero and $C\cong \PP^1$ we may choose nowhere vanishing sections $n_0$ of $\cN|_{C}$ and $l_0$ of $\cL_{B}|_{C}$. Fix a coordinate $v$ on $C$ with its zero at $0$ and its pole at $\infty$; this gives $C^{\circ} \cong \Spec(\KK[v^{\pm1}])$. The section $l_0\otimes n_{0}^{-j}$ trivializes $\cA_j|_{C}$. Write
\[
\left( s^{-\ord_{C}(F)}g_{j}\right)\big|_{C^\circ}  = f_{j} \left( l_0\otimes n_0^{-j} \right).
\]
The discussion in the previous paragraph shows that $f_{j} \in \KK[v^{\pm1}]$, and $f_j = 0$ when $g_j = 0 $ or $\ord_{C}(g_j)>\ord_{C}(F)$. Otherwise $f_j$ represents $\inn_{x_{\bullet}}(g_j)$ in this chosen trivialization. 

The section $n_0$ gives a coordinate $u$ on each fiber of $\pi_C:\pY_C\to C$ so that $\tau_{\cN}|_{\pY_C} = u \pi_C^*n_0$. We often refer to $u$ as the fiber coordinate. Consequently on $\pi_C^{-1}(C^{\circ})$ we have that 
\[
\inn_{x_{\bullet}}(F) =\sum_{j\geq0} \pi_C^*\left( f_j(l_0\otimes n_0^{-j})\right)\left( u\pi_C^*n_0\right)^{j} = \left(\sum_{j\geq0} f_j u^j \right) \left(\pi_C^{*}l_0\right).
\]
After trivializing $\pi_C^{*}(\cL_B|_C)$ by $\pi_C^*l_0$, we identify $\inn_{x_{\bullet}}(F)$ with an element in $\KK[u,v^{\pm1}]$. The exponent of $u$ records the normal degree $j$ and the exponent of $v$ comes from the initial terms of the coefficients. 

For our jet calculations we work where both coordinates are nonzero, which allows us to move any point to $(1,1)$. Doing this will allow us to reduce our problems to questions about jets at $(1,1)$ where the Laurent monomial methods discussed in Section~\ref{sec:torus} are of use. We therefore restrict to the open torus 
\[
\TT_C\coloneqq \pi_{C}^{-1}(C^{\circ}) \setminus\{u=0\}.
\]
 On this subset the function $u$ is invertible and our choices identify $\TT_C \cong \Spec(\KK[u^{\pm1},v^{\pm1}])$. So the initial terms can be viewed as Laurent polynomials on this two-dimensional torus. 

For a section $q:\Spec(R)\to \pY$, we write $\eta_q$ for its value at the generic point and $\red(q)$ for its value at the closed point. Thus, $\eta_q$ lies in the generic fiber of $\pY$, while $\red(q)$ lies in the special fiber of $\pY$ and is obtained by reducing $q$ modulo $s$. To use our Laurent polynomials to study vanishing, we need to compare jets on the generic and special fibers. The following lemma shows that vanishing to order $m$ at $\eta_q$ passes to the initial terms at $\red(q)$. The key is that the jets of the normalized section lie in a finite free $R$-module, and so a jet that vanishes over $K$ is already zero over $R$, meaning it remains zero after reduction. 

\begin{lemma}\label{lem:jet-specialization}
    Let $q$ be a section of $\pY \to \Spec(R)$ whose reduction $\red(q)$ lies in $\TT_{C}$, and let $\eta_{q}$ denote its value at the generic point. If $F\in V^{\inn}_{B}$ is nonzero and vanishes to order at least $m\geq0$ at $\eta_{q}$ then $\inn_{x_{\bullet}}(F)$ vanishes to order at least $m$ at $\red(q)$. 
\end{lemma}

\begin{proof}
	If $m=0$ there is nothing to prove, so assume $m>0$. By the preceding discussion $s^{-\ord_{C}(F)}F$ is regular on an open set $\pU\subset \pY$ containing $\red(q)$. Since $\red(q)$ lies over $C^{\circ}$ its image in $\pB$ lies away from the nodes of the special fiber. Since $\pB\to \Spec(R)$ is smooth of relative dimension one near $\pi(\red(q))$, $\pY\to \Spec(R)$ is smooth of relative dimension two near $\red(q)$. Since the smooth locus is open, we may shrink $\pU$ to also be smooth over $\Spec(R)$. Also by removing the preimage of the other components of the special fiber of $\pB$ we may assume that the special fiber of $\pU$ is contained in $\pi_{C}^{-1}(C^{\circ})$. Since $q^{-1}(\pU)$ contains the closed point of $\Spec(R)$ the map $q$ factors through $\pU$. Write $\pU_K$ and $\pU_{\KK}$ for the generic and special fibers of $\pU$ respectively. This discussion can be summarized in the following commutative diagram
	\[
	\begin{tikzcd}[row sep = 4.5em, column sep = 4.5em]
	\Spec(\KK) \arrow[r, "\red(q)"] \arrow[d, hook] & \pU_{\KK} \arrow[r, hook] \arrow[d, hook] & \pi_{C}^{-1}\left(C^{\circ}\right) \arrow[r, "\pi_C"]\arrow[d, hook] & C^{\circ} \arrow[d, hook] \\
	\Spec(R) \arrow[r, "q"] & \pU \arrow[r, hook] & \pY \arrow[r, "\pi"] & \pB
	\end{tikzcd}.
	\]
	
	Let $\cI_{q} \subset \cO_{\pU}$ be the ideal sheaf of the image of $q$ in $\pU$, and let $Z_{m} = \VV(\cI_{q}^{m})\subset \pU$. The generic and special fibers of $Z_m$ are the fat points of order $m$ at $\eta_q$ and $\red(q)$ respectively. The same regular immersion argument used in Lemma~\ref{lem:uni-fat-bundle} shows that $Z_m \to \Spec(R)$ is finite and locally free of degree $\binom{m+1}{2}$.  Thus, $M\coloneqq H^0\left(Z_m, \left(\pi^*\cL_{B}\right)|_{Z_m}\right)$ is a finite free $R$-module. The restriction to $Z_m$ and its fibers gives a commutative diagram
	\[
	\begin{tikzcd}[row sep = 4.5em, column sep = 4.5em]
	H^0\left(\pU_{\KK}, (\pi^*\cL_B)|_{\pU_{\KK}}\right) \arrow[r] & M\otimes_R \KK \\
	H^0\left(\pU, \pi^*\cL_B\right) \arrow[r] \arrow[u] \arrow[d] & M \arrow[u, two heads] \arrow[d, hook] \\
	H^0\left(\pU_{K}, (\pi^*\cL_B)|_{\pU_{K}}\right) \arrow[r] & M\otimes_{R} K
	\end{tikzcd}. 
	\]
	The horizontal maps record the jets of order less than $m$. The upper right arrow is reduction modulo $s$ and so is surjective, and the lower right arrow is injective since $M$ is a free $R$-module. The element $s^{-\ord_{C}(F)}F$ lives in $H^0\left(\pU, \pi^*\cL_B\right)$, which is on the left of the middle row. Multiplication by the nonzero scalar $s^{-\ord_{C}(F)}\in K^{\times}$ does not change the order of vanishing of $F$ at $\eta_q$. Thus, $s^{-\ord_{C}(F)}F$ still vanishes to order at least $m$ at $\eta_q$, and so its restriction to the generic fiber of $Z_m$ is zero. Since the lower square commutes this is the image of $(s^{-\ord_{C}(F)}F)|_{Z_m} \in M$ under the map in the lower right column. Since this map is injective this implies that $(s^{-\ord_{C}(F)}F)|_{Z_m}$ was already zero. 
	
	The upper square then shows that the image of $(s^{-\ord_{C}(F)}F)|_{\pU_{\KK}}$ in $M\otimes_{R}\KK$ is equal to zero. Since this restriction agrees with the restriction defining $\inn_{x_{\bullet}}(F)$ we have that $\inn_{x_{\bullet}}(F)$ vanishes to order at least $m$ at $\red(q)$. 	
\end{proof}

%%%%%%%%%%%%%%%%%%%%%%%%%%%%%%%%%%%%%%%%%%%%%%%%%%%%%%%%%%%%%%%%
%%%%%%%%%%%%%%%%%%%%%%%%%%%%%%%%%%%%%%%%%%%%%%%%%%%%%%%%%%%%%%%
\section{Interpolation on a Torus}\label{sec:torus}
%%%%%%%%%%%%%%%%%%%%%%%%%%%%%%%%%%%%%%%%%%%%%%%%%%%%%%%%%%%%%%%%
%%%%%%%%%%%%%%%%%%%%%%%%%%%%%%%%%%%%%%%%%%%%%%%%%%%%%%%%%%%%%%%%

Throughout this section let $\TT\coloneqq \Spec(\KK[u^{\pm1}, v^{\pm1}])$ be a $2$-dimensional torus and write $\one = (1,1) \in \TT$. For a finite set $\sS\subset \ZZ^2$ let
\[
W_{\sS}\coloneqq \Span_{\KK}\left\langle u^{i}v^{k} \;\; \big| \;\; (i,k) \in \sS \right\rangle \subset \KK[u^{\pm1}, v^{\pm1}].
\]
The goal of this section is to give concrete combinatorial conditions on $\sS$ that guarantee that no nonzero element of $W_{\sS}$ vanishes to prescribed orders at general points of the line $\ell=\{u=1\} \subset \TT$. We obtain these conditions by letting the points slide along the line $\ell$ until they collide at $\one$. The resulting zero-dimensional subscheme is described by an order ideal, which turns the vanishing problem into interpolation on the exponent set $\sS$. 

%%%%%%%%%%%%%%%%%%%%%%%%%%%%%%%%%%%%%%%%%%%%%%%%%%%%%%%%%%%%%%%%
\subsection{Interpolation at the Identity}
%%%%%%%%%%%%%%%%%%%%%%%%%%%%%%%%%%%%%%%%%%%%%%%%%%%%%%%%%%%%%%%%

For $(p,q) \in \ZZ^{2}_{\geq0}$, let $\partial^{(p,q)}$  be the mixed partial derivative of order $(p,q)$ and $\partial^{(p,q)}_{\one}$ be the evaluation at $\one$:
\[
\partial^{(p,q)}f \coloneqq \frac{\partial^{p+q}}{\partial u^{p}\partial v^{q}}f \quad \quad \text{and}\quad \quad \partial^{(p,q)}_{\one}f \coloneqq \frac{\partial^{p+q}}{\partial u^{p}\partial v^{q}}f\bigg|_{\one}.
\]
Notice for a Laurent monomial $u^{a}v^{b}\in \KK[u^{\pm1},v^{\pm1}]$ we have:
\[
\partial^{(p,q)}_{\one}\left(u^{a}v^{b}\right)=a(a-1)\cdots(a-p+1)\cdot b(b-1)\cdots(b-q+1)=p!q!\binom{a}{p}\binom{b}{q}.
\]
Here we let $\binom{a}{p}=a(a-1)\cdots(a-p+1)/p!$, and similarly for $\binom{b}{q}$, with the empty product being equal to $1$, so that the above formula applies when $a$ or $b$ is negative. There is a ring homomorphism $\Tay:\KK[u^{\pm1},v^{\pm1}] \to \KK[[x,y]]$ sending a Laurent polynomial to its Taylor series expansion at $\one$:
\[
\Tay(f) \coloneqq \sum_{(p,q)\in\ZZ^2_{\geq0}} \frac{1}{p!q!}\partial^{(p,q)}_{\one}(f)x^{p}y^q.
\]
Equivalently, $\Tay$ is the unique $\KK$-algebra homomorphism $\KK[u^{\pm1},v^{\pm1}] \to \KK[[x,y]]$ with $u\mapsto x+1$ and $v\mapsto y+1$. It is well-defined since $x+1$ and $y+1$ are units in $\KK[[x,y]]$. 

Give $\ZZ^{2}_{\geq0}$ the partial order where $(p,q)\leq (p',q')$ if and only if $p \leq p'$ and $q\leq q'$. An \emph{order ideal}  is a finite subset $\sE \subset \ZZ^{2}_{\geq0}$ such that if $(p',q')\in \sE$ and $(p,q)\leq (p',q')$ then $(p,q)\in \sE$. Pictorially, an order ideal $\sE \subset \ZZ^2_{\geq0}$ is a staircase of lattice points, or equivalently a Young diagram (in French notation) with $\#\sE$ boxes. If $\sE\subset \ZZ^{2}_{\geq0}$ is an order ideal we let 
   \[
    J_{\sE} \coloneqq \left\{f \in \KK[u^{\pm1},v^{\pm1}] \;\; \mid\;\; \partial^{(p,q)}_{\one}f=0 \text{ for all $(p,q)\in \sE$}\right\}.
    \]
Thus, $J_{\sE}$ is the set of Laurent polynomials whose Taylor coefficients at $\one$ indexed by $\sE$ all vanish. From  $\sE$ we also define a monomial ideal 
\[
\mfm_{\sE} \coloneqq \left\langle x^{p}y^{q} \;\; \big| \;\; (p,q) \in \ZZ^{2}_{\geq0}\setminus \sE \right\rangle \subset \KK[[x,y]],
\]
which is generated by the monomials not indexed by $\sE$. 

\begin{lemma}\label{lem:jets}
    If $\sE\subset \ZZ^{2}_{\geq0}$ is an order ideal then $J_{\sE}$ is an ideal. Further,  $J_{\sE}$ has the following properties:
    \begin{enumerate}
        \item The ideal $J_{\sE}$ is the preimage of $\mfm_{\sE}$ via $\Tay$.
        \item The ring $\KK[u^{\pm1},v^{\pm1}]/J_{\sE}$ is Artinian with $\{(u-1)^p(v-1)^q \; |\; (p,q)\in \sE \}$ being a $\KK$-basis for the quotient. Hence
        \[
        \dim_{\KK} \KK[u^{\pm1},v^{\pm1}]/J_{\sE} = \#\sE.
        \]
    \end{enumerate}
\end{lemma}

\begin{proof}
	The identity $J_{\sE}=\Tay^{-1}(\mfm_\sE)$ is immediate from the definition of $\Tay$, and this implies $J_{\sE}$ is an ideal. Turning to part (2) the map $\Tay$ descends to an injection $\KK[u^{\pm1},v^{\pm1}]/J_{\sE}\hookrightarrow\KK[[x,y]]/\mfm_{\sE}$. Since $\mfm_{\sE}$ is a monomial ideal generated by monomials not in $\sE$ the quotient $\KK[[x,y]]/\mfm_{\sE}$ has basis $x^py^q$ for $(p,q)\in \sE$. The preimages of these elements form a basis of $\KK[u^{\pm1},v^{\pm1}]/J_{\sE}$, and the preimage of $x^py^q$ is $(u-1)^p(v-1)^q$ as claimed. 
\end{proof}

The operators $\theta_u \coloneqq u\tfrac{\partial}{\partial u}$ and $\theta_v \coloneqq v\tfrac{\partial}{\partial v}$ act diagonally on Laurent monomials:
\[
\theta_{u}\left(u^{j}v^{k}\right) = j\cdot u^{j}v^{k} \quad \quad \text{and} \quad \quad  \theta_{v}\left(u^{j}v^{k}\right) = k\cdot u^{j}v^{k}. 
\]
Iteratively applying this gives the following version of Euler's identity.

\begin{lemma}\label{lem:euler}
Fix $p,q\in \ZZ_{\geq0}$. If $u^{j}v^{k}$ and $f$ are elements of $\KK[u^{\pm1},v^{\pm1}]$ then
\[
\theta^{p}_{u}\theta^q_{v}(u^{j}v^{k}) = j^pk^q(u^{j}v^{k}) \quad \quad \text{and} \quad \quad \left(\theta^{p}_{u}\theta^q_{v}f\right)(\one) = 
\sum_{i=0}^{p}\sum_{l=0}^{q}  \genfrac{\{}{\}}{0pt}{}{p}{i}\genfrac{\{}{\}}{0pt}{}{q}{l}\partial^{(i,l)}_{\one}f,
\]
where $\genfrac{\{}{\}}{0pt}{}{p}{i}$ denotes a Stirling number of the second kind, with $\genfrac{\{}{\}}{0pt}{}{0}{0}=1$. 
\end{lemma}

\begin{proof}
    The displayed identity on the left-hand side is a direct computation by iterating the formulas in the previous paragraph. Turning towards the second identity, by induction and the identity
    \[
    \genfrac{\{}{\}}{0pt}{}{p+1}{i} = i\genfrac{\{}{\}}{0pt}{}{p}{i} + \genfrac{\{}{\}}{0pt}{}{p}{i-1}
    \]
     we have that
    \[
    \theta_u^{p}=\sum_{i=0}^{p} \genfrac{\{}{\}}{0pt}{}{p}{i} u^i\partial_{u}^i.
    \]
    Using this and the analogous identity for $\theta_v^{q}$ gives the claim. 
\end{proof}

By Lemma~\ref{lem:jets} the space of  \emph{$\sE$-jets at $\one$} is $\KK[u^{\pm1},v^{\pm1}]/J_{\sE}$ and the polynomials $(u-1)^p(v-1)^q$, for $(p,q)\in \sE$, form a basis for this space. Its dual consists of linear functions on $\KK[u^{\pm1},v^{\pm1}]$ that vanish on $J_{\sE}$. This is also equal to the Macaulay inverse system of $J_{\sE}$ \cite{Macaulay94}; see also \cite[\S~21.2]{eisenbud95}. The use of Macaulay inverse systems in studying fat points has seen significant attention including \cite{emsalemIarrobino95,geramita95}. This space has two useful bases:
\begin{enumerate}
\item The derivative basis $f\mapsto \partial_{\one}^{(p,q)}(f)$. Dividing the $(p,q)$-th such map by $p!q!$ gives the dual basis to $(u-1)^p(v-1)^q$. 
\item The Euler basis $f\mapsto (\theta_{u}^p\theta_v^q)(f)(\one)$.
\end{enumerate}
By Lemma~\ref{lem:euler}, each Euler functional is a linear combination of the derivative functionals whose indices are coordinatewise less than or equal to $(p,q)$. Since $\sE$ is an order ideal all of the indices in such a sum lie in $\sE$. Ordering $\sE$ by an extension of the coordinatewise order gives a triangular change-of-basis matrix whose diagonal entries are $\genfrac{\{}{\}}{0pt}{}{p}{p}\genfrac{\{}{\}}{0pt}{}{q}{q}=1$.

The derivative basis records vanishing of jets, while the Euler basis is convenient for working with Laurent monomials. In the monomial basis of $W_{\sS}$ and the basis for $\KK[u^{\pm1},v^{\pm1}]/J_{\sE}$ dual to the Euler basis, the jet map has rows indexed by $(p,q)\in \sE$ and columns indexed by $(j,k)\in \sS$ with entries $j^pk^q$. Corollary~\ref{cor:order-ideals} uses this matrix to turn vanishing of jets into interpolation on the exponent set. This duality has appeared in several places, including \cite[Prop.~13]{dumnicki07}, \cite[\S~8]{castravet18}, and \cite[\S~2]{gonzalezKaru16}.

\begin{lemma}\label{lem:torus-ideal-euler}
    If $\sE\subset \ZZ^{2}_{\geq0}$ is an order ideal then $f \in J_{\sE}$ if and only if $(\theta^{p}_{u}\theta^q_{v}f)(\one)=0$ for all $(p,q) \in \sE$. 
\end{lemma}

\begin{proof}
	By the discussion above the matrix relating the derivative basis and Euler basis is triangular with ones on the diagonal. Thus, it is invertible, meaning that all of the Euler evaluations vanish if and only if all of the derivative evaluations vanish. By Lemma~\ref{lem:jets} this is equivalent to $f\in J_{\sE}$. 
\end{proof}

Let $\sS \subset \ZZ^{2}$ be a finite subset. Let $\KK^{\langle \sS\rangle}$ be the $\KK$-vector space of all functions $\sS\to \KK$.  We are interested in knowing when $W_{\sS}\cap J_{\sE} =0$. For $(p,q)\in \ZZ^{2}_{\geq0}$ let $\se_{(p,q)}:\sS\to \KK$ denote the function $(a,b)\mapsto a^pb^q$, i.e., the monomial $x^py^q$ evaluated at the points $\sS\subset \ZZ^2\subset\KK^2$. We consider $\se_{(p,q)}$ as a linear functional $W_{\sS}\to \KK$ by extending linearly. By Lemma~\ref{lem:euler} this is the functional $f\mapsto (\theta^{p}_{u}\theta_{v}^{q}f)(\one)$.

\begin{corollary}\label{cor:order-ideals}
    Let $\sS \subset \ZZ^2$ be finite and let $\sE \subset \ZZ^{2}_{\geq0}$ be an order ideal. 
    \begin{enumerate}
        \item An element $f\in W_{\sS}$ lies in $J_{\sE}$ if and only if $\se_{(p,q)}(f) = 0$ for every $(p,q)\in \sE$.
        \item The intersection $W_{\sS}\cap J_{\sE}=0$ if and only if $\{\se_{(p,q)} | (p,q)\in \sE\}$ spans $\KK^{\langle \sS\rangle}$. 
    \end{enumerate}
\end{corollary}

\begin{proof}
    Part (1) follows immediately from Lemma~\ref{lem:torus-ideal-euler}. For part (2) notice that by part (1) the kernel of the linear map
    \[
    \begin{tikzcd}[row sep = .2em, column sep = 3em]
        W_{\sS} \arrow[r] & \KK^{\langle \sE\rangle}\\
        f \arrow[r,mapsto] &\left(\se_{(p,q)}(f)\right)_{(p,q)\in \sE}
    \end{tikzcd}
    \]
    is $W_{\sS}\cap J_{\sE}$. In the monomial basis, the matrix of this map has rows indexed by $\sE$ and columns indexed by $\sS$; its entries are of the form $j^pk^q$ for $(j,k)\in \sS$ and $(p,q)\in\sE$. So $W_{\sS}\cap J_{\sE}=0$ if and only if its row span is $\KK^{\langle\sS\rangle}$.
\end{proof}

In other words, regarding $\sS$ as a finite set of points in $\KK^2$, we have that $W_{\sS}\cap J_{\sE}=0$  exactly when every function $\sS\to \KK$ is the restriction of a polynomial whose monomials $x^py^q$ are indexed by $(p,q)\in \sE$. This turns vanishing at the fat point defined by $J_{\sE}$ into a polynomial interpolation question on the exponent set $\sS$. 

%%%%%%%%%%%%%%%%%%%%%%%%%%%%%%%%%%%%%%%%%%%%%%%%%%%%%%%%%%%%%%%%
\subsection{Colliding Collinear Fat Points}
%%%%%%%%%%%%%%%%%%%%%%%%%%%%%%%%%%%%%%%%%%%%%%%%%%%%%%%%%%%%%%%%

Fix an integer $n\geq1$ and multiplicities $\mm=(m_1,\ldots,m_n) \in \ZZ^n_{\geq0}$. We now wish to consider fat points in $\TT$ consisting of $n$ distinct points lying on the line $\ell=\VV(u-1)\subset \TT$ with multiplicities $\mm$. Precisely, if $\overline{\xi}=(\xi_1,\ldots,\xi_n)\in \Conf_{n}(\KK^{\times})$ consider the points $(1,\xi_i)\in \ell$. Let $Z_{\overline{\xi}}(\mm)\subset \TT$ be the subscheme defined by
\[
I_{Z_{\overline{\xi}}(\mm)}= \bigcap_{i=1}^{n}\langle u-1,v-\xi_i\rangle^{m_i} \subset \KK[u^{\pm1},v^{\pm1}].
\]
Notice $f\in I_{Z_{\overline{\xi}}(\mm)}$ if and only if $f$ vanishes to order at least $m_i$ at the point $(1,\xi_i)$ for every $i$. We will slide these points along the line $\ell$ until they collide at $\one$, keeping track of all the vanishing conditions on the resulting limiting subscheme.

We first translate the sequence of multiplicities into an order ideal, which encodes which partial derivatives must be checked at the collided fat point. For an integer $k\geq0$ let
\[
c_{k}\coloneqq \sum_{i=1}^{n}(m_i-k)_{+}.
\]
By construction the sequence of $c_k$'s is weakly decreasing: $c_{0}\geq c_{1} \geq \cdots$. Moreover, if $k\geq M\coloneqq \max\{m_i\}$ then $c_{k}=0$. Hence we may define an order ideal $\sE(\mm)$ from the $c_{k}$'s as follows:
\[
\sE(\mm) \coloneqq \left\{ (p,q) \in \ZZ^{2}_{\geq0} \;\; \big| \;\; q < c_{p} \right\} \subset \ZZ^2_{\geq0}.
\]
The $k$-th column of $\sE(\mm)$ has height $c_{k}$ and so 
\[
\#\sE(\mm)=\sum_{p\geq0}c_p = \sum_{i=1}^{n}\sum_{p=0}^{m_i-1}(m_i-p) = \sum_{i=1}^{n}\binom{m_i+1}{2}.
\]
One particularly important case is that of equal multiplicities, when $m=m_1=m_2=\cdots=m_n$ so that $\mm=m^n$. In this case 
\[
\sE(m^n)=\left\{(p,q)\in \ZZ^{2}_{\geq0}\;\; \big|\;\; q < n(m-p)_{+}\right\} \subset \ZZ^2_{\geq0}.
\]

There is a ring map $\KK[u^{\pm1},v^{\pm1}]\to \KK[v^{\pm1}][[u-1]]$ sending  $f$ to its Taylor expansion normal to the line $\ell$:
\[
f \longmapsto \sum_{p\geq0}a_p(f)(u-1)^{p} \quad \quad \text{where}\quad \quad a_p(f)= \frac{1}{p!}\frac{\partial^{p}}{\partial u^p}f\bigg|_{u=1} \in \KK[v^{\pm1}].
\]
For $\overline{\xi} \in (\KK^{\times})^n$ and an integer $0\leq p \leq M$ set 
\[
Q_{\overline{\xi},p} \coloneqq \prod_{i=1}^{n} \left(v-\xi_{i}\right)^{(m_i-p)_{+}} \in \KK[v].
\]
Each $Q_{\overline{\xi},p}$ vanishes at the points $\xi_1,\ldots,\xi_n$ with multiplicities at least $(m_1-p)_{+},\ldots,(m_n-p)_{+}$, with equality when the points are distinct. We have that $Q_{\overline{\xi},M}=1$. Define
\[
J_{\overline{\xi}} \coloneqq \left\langle \left(u-1\right)^pQ_{\overline{\xi},p} \;\; \big|\;\; 0\leq p \leq M\right\rangle.
\]
Since $(u-1)^M\in J_{\overline{\xi}}$ membership in $J_{\overline{\xi}}$ is determined by the first $M$ Taylor coefficients. Since $Q_{\overline{\xi},p}$ divides $Q_{\overline{\xi},k}$ for all $k\leq p$ this shows that 
\[
J_{\overline{\xi}}  = \left\{f \in \KK[u^{\pm1},v^{\pm1}] \;\; \bigg| \;\; Q_{\overline{\xi},p} \big| a_p(f) \text{   for $0\leq p < M$} \right\}.
\]
The key point is that the generators of $J_{\overline{\xi}}$ make sense even when some of the $\xi_i$ coincide. Lemma~\ref{lem:coliding-colinear} shows that letting $\overline{\xi}$ vary gives a flat family of subschemes of $\TT$ whose fiber at distinct $\xi_i$ is the fat point scheme $Z_{\overline{\xi}}(\mm)$, and whose fiber at $\xi_1=\cdots=\xi_n=1$ is the scheme defined by $J_{\sE(\mm)}$. This is an iterated instance of the staircase-addition construction in \cite[Prop.~5]{evain98}.
 
\begin{lemma}\label{lem:coliding-colinear}
    Fix $(m_1,\ldots,m_n)\in \ZZ^{n}_{\geq0}$. 
    \begin{enumerate}
        \item Treat $\xi_1,\ldots,\xi_n$ as variables and set $T=\KK[\xi_{1}^{\pm1},\ldots,\xi_{n}^{\pm1}]$ so that $J_{\overline{\xi}}\subset T[u^{\pm1},v^{\pm1}]$. The quotient $T[u^{\pm1},v^{\pm1}]/J_{\overline{\xi}}$ is a free $T$-module with basis $(u-1)^pv^q$ for $(p,q)\in \sE(\mm)$. 
        \item If $\xi_1,\ldots,\xi_n\in \KK^\times$ are distinct then $J_{\overline{\xi}}= I_{Z_{\overline{\xi}}(\mm)}$. 
        \item If $\xi_i=1$ for all $i$ then $J_{\overline{\xi}}=J_{\sE(\mm)}$.
        \item Fix a finite subset $\sS\subset \ZZ^2$. The set of points $\overline{\xi}\in (\KK^\times)^n$ where $W_{\sS}\cap J_{\overline{\xi}}=0$ is a Zariski open subset of $(\KK^{\times})^n$.
    \end{enumerate}
\end{lemma}

\begin{proof}
	When $M=0$ the four claims are immediate, and so we assume $M\geq1$.  Towards (1) notice that $u^{-1}$ is a polynomial in $u-1$ modulo $(u-1)^{M}$. Thus, every Laurent polynomial $f\in T[u^{\pm1},v^{\pm1}]$ has a unique expansion of the form 
	\[
	f\equiv \sum_{p < M} a_p(f)(u-1)^p \pmod{(u-1)^M}
	\]
	for $a_p(f) \in T[v^{\pm1}]$. In $T[u^{\pm1},v^{\pm1}]/\langle (u-1)^M\rangle$ the image of $J_{\overline{\xi}}$ is homogeneous in $u-1$. Since $Q_{\overline{\xi},p}$ divides $Q_{\overline{\xi},k}$ for all $k\leq p$, the degree $p$ part of this ideal is generated by the class of $(u-1)^pQ_{\overline{\xi},p}$. Therefore, as $T$-modules
	\[
	T[u^{\pm1},v^{\pm1}]/J_{\overline{\xi}} \cong \bigoplus_{p=0}^{M-1} T[v^{\pm1}]/\langle Q_{\overline{\xi},p}\rangle \cdot (u-1)^p.
	\]
	Thus, for each fixed $p$ with $0\leq p < M$, it is enough to show that the corresponding summand is a free $T$-module with basis $(u-1)^pv^q$ for $0\leq q < c_p$. Since the constant term of $Q_{\overline{\xi},p}$ is a unit in $T$ there is an isomorphism $T[v^{\pm1}]/\langle Q_{\overline{\xi},p} \rangle \cong T[v]/\langle Q_{\overline{\xi},p} \rangle$. Further, since $Q_{\overline{\xi},p}$ is monic, the latter is a free $T$-module with basis $1,v,\ldots,v^{c_{p}-1}$. 
	
	Turning our attention towards part (2), a Laurent polynomial $f\in \KK[u^{\pm1},v^{\pm1}]$ vanishes to order at least $m_i$ at $(1,\xi_i)$ if and only if $f \in \langle u-1, v-\xi_i \rangle^{m_i}$. This occurs if and only if $a_p(f)$ vanishes to order at least $(m_i-p)_{+}$ at $\xi_i$ for every $p\geq0$. For distinct $\xi_i$ these factors are coprime, and so this holds for all $i$ if and only if their product, which is $Q_{\overline{\xi},p}$, divides $a_p(f)$ for $0 \leq p < M$. This is our second description of $J_{\overline{\xi}}$ given above. 
	
	For (3) when $\xi_i = 1$ for all $i$ we have that $Q_{\overline{\xi},p}=(v-1)^{c_p}$. Thus, the generators of $J_{\overline{\xi}}$ are $(u-1)^p(v-1)^{c_p}$ for $0\leq p <M$ and $(u-1)^M$. By the Taylor coefficient description in Lemma~\ref{lem:jets} this is exactly the ideal $J_{\sE(\mm)}$.
	
	Finally for (4) consider the map $W_{\sS}\otimes_{\KK} T \to T[u^{\pm1},v^{\pm1}]/J_{\overline{\xi}}$ given by composing the inclusion $W_{\sS}\otimes_{\KK}T\hookrightarrow T[u^{\pm1},v^{\pm1}]$ with the quotient map. By (1) this is a map of $T$-free modules. After specializing at a point $\overline{\xi}$ the kernel of this map is $W_{\sS}\cap J_{\overline{\xi}}$. The claim follows because injectivity on fibers is an open condition for a morphism of free modules. 
\end{proof}

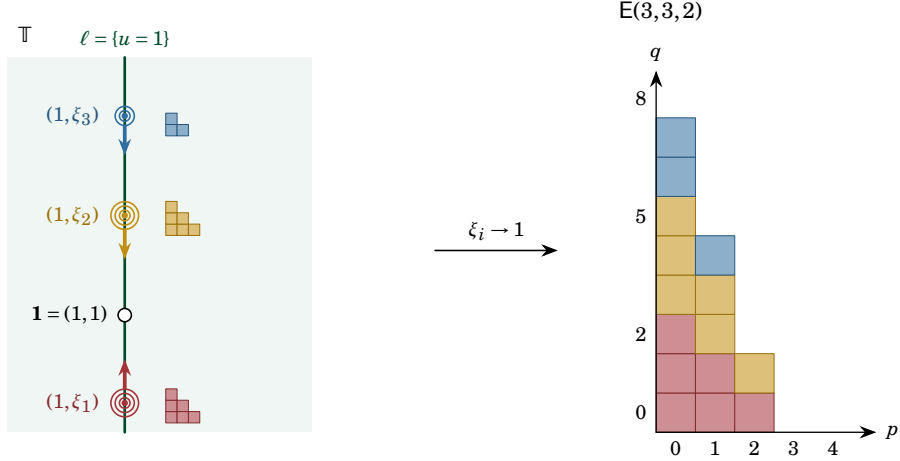
\begin{figure}
\centering
\begin{tikzpicture}[fig, x=1.3cm, y=1.3cm]
  \colorlet{P1}{FigRed}\colorlet{P2}{FigGold}\colorlet{P3}{FigBlue}
  %------------------------------------------------ (a) two triple points and a double point on the line u = 1
  \begin{scope}[shift={(0,0)}, x=1.55cm, y=1.55cm]
    \fill[DartmouthGreen!6] (0,0) rectangle (2.6,3.2);
    \draw[DartmouthGreen!70!black, line width=1.0pt] (1,0) -- (1,3.2);
    \node[tinylbl, DartmouthGreen!70!black, anchor=south] at (1,3.2) {$\ell=\{u=1\}$};
    % the identity of the torus
    \node[circle, draw=black, line width=0.6pt, fill=white, inner sep=0pt, minimum size=5pt] at (1,1) {};
    \node[tinylbl, anchor=east] at (0.93,1) {$\one=(1,1)$};
    % fat points (height / colour / index / multiplicity), each with its own staircase E(m_i)
\foreach \y/\c/\i/\m in {0.25/P1/1/3,1.85/P2/2/3,2.7/P3/3/2} {
  \foreach \k in {1,...,\m} {
    \pgfmathsetmacro{\rr}{0.4+1.6*\k}
    \draw[\c, line width=0.6pt] (1,\y) circle (\rr pt);
  }
  \fill[\c] (1,\y) circle (1.2pt);
  \node[tinylbl, anchor=east, \c!80!black]
    at (0.86,\y) {$(1,\xi_{\i})$};

  % Start at the outer circle's edge and point toward (1,1).
  \pgfmathsetmacro{\arrowoffset}{(\y<1 ? 1 : -1)*(0.7+1.6*\m)}
\draw[->, \c, line width=1.4pt]
  ([yshift=\arrowoffset pt]1,\y)
  -- ++(0,{(\y<1 ? 1 : -1)*0.25});

  % Staircase E(m): columns of heights m, m-1, ..., 1.
  \pgfmathtruncatemacro{\mm}{\m-1}
  \begin{scope}[shift={(1.35,{\y-0.17})}, x=0.15cm, y=0.15cm]
    \foreach \s in {0,...,\mm} {
      \pgfmathtruncatemacro{\h}{\m-\s}
      \foreach \t in {1,...,\h}
        \filldraw[fill=\c!55, draw=\c!80!black, line width=0.2pt]
          (\s,\t-1) rectangle (\s+1,\t);
    }
  \end{scope}
}
    \node[lbl] at (0.15,3.4) {$\TT$};
  \end{scope}
  \draw[->, line width=0.6pt] (4.35,1.85) -- (5.6,1.85) node[midway, above, tinylbl, align=center] {$\xi_i\to 1$};
  %------------------------------------------------ (b) the collided order ideal E(3,3,2)
  \begin{scope}[shift={(6.6,0)}, x=0.52cm, y=0.52cm]
    % column s contains (m_i - s)_+ boxes from the i-th point, stacked in the order i = 1, 2, 3
    % (column / first t / number of boxes / colour)
    \foreach \s/\tz/\h/\c in {0/0/3/P1, 0/3/3/P2, 0/6/2/P3,
                              1/0/2/P1, 1/2/2/P2, 1/4/1/P3,
                              2/0/1/P1, 2/1/1/P2}
      \foreach \k in {1,...,\h}
        \filldraw[fill=\c!55, draw=\c!80!black, line width=0.3pt] (\s,\tz+\k-1) rectangle (\s+1,\tz+\k);
    \draw[->, line width=0.5pt] (0,0) -- (5.6,0) node[right, tinylbl] {$p$};
    \draw[->, line width=0.5pt] (0,0) -- (0,9.2) node[above, tinylbl] {$q$};
    \foreach \s in {0,...,4} \node[tinylbl, anchor=north] at (\s+0.5,0) {$\s$};
    \foreach \t in {0,2,5,8} \node[tinylbl, anchor=east] at (0,\t+0.5) {$\t$};
    \node[lbl, anchor=south west] at (-1.2,10.2) {$\sE(3,3,2)$};
  \end{scope}
\end{tikzpicture}
\caption{Two triple points and one double point on the line $\ell$, each drawn with its staircase $\sE(m_i)$ recording the partial derivatives defining it, colliding at $\one$. The limit subscheme is described by the order ideal $\sE(3,3,2)$, whose columns have heights $c_0=8$, $c_1=5$, $c_2=2$; see Example~\ref{exm:collision}.}
\label{fig:collide}
\end{figure}

\begin{example}\label{exm:collision}
	Consider the situation of Lemma~\ref{lem:coliding-colinear} when $\mm=(3,3,2)$ so that two triple points and one double point on the line $\ell$ collide at $\one$. (See Figure~\ref{fig:collide}.) Before the collision their ideal is
	\[
	J_{\overline{\xi}} = \left\langle (v-\xi_1)^{3}(v-\xi_2)^3(v-\xi_3)^2, \; (u-1)(v-\xi_1)^{2}(v-\xi_2)^2(v-\xi_3), \; (u-1)^2(v-\xi_1)(v-\xi_2), \; (u-1)^{3}\right\rangle.
	\]
	This defines a length $6+6+3=15$ subscheme. Note the double point imposes no conditions in the $p=2$ column and so $Q_{\overline{\xi},2}=(v-\xi_1)(v-\xi_2)$ does not involve $\xi_3$. 
	
	After the collision we obtain the subscheme defined by $J_{\sE(\mm)}$. The column heights of $\sE(3,3,2)$ are $c_{p}=2(3-p)_{+}+(2-p)_{+}$, namely $c_0=8$, $c_1=5$, $c_2=2$, and zero otherwise. Thus, as $\xi_1,\xi_2,\xi_3\to 1$ the flat limit is defined by
	\[
	J_{\sE(\mm)} =\left\langle (v-1)^8, \; (u-1)(v-1)^5, \; (u-1)^2(v-1)^2, \; (u-1)^3 \right\rangle,
	\]
	which defines a subscheme of length $8+5+2=15$, the same as the length of the original union of two triple points and a double point. 
\end{example}

%%%%%%%%%%%%%%%%%%%%%%%%%%%%%%%%%%%%%%%%%%%%%%%%%%%%%%%%%%%%%%%%
\subsection{Interpolation by Rows}
%%%%%%%%%%%%%%%%%%%%%%%%%%%%%%%%%%%%%%%%%%%%%%%%%%%%%%%%%%%%%%%%

For a finite set $\sS \subset \ZZ^2$  and  $k\in \ZZ$ the \emph{$k$-th row of $\sS$} is  $R_{k}(\sS)\coloneqq \{j \in \ZZ \; | \; (j,k) \in \sS\}$. The \emph{discrete row profile} of $\sS$ is the function
\[
N_{\sS}(p) \coloneqq \#\left\{ k\in \ZZ \;\; \big| \:\#R_{k}(\sS)> p\right\} \quad \quad p\in \ZZ_{\geq0}.
\]
Thus, $N_{\sS}(p)$ counts the rows of $\sS$ containing at least $p+1$ points. 

\begin{definition}
    The \emph{lex staircase} of a finite set $\sS\subset \ZZ^2$ is
    \[
\sL(\sS) \coloneqq \left\{ (p,q) \in \ZZ^{2}_{\geq0} \;\; \big| \;\; q<N_{\sS}(p)\right\}.
\]
\end{definition}

The $p$-th column of $\sL(\sS)$ has one lattice point for each row of $\sS$ containing at least $p+1$ points. Since $N_{\sS}(p)$ is weakly decreasing in $p$ the set $\sL(\sS)$ is an order ideal and 
\[
\#\sL(\sS)=\sum_{p\geq0}N_{\sS}(p)=\#\sS.
\]
Like $\sE(\mm)$ this staircase indexes partial derivatives, rather than the exponents in the set $\sS$. For instance if $\sS$ has three nonempty rows containing $3$, $2$, and $2$ points, then $N_{\sS}(0)=3$, $N_{\sS}(1)=3$, and $N_{\sS}(2)=1$, so $\sL(\sS)$ has columns of height $3,3,1$. 

Lemma~\ref{lem:key-interpolation} says that the monomials $x^py^q$ for $(p,q)\in \sL(\sS)$ form an interpolation basis for $\sS \subset \KK^2$. In fact, these are the standard monomials for the vanishing ideal of $\sS$ with respect to the lexicographic order with $x>y$. The lemma is a special case of \cite{cerliencoMureddu95}, and can also be deduced from \cite{felszeghyRathRonya06}. 

\begin{lemma}\label{lem:key-interpolation}
    Let $\sS\subset \ZZ^2$ be a finite set. The functions $\se_{(p,q)}$ indexed by $(p,q)\in \sL(\sS)$ form a basis for $\KK^{\langle\sS\rangle}$.
\end{lemma}

\begin{proof}
	Since $\#\sL(\sS) = \#\sS$, it suffices to prove that $\{\se_{(p,q)} \; | \; (p,q)\in \sL(\sS)\}$ spans $\KK^{\langle\sS\rangle}$. Let the nonempty rows of $\sS$ be at heights $k_1,\ldots,k_a$ and order these so that their lengths satisfy $\lambda_1\geq \cdots\geq \lambda_a$.  Now let $g:\sS \to \KK$ be an arbitrary function. We will interpolate $g$ one row at a time by induction.
	
	Starting with the first row at height $k_1$, standard polynomial interpolation gives a polynomial $g_1\in \KK[t]$ with $\deg(g_1)< \lambda_1$ such that $g_1(j) = g(j,k_1)$ for all $(j,k_1) \in \sS$. Since $\se_{(p,0)}(j,k)=j^p$ we can clearly write the function $(j,k)\mapsto g_1(j)$ as a $\KK$-linear combination of $\se_{(p,0)}$ for $0\leq p < \lambda_1$. For each such $p$ the row at height $k_1$ contains at least $p+1$ points. Thus, $N_{\sS}(p) \geq 1$ and so $(p,0) \in \sL(\sS)$. 
	
	Now assume that there exists a function $f:\sS\to \KK$ which is a $\KK$-linear combination of the $\se_{(p,q)}$ for $(p,q)\in \sL(\sS)$ such that $f$ agrees with $g$ on the rows at heights $k_1,\ldots,k_{i-1}$. Again using standard one-variable polynomial interpolation there exists a polynomial $g_i\in \KK[t]$ with $\deg(g_i) < \lambda_i$ such that
	\[
	g_{i}(j) = \frac{g(j, k_i) - f(j,k_i)}{\displaystyle \prod_{i'=1}^{i-1} \left(k_i - k_{i'}\right)}
	\]
	for all $(j,k_i)\in \sS$. (Note the denominator is nonzero since the heights are distinct.) Consider the function
	\[
	(j,k) \longmapsto \left(\prod_{i'=1}^{i-1}\left(k-k_{i'}\right)\right) g_{i}(j).
	\]
	Notice if $k=k_{i'}$ for $i'=1,2,\ldots,i-1$ then this function is zero and at $(j,k_i)$ its value is exactly $g(j,k_i) - f(j,k_i)$. Thus, adding it to $f$ we get a function that agrees with $g$ on the rows $k_1,\ldots,k_i$. It remains to check that the correction function above can be written as a linear combination of the $\se_{(p,q)}$ for $(p,q)\in \sL(\sS)$. Expanding the products we see it is a sum of terms of the form $j^pk^q$ with $0\leq p<\lambda_i$ and $q<i$. Since the first $i$ rows each contain at least $\lambda_i \geq p+1$ points we have that $N_{\sS}(p)\geq i >q$ as needed. 
\end{proof}

\begin{theorem}\label{thm:torus-interoplation}
    Let $\sS\subset \ZZ^2$ be a finite set, and let $\mm=(m_1,\ldots,m_n)\in \ZZ^{n}_{\geq0}$. If $\sL(\sS) \subset \sE(\mm)$ then $W_{\sS}\cap I_{Z_{\overline{\xi}}(\mm)}=0$ for a general point $\overline{\xi}\in \Conf_{n}(\KK^{\times})$, i.e., for general distinct points $(1,\xi_i)$ of $\ell$. 
\end{theorem}

\begin{proof}
	By part (4) of Lemma~\ref{lem:coliding-colinear}  the set 
	\[
	U \coloneqq \left\{ \overline{\xi}\in (\KK^\times)^n \;\; \big| \;\;  W_{\sS}\cap J_{\overline{\xi}}=0\right\}
	\]
	is a Zariski open subset of $(\KK^{\times})^n$. We first show that $U$ is nonempty.  Lemma~\ref{lem:key-interpolation} implies that the $\se_{(p,q)}$ for $(p,q)\in \sL(\sS)$ form a basis for $\KK^{\langle \sS\rangle}$. The containment  $\sL(\sS) \subset \sE(\mm)$  and Corollary~\ref{cor:order-ideals} give $W_{\sS}\cap J_{\sE(\mm)} =0$. By part (3) of Lemma~\ref{lem:coliding-colinear} this means $W_{\sS}\cap J_{\overline{\xi}}=0$ for $\overline{\xi}=(1,\ldots,1)$. Thus, $(1,\ldots,1) \in U$. 
	
	The configuration space $\Conf_{n}(\KK^\times)$ is a nonempty open subset of the irreducible variety $(\KK^\times)^n$. Hence $U\cap \Conf_{n}(\KK^\times)$ is nonempty and open. On this intersection part (2) of  Lemma~\ref{lem:coliding-colinear} shows $J_{\overline{\xi}}=I_{Z_{\overline{\xi}}(\mm)}$. Thus, $W_{\sS}\cap I_{Z_{\overline{\xi}}(\mm)}=0$ on $U\cap \Conf_{n}(\KK^\times)$.
\end{proof}

Since $\sL(\sS)$ and $\sE(\mm)$ are order ideals whose $p$-th columns have heights $N_{\sS}(p)$ and $c_p$, respectively, Theorem~\ref{thm:torus-interoplation} can be phrased entirely in terms of row counts. 

\begin{corollary}\label{cor:row-counts}
    Let $\sS\subset \ZZ^2$ be a finite set, and let $\mm=(m_1,\ldots,m_n)\in \ZZ^{n}_{\geq0}$. If, for every integer $p\geq 0$,
    \[
    N_{\sS}\left(p \right) \leq \sum_{i=1}^{n}(m_i-p)_{+},
    \]
    then $W_{\sS}\cap I_{Z_{\overline{\xi}}(\mm)}=0$ for a general point $\overline{\xi} \in \Conf_{n}(\KK^{\times})$. 
\end{corollary}

\begin{proof}
 	Since $\sL(\sS)$ and $\sE(\mm)$ are order ideals containment is determined by the heights of each column. The inequality is precisely the condition that implies $\sL(\sS)\subset \sE(\mm)$. Applying Theorem~\ref{thm:torus-interoplation} gives the result.
\end{proof}

%%%%%%%%%%%%%%%%%%%%%%%%%%%%%%%%%%%%%%%%%%%%%%%%%%%%%%%%%%%%%%%%
\subsection{Interpolation by Rows in a Polygon}
%%%%%%%%%%%%%%%%%%%%%%%%%%%%%%%%%%%%%%%%%%%%%%%%%%%%%%%%%%%%%%%%

In Section~\ref{sec:AVC-Bring-It-HOME} the exponent set $\sS$ will consist of the lattice points in a translate of $m\sP$ for a fixed, typically non-lattice, triangle $\sP$.  In this section we develop an analog of Corollary~\ref{cor:row-counts} using lengths of horizontal slices of $\sP$ instead of the row counts $N_{\sS}(p)$. Let $x,y$ be coordinates on $\RR^2$ in the usual way. The \emph{interval of heights} of a nonempty compact convex subset $\sP \subset \RR^2$ is the set 
	\[
	\sI_{\sP} \coloneqq \left\{ y \in \RR \;\; \mid \;\; \sP \cap (\RR\times\{y\}) \neq\varnothing\right\}.
	\]
Since $\sI_{\sP}$ is the image of $\sP$ under the projection $(x,y)\mapsto y$, and $\sP$ is compact and convex, $\sI_{\sP}$ is closed and convex. Thus, $\sI_{\sP}$ is a closed interval. For $y\in \RR$ define $\omega_{\sP}(y)$ to be the length of the horizontal slice $\sP\cap(\RR\times\{y\})$, with width zero for empty slices. Set $\omega_{\max}(\sP) \coloneqq \max\left\{\omega_{\sP}(y) \;\; \mid \;\; y \in \sI_{\sP} \right\}$. The \emph{continuous row profile} of $\sP$ is
	\[
	N_{\sP}(\sigma) \coloneqq \mathrm{length}\left\{ y \in \sI_{\sP} \; \; \mid \;\; \omega_{\sP}(y)\geq \sigma \right\} \quad \quad \sigma\geq0.
	\]
This is the continuous analog of the discrete row profile $N_{\sS}(p)$ we introduced for finite sets $\sS\subset \ZZ^2$. With our conventions $N_{\sP}(0)$ is equal to the length of $\sI_{\sP}$. When $\sP$ has positive area, the continuous row profile of $\sP$ was studied in \cite[Thm.~13]{gundlach21}. 

The following lemma keeps track of the rounding error when passing from continuous to discrete row profiles. In particular, the $\Gap$ term compensates for the fact that an interval of length $p$ can contain as many as $p+1$ integers. The condition that $m\Gap\geq1$ ensures this error is also absorbed after dilating the polygon by $m$. 

\begin{lemma}\label{lem:polygon-torus}
	Fix $m,n\in \ZZ_{\geq1}$ and $\Gap\in \RR_{>0}$. If  $\sP\subset \RR^2$ is a nonempty compact convex subset such that
		\begin{enumerate}
		\item $\omega_{\max}(\sP) < 1$,
		\item $N_{\sP}(\sigma)  \leq n(1-\sigma)-\Gap$ for all $0\leq \sigma \leq \omega_{\max}(\sP)$, and
		\item $m\Gap\geq1$,
	\end{enumerate}
	then for every $\vv \in \RR^2$ and every finite set $\sS \subset (\vv+m\sP)\cap \ZZ^2$, we have $\sL(\sS) \subset \sE(m^{n})$. 
\end{lemma}

\begin{proof}
	Fix an integer $p\geq0$ and let $\vv=(v_1,v_2)\in \RR^2$. A row of $\sS$ containing at least $p+1$ lattice points has horizontal width at least $p$. Therefore, the height $y$ of any row of $\sS$ counted by $N_{\sS}(p)$ is in the set
	\[
	v_2 + m\left\{y \in \sI_{\sP} \;\; \middle|\;\; \omega_{\sP}(y)\geq p/m \right\}.
	\] 
	Since $\sP$ is convex the width function $\omega_{\sP}$ is concave on $\sI_{\sP}$. Thus, the set after the $m$ in the displayed equation above is an interval or is empty. If $p/m>\omega_{\max}(\sP)$ then this set is empty and $N_{\sS}(p)=0$. We therefore restrict our attention to $0\leq p/m \leq \omega_{\max}(\sP)<1$. In this case the set is an interval of length $N_{\sP}(p/m)$. Since an interval of length $\lambda$ contains at most $\lambda+1$ lattice points we have that
	\[
	N_{\sS}(p) \leq mN_{\sP}(p/m)+1 \leq n(m-p)-m\Gap + 1 \leq n(m-p).
	\]
	For $p\geq m$ we already have that $N_{\sS}(p) = 0$, and so $N_{\sS}(p)\leq n(m-p)_{+}$ for every $p\geq0$. This is precisely the inequality  showing that $\sL(\sS) \subset \sE(m^{n})$.
\end{proof}

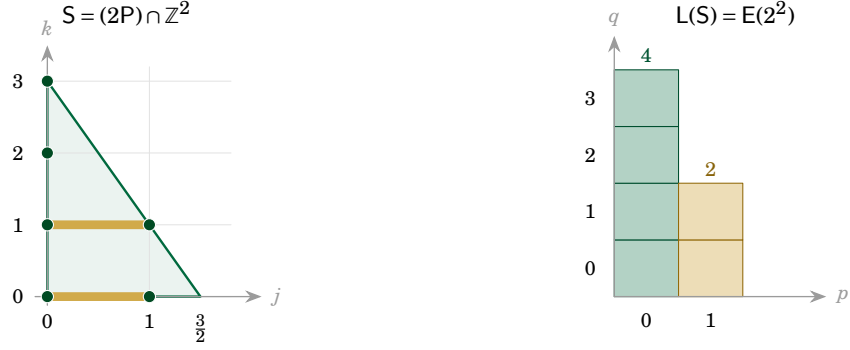
\begin{figure}
\centering
\begin{tikzpicture}[fig, x=1cm, y=1cm]
  \node[lbl] at (2.05,4.05) {$\sS=(2\sP)\cap\ZZ^2$};
  \node[lbl] at (10.1,4.05) {$\sL(\sS)=\sE(2^2)$};
  \begin{scope}[shift={(1.0,.3)},x=1.35cm,y=.95cm]
    \fill[DartmouthGreen!8] (0,0)--(1.5,0)--(0,3)--cycle;
    \foreach \k in {0,1,2,3}
      \draw[FigGray!22,line width=.3pt] (0,\k)--(1.8,\k);
    \foreach \j in {0,1}
      \draw[FigGray!22,line width=.3pt] (\j,0)--(\j,3.2);
    \draw[DartmouthGreen,line width=.9pt] (0,0)--(1.5,0)--(0,3)--cycle;
    \draw[->,FigGray!75,line width=.5pt] (-.12,0)--(2.1,0)
      node[right,tinylbl] {$j$};
    \draw[->,FigGray!75,line width=.5pt] (0,-.12)--(0,3.5)
      node[above,tinylbl] {$k$};
    \foreach \k in {0,1}
      \draw[FigGold!75,line width=3.3pt] (0,\k)--(1,\k);
    \foreach \j/\k in {0/0,1/0,0/1,1/1,0/2,0/3}
      \node[circle,fill=DartmouthDarkGreen,draw=white,line width=.4pt,
        inner sep=0pt,minimum size=4.6pt] at (\j,\k) {};
    \foreach \k in {0,1,2,3}
      \node[tinylbl,anchor=east] at (-.14,\k) {$\k$};
    \foreach \j in {0,1}
      \node[tinylbl,anchor=north] at (\j,-.12) {$\j$};
    \node[tinylbl,anchor=north] at (1.5,-.12) {$\frac32$};
  \end{scope}
  \begin{scope}[shift={(8.5,.3)},x=.85cm,y=.75cm]
    % Cell (p,q) records a derivative, as in the collision figure.
    \foreach \p/\h/\col in {0/4/DartmouthGreen,1/2/FigGold}
      \foreach \q in {1,...,\h}
        \filldraw[fill=\col!30,draw=\col!75!black,line width=.4pt]
          (\p,\q-1) rectangle (\p+1,\q);
    \draw[->,FigGray!75,line width=.5pt] (0,0)--(3.3,0)
      node[right,tinylbl] {$p$};
    \draw[->,FigGray!75,line width=.5pt] (0,0)--(0,4.6)
      node[above,tinylbl] {$q$};
    \foreach \p in {0,1}
      \node[tinylbl,anchor=north] at (\p+.5,-.12) {$\p$};
    \foreach \q in {0,1,2,3}
      \node[tinylbl,anchor=east] at (-.14,\q+.5) {$\q$};
    \node[tinylbl,DartmouthGreen!80!black] at (.5,4.25) {$4$};
    \node[tinylbl,FigGold!70!black] at (1.5,2.25) {$2$};
  \end{scope}
\end{tikzpicture}
\caption{Left: the polygon $2\sP$ and its six lattice points from Example~\ref{exm:torous-double-int}. Right: the lex staircase $\sL(\sS)$ inside $\sE(2^2)$.}
\label{fig:interp-example}
\end{figure}

\begin{example}\label{exm:torous-double-int}
	Let $\sP\subset \RR^2$ be the triangle with vertices $(0,0)$, $(3/4,0)$, and $(0,3/2)$; see Figure~\ref{fig:interp-example}. The horizontal slice of $\sP$ at height $y$ has width $\omega_{\sP}(y)=3/4-y/2$ for $0\leq y\leq 3/2$. Thus, $\omega_{\max}(\sP)=3/4$ and 
	\[
	N_{\sP}(\sigma)=2(1-\sigma)-\frac{1}{2} \quad \quad 0\leq \sigma \leq \frac{3}{4}.
	\]
	Thus, we may apply Lemma~\ref{lem:polygon-torus} with $n=2$, $m=2$, and $\Gap=1/2$. Taking $\vv=(0,0)$ and $\sS=(2\sP)\cap \ZZ^2$ the lemma gives  $\sL(\sS) \subset \sE(2^2)$. In this small case we can also check the containment directly: $N_{\sS}(0)=4$, $N_{\sS}(1)=2$ and $N_{\sS}(p)=0$ for all $p\geq2$. Theorem~\ref{thm:torus-interoplation} therefore shows that no nonzero element of 
	\[
	W_{\sS} = \Span_{\KK}\left\langle 1,u,v,uv,v^2,v^3 \right\rangle
	\]
	vanishes to order at least two at two general points on the line $\ell$.
\end{example}
%%%%%%%%%%%%%%%%%%%%%%%%%%%%%%%%%%%%%%%%%%%%%%%%%%%%%%%%%%%%%%%%
%%%%%%%%%%%%%%%%%%%%%%%%%%%%%%%%%%%%%%%%%%%%%%%%%%%%%%%%%%%%%%%%
\section{Tropical Vanishing Criteria}\label{sec:tropical-criteria}
%%%%%%%%%%%%%%%%%%%%%%%%%%%%%%%%%%%%%%%%%%%%%%%%%%%%%%%%%%%%%%%%
%%%%%%%%%%%%%%%%%%%%%%%%%%%%%%%%%%%%%%%%%%%%%%%%%%%%%%%%%%%%%%%%
We now put the constructions from the previous three sections together, translating a vanishing question on a surface into a question about a finite subset of $\ZZ^2$: Section~\ref{sec:specialization-normal} reduces vanishing on $S$ to vanishing for the initial linear series on $Y$, Section~\ref{sec:inital-terms} turns elements of the initial linear series into Laurent polynomials with exponents in a finite set $\SI(V) \subset \ZZ^2$, and Section~\ref{sec:torus} gives a lattice-point criterion ensuring no such nonzero Laurent polynomial vanishes to the prescribed orders.

%%%%%%%%%%%%%%%%%%%%%%%%%%%%%%%%%%%%%%%%%%%%%%%%%%%%%%%%%%%%%%%%
\subsection{The General Criteria}\label{subsec:gen-criteria}
%%%%%%%%%%%%%%%%%%%%%%%%%%%%%%%%%%%%%%%%%%%%%%%%%%%%%%%%%%%%%%%%

We work in the setting of Section~\ref{sec:inital-terms}: $R=\KK[[s]]$ and $K=\KK((s))$ with $\val(s)=1/e$. Let $S$  be a smooth projective geometrically connected surface over $K$, and let $B\subset S$ be a smooth projective curve. Fix a strongly semistable regular model $\pB$ of $B$ whose dual graph $\Gamma$ is a cycle, and a vertex $x_{\bullet}\in \Gamma$ whose component $C=C_{x_{\bullet}}$ is rational.  Let $L$ be a line bundle on $S$, and set $L_{B}=L|_{B}$, $V=H^0(S,L)$, and $N=N_{B/S}$. Fix extensions $\cL_{B}$ and $\cN$ of $L_B$ and $N$ to the model $\pB$ that have degree zero on $C$. Further, fix identifications of their generic fibers and the compatible trivializations from Section~\ref{sec:inital-terms}. We use the resulting coordinates $\TT_{C} \cong \Spec(\KK[u^{\pm1},v^{\pm1}])$ where $u$ is the normal-fiber coordinate and $v$ is the coordinate on $C^{\circ}$. Write $Y$ for the total space of $N$ and $\pY$ for the total space of $\cN$. 

\begin{definition}
    For an integer $j\geq0$ define:
    \[
    \SI_{j}(V)\coloneqq \{j\}\times \left(\bigcup_{0\neq g \in V_{B,j}}\supp \inn_{x_{\bullet}}(g)\right) \subset \ZZ^2.
    \]
    The set $\SI(V)=\cup_{j\geq0}\SI_{j}(V)$ is the \emph{support of initial terms} of $V$. 
\end{definition}

Since $V_{B,j} = 0$ for $j\gg0$ and Corollary~\ref{cor:support-interval} bounds the supports of the initial terms of all nonzero elements of each $V_{B,j}$, the set $\SI(V)$ is finite. 

\begin{theorem}\label{thm:tropical-jets-criteria}
    Fix $\mm\in \ZZ^n_{\geq0}$. Let $\sS\subset \ZZ^2$ be a finite subset such that $\SI(V)\subset \sS$. If there exist distinct points $\xi_1,\ldots,\xi_n\in \KK^{\times}$ such that the evaluation map
    \[
    \begin{tikzcd}[row sep = .2em, column sep = 3.5em]
    W_{\sS} \arrow[r] & \displaystyle \bigoplus_{i=1}^{n}\left(\cO_{\TT,(1,\xi_i)}\big/\mfm_{(1,\xi_i)}^{m_i}\right)
    \end{tikzcd}
    \]
    is injective (or equivalently $W_{\sS}\cap I_{Z_{\overline{\xi}}(\mm)}=0$), then there exist distinct points $q_1,\ldots,q_n\in Y(K)$ such that $V_{B}^{\inn}(-Z_{\qq}(\mm))=0$.
\end{theorem}

\begin{proof}
	Let $(1,\xi_1),\ldots,(1,\xi_n) \in \TT_C$. Since $\pB\to \Spec(R)$ is smooth along $C^{\circ}$, the total space $\pY$ is smooth over $\Spec(R)$ near each point $(1,\xi_i)$. The ring $R$ is complete, and hence is henselian \cite[\href{https://stacks.math.columbia.edu/tag/04GM}{Tag 04GM}]{stacks-project}. Applying smooth lifting on an affine smooth neighborhood of $(1,\xi_i)$ gives a section $q_i: \Spec(R)\to \pY$ whose reduction is $(1,\xi_i)$ \cite[\href{https://stacks.math.columbia.edu/tag/0H74}{Tag 0H74}]{stacks-project}. If two of these sections had the same generic point then, since $\pY$ is separated they would agree everywhere \cite[\href{https://stacks.math.columbia.edu/tag/01KZ}{Tag 01KZ}]{stacks-project}, and in particular have the same reduction. Hence the generic points, which we also denote by $q_i \in Y(K)$, are distinct. 	
	Suppose $F \in V^{\inn}_{B}$ is nonzero and vanishes to order at least $m_i$ at each $q_i$. In the fixed compatible trivializations
	\[
	\inn_{x_{\bullet}}(F) =\sum_{\substack{j\geq0 \\ \ord_{C}(g_j) = \ord_{C}(F)}} \inn_{x_{\bullet}}(g_j)u^j.
	\]
	The support of the $j$-th summand is contained in $\SI_j(V)$, and so $\inn_{x_{\bullet}}(F)$ is a nonzero element of $W_{\sS}$. Lemma~\ref{lem:jet-specialization} implies that $\inn_{x_{\bullet}}(F)$ vanishes to order at least $m_i$ at every $(1,\xi_i)$. Thus, $\inn_{x_{\bullet}}(F)$ is a nonzero element in the kernel of the displayed evaluation map, contradicting its injectivity. 
\end{proof}

\begin{corollary}\label{cor:trop-vanishing}
    Fix $\mm\in \ZZ^n_{\geq0}$ and a finite subset $\sS\subset \ZZ^2$ such that $\SI(V)\subset \sS$. If there are distinct points $\xi_1,\ldots,\xi_n\in \KK^{\times}$ such that the evaluation map
    \[
    \begin{tikzcd}[row sep = .2em, column sep = 3.5em]
    W_{\sS} \arrow[r] & \displaystyle \bigoplus_{i=1}^{n}\left(\cO_{\TT,(1,\xi_i)}\big/\mfm_{(1,\xi_i)}^{m_i}\right)
    \end{tikzcd}
    \]
    is injective, then for a general point $\pp\in \Conf_{n}(S_{\overline{K}})$ we have that 
    \[
    H^0\left(S_{\overline{K}},L_{\overline{K}}\otimes \cI_{Z_{\pp}(\mm)}\right)=0.
    \]
\end{corollary}

\begin{proof}
	Theorem~\ref{thm:tropical-jets-criteria} gives distinct points $q_1,\ldots,q_n \in Y(K)$ such that $V_{B}^{\inn}(-Z_{\qq}(\mm))=0$. Since the initial linear series and evaluation maps commute with extending the ground field the same vanishing holds after extending to $\overline{K}$. We now apply Theorem~\ref{thm:initial-series-to-vanishing} with $\LL=\overline{K}$. 
\end{proof}

\begin{corollary}\label{cor:the-vanishing-we-use}
    Fix $\mm\in \ZZ^n_{\geq0}$ and a finite subset $\sS\subset \ZZ^2$ such that $\SI(V)\subset \sS$.  If $\sL(\sS) \subset \sE(\mm)$, then for a general point $\pp\in \Conf_{n}(S_{\overline{K}})$ we have that 
    \[
    H^0\left(S_{\overline{K}},L_{\overline{K}}\otimes \cI_{Z_{\pp}(\mm)}\right)=0.
    \]
\end{corollary}

\begin{proof}
	By Theorem~\ref{thm:torus-interoplation} the containment $\sL(\sS) \subset \sE(\mm)$ implies the evaluation map in Corollary~\ref{cor:trop-vanishing} is injective for general distinct points $(1,\xi_1),\ldots,(1,\xi_n)\in \TT$. 
\end{proof}

Although the preceding vanishing statements are formulated over $\overline{K}$, in our applications $S$ and $L$ are obtained by base change from a surface and line bundle defined over $\KK$.  Their nonvanishing loci are therefore defined over the original ground field. More precisely, these loci are determinantal, and their formation commutes with field extension. We record this in the following lemma, where $\kappa(\pp)$ denotes the residue field of a point $\pp$. 

\begin{lemma}\label{lem:field-descent}
    Let $\FF$ be any field. Let $S$ be a smooth projective geometrically integral surface defined over $\FF$, and let $L$ be a line bundle on $S$. Fix $\mm\in \ZZ^n_{\geq1}$ and let
    \[
    \Sigma_{L,\mm}\coloneqq \left\{ \pp\in \Conf_{n}(S) \;\; \big| \;\; H^0\left(S_{\kappa(\pp)}, L_{\kappa(\pp)}\otimes \cI_{Z_{\pp}(\mm)}\right)\neq0\right\}
    \]
    be the locus of $n$ points admitting a nonzero section of $L$ with multiplicity at least $m_i$ at $p_i$ for every $i$.
    \begin{enumerate}
    \item The set $\Sigma_{L,\mm}$ is Zariski closed in $\Conf_{n}(S)$, and with its natural determinantal scheme structure, its formation commutes with field extension. 
    \item  Suppose that for some field extension $\LL/\FF$ there is a point $\overline{s}\in \Conf_{n}(S_{\LL})(\LL)$ such that $H^0(S_{\LL}, L_{\LL}\otimes \cI_{Z_{\overline{s}}(\mm)})=0$. Then $\Sigma_{L,\mm}$ is a proper closed subset of $\Conf_{n}(S)$. 
    \end{enumerate}
\end{lemma}

\begin{proof}
    We may assume $H^0(S,L)\neq0$, since otherwise the claims are immediate. We use the universal fat-point construction of Section~\ref{ref:fat-points-family}. Consider the constant smooth family $S\times_{\FF} \Conf_{n}(S)\to \Conf_{n}(S)$ with its $n$ disjoint tautological sections. The argument from Lemma~\ref{lem:uni-fat-bundle} gives a vector bundle $\cP$ on $\Conf_n(S)$ whose fiber at a configuration $\overline{s}\in \Conf_{n}(S)$ is $H^0(Z_{\overline{s}}(\mm), (L_{\kappa(\overline{s})})|_{Z_{\overline{s}}(\mm)})$. As in the proof of Theorem~\ref{thm:initial-series-to-vanishing}, the restriction to the universal fat point scheme defines an evaluation map 
    \[
        \begin{tikzcd}[row sep = .2em, column sep = 3em]
        H^0(S, L)\otimes_{\FF}\cO_{\Conf_{n}(S)} \arrow[r,"\ev"] & \cP.
    \end{tikzcd}
    \]
    The kernel of its fiber at $\pp$ is $H^0(S_{\kappa(\pp)}, L_{\kappa(\pp)}\otimes \cI_{Z_{\pp}(\mm)})$. Hence the complement of $\Sigma_{L,\mm}$ is the locus where $\ev$ is injective on fibers, which is open. Thus, $\Sigma_{L,\mm}$ is closed. We equip it with the determinantal scheme structure defined locally by the minors of size $\dim_{\FF} H^0(S,L)$ of $\ev$. 
    
    By the compatibility with field extension discussed at the end of Section~\ref{ref:fat-points-family} this construction commutes with field extension: the pullback $\ev_{\LL}=\ev\otimes_{\FF}\LL$ is the evaluation map for $L_{\LL}$ and $S_{\LL}$. Since minors pull back to minors, the degeneracy locus of $\ev_{\LL}$ is $\Sigma_{L,\mm}\times_{\FF}\LL$. This proves (1).

    The vanishing at the point $\overline{s}\in\Conf_{n}(S_{\LL})(\LL)$ implies that the fiber of $\ev_{\LL}$ at $\overline{s}$ is injective. Thus, $\overline{s}\not\in \Sigma_{L,\mm}\times_{\FF}\LL$, meaning that $\Sigma_{L,\mm}\times_{\FF}\LL$ is a closed proper subset of $\Conf_{n}(S_{\LL})$. Consequently, $\Sigma_{L,\mm}$ is a proper closed subset of $\Conf_{n}(S)$, proving (2).
\end{proof}

%%%%%%%%%%%%%%%%%%%%%%%%%%%%%%%%%%%%%%%%%%%%%%%%%%%%%%%%%%%%%%%%
\subsection{The Sheared Support Diagram}\label{subsec:tropical-shear}
%%%%%%%%%%%%%%%%%%%%%%%%%%%%%%%%%%%%%%%%%%%%%%%%%%%%%%%%%%%%%%%%

We now construct an explicit finite set $\sS\subset \ZZ^2$ containing $\SI(V)$ by applying Corollary~\ref{cor:support-interval} column by column. For each integer $j\geq0$ choose an effective divisor $Q_{j}$ on the curve $B$ whose horizontal closure avoids $C$ and  on which every section of $V_{B,j}$ vanishes. We call $Q_j$ a \emph{base divisor} because it records the basepoints forced on the coefficient spaces $V_{B,j}$. Such divisors always exist since $Q_j=0$ is allowed.  Define
\[
\sD_{j}\coloneqq \dd(\cL_B)-j\dd(\cN)-\Sp(Q_j),
\]
which is a divisor on $\Gamma$. Set $\beta_j\coloneqq \mom_{x_{\bullet}}(\sD_j)$. Every $g\in V_{B,j}$ is a section of the generic fiber of $\cL_{B}\otimes\cN^{-j}(-\overline{Q}_{j})$. By definition $\dd(\cO_{\pB}(\overline{Q}_j)) = \Sp(Q_j)$ and so the multidegree of $\cL_{B}\otimes\cN^{-j}(-\overline{Q}_{j})$ is exactly $\sD_j$.

There is one point where some care is needed: the initial terms were defined using sections of $\cA_j$, while now we will want to use the support bound that comes from $\cA_j(-\overline{Q}_j)$. Since $\overline{Q}_j$ is disjoint from $C$, the natural inclusion $\cA_j(-\overline{Q}_j)\hookrightarrow \cA_j$ is an isomorphism on an open neighborhood of $C$. Thus, $\cA_{j}(-\overline{Q}_j)$ still has degree zero on $C$. Under this identification, each nonzero coefficient has the same order along $C$ in either bundle. Dividing by the same power of $s$ and restricting to $C$ therefore gives a natural identification of the corresponding residual sections. The trivialization of $\cA_j|_{C}$ induces a trivialization on  $\cA_{j}(-\overline{Q}_j)|_{C}$, and using this induced trivialization we obtain the same Laurent polynomial. Hence the two initial term calculations give the same support for each coefficient. 

\begin{proposition}\label{prop:shear-set}
    With the notation as above, the finite set
    \[
    \sS\coloneqq \bigcup_{\substack{j\geq0 \\ V_{B,j}\neq0}} \{j\} \times\left(\left[\beta_j, \beta_j+\deg(\sD_j)\right]\cap \ZZ\right)
    \]
    contains $\SI(V)$.
\end{proposition}

\begin{proof}
	The discussion in the previous paragraph  identifies the initial terms obtained from $\cL_{B}\otimes \cN^{-j}(-\overline{Q}_j)$ with the original initial terms obtained from $\cL_{B}\otimes \cN^{-j}$. Since $\cL_{B}\otimes \cN^{-j}(-\overline{Q}_j)$ has degree zero on $C$ and multidegree $\sD_j$, Corollary~\ref{cor:support-interval} gives
	\[
	\supp \inn_{x_{\bullet}}(g_j)\subset \left[\beta_j, \beta_j+\deg(\sD_j)\right]\cap\ZZ
	\]
	for every nonzero element $g_j\in V_{B,j}$. Taking the union over all $j$ proves the containment. The set $\sS$ is finite since $V_{B,j}$ is nonzero for only finitely many $j$ and each interval is bounded. 
\end{proof}

It is worth thinking about the shape of $\sS$ for a moment. Set $\alpha=\mom_{x_{\bullet}}(\cN)$. The additivity of tropical moment and degree gives
\[
\beta_j=\mom_{x_{\bullet}}\left(\dd(\cL_B)\right)-j\alpha-\mom_{x_{\bullet}}\left(\Sp(Q_j)\right) \quad \text{and} \quad \deg\left(\sD_j\right) = \deg\left(L_B\right) - j \deg\left(N\right) - \deg\left(Q_j\right). 
\]
Thus, the real interval bounding the $j$-th column of $\sS$ begins at $\beta_j$ and has length $\deg(\sD_j)$. If $Q_j$ is independent of $j$, the lower endpoint is linear in $j$ with slope $-\alpha$. In general, the base divisors contribute the extra term $-\mom_{x_{\bullet}}(\Sp(Q_j))$ and shorten the bounding interval by $\deg(Q_j)$. We call $\alpha$ the \emph{shear} of $\sS$ because we can picture its contribution as sliding the $j$-th column vertically by $-j\alpha$, without changing its length. This gives the diagram an overall tilt; the base divisor can then be used to shorten certain columns.  The shear is determined by the chosen extension $\cN$  of $N$ and is independent of the base divisors. The separate interpolation hypothesis of Theorem~\ref{thm:tropical-jets-criteria} must still be checked for this set. 

In the proof of Theorem~\ref{main-thm:technical-nagata} we will take $Q_{j}=0$ for every $j$. In this situation the normal bundle has positive degree, so that the lengths of the support intervals decrease linearly with $j$, and the resulting support diagram lies in a triangle. We have allowed for nonzero base divisors in Proposition~\ref{prop:shear-set} because prescribing vanishing conditions on the coefficient spaces can give small support diagrams in other interpolation problems. For example, one may give an alternative, more technical proof of Theorem~\ref{main-thm:technical-nagata} by specializing $b=\nu^2h$ of the $r$ points to $B$ and then blowing them up, so that the resulting strict transform of $B$ has normal bundle of degree zero. The base divisors $Q_{j}=(j-m)_{+}(a_1+\cdots+a_b)$, where the $a_i$ are the points specialized to $B$,  record the vanishing forced along the exceptional curves. Note this alternative proof is more in line with the classical ideas in \cite{nagata59,harbourne03}.

%%%%%%%%%%%%%%%%%%%%%%%%%%%%%%%%%%%%%%%%%%%%%%%%%%%%%%%%%%%%%%%%
%%%%%%%%%%%%%%%%%%%%%%%%%%%%%%%%%%%%%%%%%%%%%%%%%%%%%%%%%%%%%%%%
\section{Proofs of Main Results}\label{sec:AVC-Bring-It-HOME}
%%%%%%%%%%%%%%%%%%%%%%%%%%%%%%%%%%%%%%%%%%%%%%%%%%%%%%%%%%%%%%%%
%%%%%%%%%%%%%%%%%%%%%%%%%%%%%%%%%%%%%%%%%%%%%%%%%%%%%%%%%%%%%%%%

In this section we return to the setting of Theorem~\ref{main-thm:technical-nagata}. Fix a smooth projective connected surface $X$ defined over $\KK$ and an ample line bundle $H$ on $X$. Set $h\coloneqq \langle H^2\rangle$.

%%%%%%%%%%%%%%%%%%%%%%%%%%%%%%%%%%%%%%%%%%%%%%%%%%%%%%%%%%%%%%%%
\subsection{From Uniform Vanishing to Nefness}
%%%%%%%%%%%%%%%%%%%%%%%%%%%%%%%%%%%%%%%%%%%%%%%%%%%%%%%%%%%%%%%%

All objects (varieties, points, etc.) in this subsection are defined over $\KK$. For a point $\pp \in \Conf_{r}(X)(\KK)$ let $\mu_{\pp}:X_{\pp}\to X$ be the blow-up of $X$ at the points in $\pp$. We write $E=E_1+\cdots+E_r$ for the exceptional divisor of this blow-up, and abuse notation by writing $H$ for $\mu^*_{\pp}H$. 

\begin{lemma}\label{lem:vanishing-to-nefness}
    Let $\nu \geq1$, let $r > \nu^2h$, and let $\pp \in \Conf_{r}(X)$. If, for all positive integers $d$ and $m$,
    \[
    H^0\left(X_{\pp}, \cO_{X_{\pp}}(dH-mE)\right)=0 \quad \text{ whenever} \quad \nu < \frac{d}{m} <\sqrt{\frac{r}{h}},
    \]
    then $H-\sqrt{h/r}E$ is nef on $X_{\pp}$. 
\end{lemma}

\begin{proof}
    To simplify notation let $\lambda = \sqrt{h/r}$ and consider the $\RR$-divisor $A = H-\lambda E$. Recall that by the definition of $\lambda$ we have that $\langle A^2\rangle = 0$. A computation shows that $\langle H \cdot A\rangle = h >0$. Towards a contradiction suppose that $C \subset X_{\pp}$ is an integral curve such that $\langle A\cdot C\rangle = a<0$. Since $\langle A\cdot E_i\rangle = \lambda>0$ for all $i$, the curve $C$ is not one of the exceptional curves. Thus, $\langle H\cdot C\rangle >0$. For $t>0$ we have
    \[
    \frac{\langle (A-tC)^2\rangle}{t}=-2a+t\langle C^2\rangle \quad \quad \text{and} \quad \quad \langle H\cdot (A-tC)\rangle = h -t\langle H\cdot C\rangle.
    \]
    So as $t\to 0^{+}$ we see these approach $-2a>0$ and $h>0$. Thus, we may fix a small rational number $t>0$ such that $ \langle (A-tC)^2\rangle >0 $ and $\langle H\cdot (A-tC)\rangle>0$. Since $\lambda$ may be irrational the divisor $A-tC$ may not be a $\QQ$-divisor. Since $r>\nu^2 h$ we have that $\lambda < \nu^{-1}$, and we may choose a rational number $c$ such that $\lambda < c < \nu^{-1}$ and $D\coloneqq H-cE-tC$ still satisfies $\langle D^2\rangle >0$ and $\langle H\cdot D \rangle >0$.
    
    Choose a positive integer $d$ large enough so that $dt$ and $dc$ are integers and $dD$ is a $\ZZ$-divisor. Enlarging $d$ if necessary we may further assume that
    \[
    \langle H\cdot(K_{X_{\pp}}-dD)\rangle = \langle H \cdot K_{X_{\pp}}\rangle - d\langle H\cdot D\rangle < 0.
    \]
    Since $H$ is nef this means $K_{X_{\pp}}-dD$ is not effective, and so by Serre duality $H^2(X_{\pp}, \cO_{X_{\pp}}(dD))=0$. Together with Riemann--Roch for a surface, this shows that, for $d\gg0$:
    \[
    \dim H^0\left(X_{\pp}, \cO_{X_{\pp}}(dD)\right) \geq \chi\left(X_{\pp}, \cO_{X_{\pp}}(dD)\right)=\chi\left(\cO_{X_{\pp}}\right) + \frac{d^2\langle D^2\rangle - d\langle D\cdot K_{X_{\pp}}\rangle}{2}>0.
    \]
    Multiplying a nonzero section of $dD$ by a section defining the effective divisor $dtC$ gives a nonzero section in $H^0(X_{\pp}, dH-mE)$ where $m\coloneqq dc$. However, by construction $d,m$ are positive integers with 
    \[
    \nu < \frac{d}{m} = \frac{1}{c} < \frac{1}{\lambda} = \sqrt{\frac{r}{h}},
    \]
    providing the desired contradiction.
\end{proof}

\begin{corollary}\label{cor:vanishing-to-nagata}
    Under the hypotheses of Lemma~\ref{lem:vanishing-to-nefness},
    \[
    \varepsilon\left(X,H;\pp \right)=\sqrt{\frac{h}{r}} \quad \quad \text{and}\quad \quad\sum_{i=1}^{r}\mult_{p_i}(C) \leq \sqrt{\frac{r}{h}}\langle H\cdot C\rangle
    \]
    for every curve $C\subset X$. 
\end{corollary}

\begin{proof}
	Lemma~\ref{lem:vanishing-to-nefness} gives $\varepsilon(X,H;\pp)\geq \sqrt{h/r}$. The reverse inequality was observed in the introduction. The second claim follows from the first by Lemma~\ref{lem:nagata-to-seshadri}.
\end{proof}

%%%%%%%%%%%%%%%%%%%%%%%%%%%%%%%%%%%%%%%%%%%%%%%%%%%%%%%%%%%%%%%%
\subsection{A Nodal Curve and Its Necklace}\label{subsec:nodal-curve}
%%%%%%%%%%%%%%%%%%%%%%%%%%%%%%%%%%%%%%%%%%%%%%%%%%%%%%%%%%%%%%%%

We first show that the hypotheses of Theorem~\ref{main-thm:technical-nagata} are satisfied when $\nu=2$ and $H$ is very ample and $(X,H)\neq (\PP^2,\cO_{\PP^2}(1))$, which amounts to showing $|2H|$ contains an integral curve $B_0$ with one ordinary node.

\begin{lemma}\label{lem:2H-nodal-curve-exists}
    If $H$ is very ample and $h=\langle H^2\rangle \geq2$ then  $|2H|$ contains an integral curve whose only singularity is one ordinary node.
\end{lemma}

\begin{proof}
Fix a point $o\in X$ and let $\mu:X_{o}\to X$ be the blow-up of $X$ at $o$. Let $E_o$ denote the exceptional divisor and $A\coloneqq \mu^*H-E_{o}$. Since $H$ is very ample $A$ is globally generated, $\langle A^2\rangle = h-1 >0$, and $A|_{E_{o}} \cong \cO_{\PP^1}(1)$. Since $A$ is globally generated the image of the restriction map 
    \[
    \begin{tikzcd}[row sep = .2em, column sep = 3em]
        H^0\left(X_{o}, A\right) \arrow[r] & H^0\left(E_{o},A|_{E_o}\right) \cong H^0\left(\PP^1, \cO_{\PP^1}(1)\right)
    \end{tikzcd}
    \]
is basepoint free. Since $H^0(\PP^1,\cO_{\PP^1}(1))$ has no proper basepoint free subspace the restriction map is surjective. Since $\cO_{\PP^1}(1)$ is normally generated the restriction map $H^0(X_{o}, 2A) \to H^0(\PP^1,\cO_{\PP^1}(2))$ is also surjective. As a general element of $|\cO_{\PP^1}(2)|$ has two distinct zeros the desired condition is open and nonempty. Thus, by Bertini's theorem we may choose $D\in |2A|$ that is smooth and whose intersection with $E_o$ is transverse at two distinct points. We now show that $D$ is connected. Towards a contradiction assume that  $D_1,\ldots,D_k$ are the connected components of $D$ with $k\geq2$. Since $D\in |2A|$ and $A$ is nef we have that
    \[
    \langle D^2_i\rangle = \langle D\cdot D_i\rangle = 2\langle A\cdot D_i\rangle \geq0 
    \]
    for all $i$. Now since $\sum_i \langle D_i^2\rangle = \langle D^2 \rangle = 4(h-1)>0$ there exists an $i$ such that $\langle D_i^2\rangle >0$. Assume that $\langle D^2_1\rangle >0$. As $D_1$ and $D_2$ are connected components of $D$ we have that $\langle D_1\cdot D_2\rangle = 0$. By the Hodge index theorem either $D_2$ is numerically trivial or $\langle D_2^2\rangle <0$. This is a contradiction since an effective curve cannot be numerically trivial and by our argument above $\langle D_i^2\rangle \geq0$ for all $i$. 
    
    Thus, $D$ is smooth and connected, hence integral. Its image  $C=\mu(D)$ is an integral projective curve on $X$. Since  $D \sim 2\mu^*H-2E_{o}$, the curve $C$ lies in $|2H|$. The map $\mu$ is an isomorphism away from $E_o$, so $C$ is smooth away from $o$. At $o$ the curve $C$ has an ordinary node since $D$ meets $E_o$ transversely at two distinct points. 
\end{proof}

\begin{remark}\label{rem:degree-one}
	As suggested above the only pair $(X,H)$ not covered by Lemma~\ref{lem:2H-nodal-curve-exists} is $(\PP^2,\cO_{\PP^2}(1))$. That is, if $(X,H)\neq (\PP^2,\cO_{\PP^2}(1))$ and $H$ is very ample, then $\langle H^2\rangle \geq2$. The complete linear series $|H|$ realizes $X \subset \PP^N$  as a non-degenerate smooth surface of degree $h$. The classical bound $\deg(X)\geq\codim(X)+1= N-1$ \cite{eisenbudHarris87} shows that if $h=1$ then $N=2$ and so $X=\PP^2$.
\end{remark}

We now fix the data of Theorem~\ref{main-thm:technical-nagata} over the ground field $\KK$: 
\begin{itemize}
    \item $X$ a smooth projective connected surface over $\KK$,
    \item $H$ an ample line bundle on $X$ with $h\coloneqq\langle H^2\rangle$,
    \item $\nu\geq1$ an integer with  $b\coloneqq \nu^2h$, and
    \item $B_{0}$ an integral curve in $|\nu H|$ defined over $\KK$ whose only singularity is an ordinary node $o$, which is not a basepoint of $|\nu H|$.
\end{itemize}

%%%%%%%%%%%%%%%%%%%%%%%%%%%%%%%%%%%%%%%%%%%%%%%%%%%%%%%%%%%%%%%%
\subsubsection{The Smoothing}
%%%%%%%%%%%%%%%%%%%%%%%%%%%%%%%%%%%%%%%%%%%%%%%%%%%%%%%%%%%%%%%%

Choose a section $f_0\in H^0(X, \cO_{X}(\nu H))$ such that $B_{0}=\VV(f_0)$. Since $o$ is not a basepoint of $|\nu H|$ we may  choose a section $f_1\in H^0(X,\cO_{X}(\nu H))$ such that $B_1=\VV(f_1)$ does not contain $o$. Consider the family
\[
\pB^{(t)} \coloneqq \VV(f_0-tf_1) \subset X\times \Spec(\KK[[t]]),
\]
where we are abusing notation and thinking about $f_i$ as its pullback to $X\times \Spec(\KK[[t]])$. This family is flat and proper over $\KK[[t]]$, and the special fiber is $B_0$. Near $o$, the section $f_1$ is a unit and the family is given by $t=f_0/f_1$. Thus, the total space is regular at $(o,0)$. Since $o$ is a split node of the special fiber, \cite[\href{https://stacks.math.columbia.edu/tag/0CBX}{Tag 0CBX}]{stacks-project} gives an isomorphism $\widehat{\cO}_{\pB^{(t)},(o,0)} \cong \KK[[z_1,z_2,t]]/\langle z_1z_2-f\rangle$ of $\KK[[t]]$-algebras where $f\in \langle t\rangle$. Since $\pB^{(t)}$ is regular at $(o,0)$, after potentially absorbing a unit into the formal coordinates, we may take $f=t$ and so
\[
\widehat{\cO}_{\pB^{(t)},(o,0)} \cong \KK[[z_1,z_2,t]]/\langle z_1z_2-t\rangle
\]
as $\KK[[t]]$-algebras. In the formal local neighborhood above, the relative singular locus is defined by $z_1=z_2=0$ and so is contained in the special fiber. Since $B_0$ is smooth away from $o$ and the family is flat and locally of finite presentation, the morphism is smooth at every other point of the special fiber \cite[\href{https://stacks.math.columbia.edu/tag/01V9}{Tag 01V9}]{stacks-project}. If the generic fiber of $\pB^{(t)}$ had a singular point, properness of the relative singular locus would give a point, by specialization, on the special fiber, contradicting our calculations. Thus, the generic fiber is smooth. Since the family is flat, proper, and of finite presentation the locus of geometrically integral fibers is open \cite[Th\'{e}or\`{e}me~12.2.4(viii)]{EGAIV3}. The special fiber $B_0$ is geometrically integral since it is integral over  $\KK$, which is algebraically closed. Hence the generic fiber of $\pB^{(t)}\to \Spec(\KK[[t]])$ is also geometrically integral. 

%%%%%%%%%%%%%%%%%%%%%%%%%%%%%%%%%%%%%%%%%%%%%%%%%%%%%%%%%%%%%%%%
\subsubsection{The Necklace}
%%%%%%%%%%%%%%%%%%%%%%%%%%%%%%%%%%%%%%%%%%%%%%%%%%%%%%%%%%%%%%%%

Fix an integer $e\geq3$, which we will specify later. Consider the extension $\KK[[t]] \to \KK[[s]]$ given by $t\mapsto s^{e}$, and let $\pB^{(s)} \subset X\times \Spec(\KK[[s]])$ denote the base change of the family $\pB^{(t)}$ constructed in the previous subsubsection along this map. Locally around the node $\widehat{\cO}_{\pB^{(s)},(o,0)} \cong \KK[[z_1,z_2,s]]/\langle z_1z_2-s^e\rangle$, which is a type $A_{e-1}$ singularity. By blowing up one may resolve this singularity to give a strongly semistable regular model $\pB\to \Spec(\KK[[s]])$. The generic fiber of $\pB$ is a smooth curve $B\subset X_{K}$ defined over the field $K=\KK((s))$. We normalize the valuation on $K$ so that $\val(s)=1/e$ meaning that $\val(t) = 1$. The following description of the resolution is standard; see \cite[Ch.~10]{liu02} or \cite[\href{https://stacks.math.columbia.edu/tag/0CDE}{Tag 0CDE}, \href{https://stacks.math.columbia.edu/tag/0CDC}{Tag 0CDC}]{stacks-project}.

\begin{proposition}\label{prop:necklace}
    Let $\pB$ be constructed as above:
    \begin{enumerate}
        \item The special fiber of $\pB$ consists of the normalization $C_{*}$ of $B_{0}$ together with a chain of $e-1$ rational curves $C_{1/e},\ldots,C_{(e-1)/e}$ joining the two points of $C_{*}$ lying over $o$.
        \item The dual graph $\Gamma$ of the special fiber of $\pB$ is a cycle of length one and every edge has length $1/e$. We denote the vertex corresponding to $C_{*}$ by $x_{*}$.
        \item For $1\leq a< e$ the function $z_1/s^{a}$ restricts to a coordinate $v_{a/e}$ on $C_{a/e}\cong\PP^1$ whose zero is the node $C_{a/e}\cap C_{(a+1)/e}$ and whose pole is the node $C_{a/e}\cap C_{(a-1)/e}$ where $C_{0/e}=C_{e/e}=C_{*}$. 
    \end{enumerate}
\end{proposition}

\begin{proof}
	Since $\pB^{(s)}$ is regular away from $o$, it suffices to resolve the singular points lying over $o$. Blowing up commutes with flat base change \cite[\href{https://stacks.math.columbia.edu/tag/0805}{Tag 0805}]{stacks-project}, and completion is flat \cite[\href{https://stacks.math.columbia.edu/tag/00MB}{Tag 00MB}]{stacks-project}. We may therefore compute the resolution after completion at $o$, using the standard model $Y_e = \Spec( \KK[[s]][z_1,z_2]/\langle z_1z_2-s^e\rangle)$. By induction we prove the resolution of $Y_e$ is covered by the regular charts
	\[
	U_{a} \coloneqq \Spec \left(\KK[[s]][x_a,y_a]/\langle x_ay_a-s\rangle\right) \quad \quad \text{for $0\leq a < e$}.
	\]
	When $e=1$ we have $Y_1=U_0$, which is regular. For $e\geq2$ the blow-up of $Y_e$ at its singular point is covered by three charts: i) $\VV(z_1-x_{0}, z_2-s^{e-1}y_0)$ which is $U_0$, ii) $\VV(z_1-s^{e-1}x_{e-1}, z_2-y_{e-1})$ which is $U_{e-1}$, and iii) a chart isomorphic to $Y_{e-2}$. These charts glue with $x_{a+1}=y_{a}^{-1}$ and $y_{a+1}=sy_a$. The special fiber of the chart $U_a$ is $\VV(x_ay_a)$, i.e., two lines meeting transversely at a point $q_a$. When $e=2$, both coordinates in the last chart are invertible and so it is a subset of the first two charts. When $e\geq3$ we induct upon this last chart. The resulting maps are $z_1=s^ax_a$ and $z_2=s^{e-1-a}y_a$. For $1\leq a < e$ the lines $\{y_a=0\}\subset U_a$ and $\{x_{a-1}=0\} \subset U_{a-1}$ glue together to give a curve $C_{a/e} \cong \PP^1$ lying over $o$. The two remaining lines $\{y_0 =0 \} \subset U_0$ and $\{x_{e-1}=0\} \subset U_{e-1}$ are the strict transforms of the branches of $B_0$. This proves (1). 

Turning to (2), at each node $\widehat{\cO}_{\pB,q_a}$ is isomorphic to $\KK[[x_a,y_a,s]]/\langle x_ay_a-s\rangle$, so each of the $e$ edges has length $\val(s)=1/e$. Finally, on $U_a$ we have that $z_1/s^a = x_a$ and on $U_{a-1}$ we have $z_1/s^a = y_{a-1}^{-1}$. These charts cover $C_{a/e}$. The zero of $z_1/s^a$ is $q_a$ and its pole is $q_{a-1}$.
\end{proof}

To make computations easier we fix an isometry $x:\RR/\ZZ \to \Gamma$, $w\mapsto x(w)$ where $x(0)=x_*$ and $x(a/e)$ is equal to the vertex corresponding to the component $C_{a/e}$ for $1\leq a < e$. Choose this parameterization so that the edge corresponding to the zero of the coordinate $v_{1/e}$ leaves $x(1/e)$ in the direction of increasing $w$. When working with this parameterization we will use the shorthand $C_{w}\coloneqq C_{x(w)}$, with the indices taken modulo $1$, so $C_{0}=C_{1}=C_{*}$. For $0<w<1$ with $ew\in \ZZ$, part (3) of Proposition~\ref{prop:necklace} gives a coordinate $v_w$ on the component $C_w$ whose zero is the node in the positively oriented direction. This agrees with the conventions in Section~\ref{sec:inital-terms}. 

For the chosen rational component among $C_{1/e},\ldots,C_{(e-1)/e}$ we set $C=C_{w_{\bullet}}$ and $x_{\bullet}=x(w_\bullet)$ where $0<w_{\bullet}<1$ and $ew_{\bullet}\in \ZZ$. Then for every $w\in \RR/\ZZ$ we have
\begin{equation}\label{eq:len}
    \len_{x_\bullet}(x(w))=(w_\bullet-w)\mod\ZZ \in [0,1).
\end{equation}

\begin{figure}
\centering
\begin{tikzpicture}[fig, x=1.18cm, y=1.18cm]
  %------------------------------------------------ (a) the nodal curve and a smoothing
  \begin{scope}[shift={(0,0)}, scale=1.55]
    % B_t : y^2 = x^3 + x^2 + 0.03 (real points)
    \draw[DartmouthGreen!60, densely dashed, line width=0.6pt] plot[smooth] coordinates {
      (0.770,-1.039) (0.680,-0.898) (0.592,-0.767) (0.507,-0.646) (0.424,-0.535) (0.344,-0.434)
      (0.265,-0.345) (0.189,-0.269) (0.115,-0.212) (0.044,-0.179) (-0.025,-0.175) (-0.092,-0.194)
      (-0.157,-0.225) (-0.219,-0.260) (-0.279,-0.293) (-0.337,-0.324) (-0.392,-0.351) (-0.445,-0.374)
      (-0.496,-0.392) (-0.544,-0.406) (-0.590,-0.416) (-0.634,-0.421) (-0.675,-0.422) (-0.715,-0.419)
      (-0.751,-0.413) (-0.786,-0.403) (-0.818,-0.389) (-0.848,-0.373) (-0.876,-0.354) (-0.901,-0.332)
      (-0.924,-0.308) (-0.945,-0.281) (-0.964,-0.253) (-0.980,-0.223) (-0.994,-0.191) (-1.005,-0.158)
      (-1.014,-0.124) (-1.021,-0.089) (-1.026,-0.054) (-1.028,-0.018) (-1.028,0.018) (-1.026,0.054)
      (-1.021,0.089) (-1.014,0.124) (-1.005,0.158) (-0.994,0.191) (-0.980,0.223) (-0.964,0.253)
      (-0.945,0.281) (-0.924,0.308) (-0.901,0.332) (-0.876,0.354) (-0.848,0.373) (-0.818,0.389)
      (-0.786,0.403) (-0.751,0.413) (-0.715,0.419) (-0.675,0.422) (-0.634,0.421) (-0.590,0.416)
      (-0.544,0.406) (-0.496,0.392) (-0.445,0.374) (-0.392,0.351) (-0.337,0.324) (-0.279,0.293)
      (-0.219,0.260) (-0.157,0.225) (-0.092,0.194) (-0.025,0.175) (0.044,0.179) (0.115,0.212)
      (0.189,0.269) (0.265,0.345) (0.344,0.434) (0.424,0.535) (0.507,0.646) (0.592,0.767)
      (0.680,0.898) (0.770,1.039)};
    % B_0 : y^2 = x^2(x+1), parametrized by (s^2-1, s(s^2-1))
    \draw[DartmouthDarkGreen, line width=1.1pt] plot[domain=-1.33:1.33, samples=80]
      ({\x*\x-1},{\x*(\x*\x-1)});
    \node[pt=FigRed] at (0,0) {};
    \node[tinylbl, anchor=west] at (0.06,-0.02) {$o$};
    \node[lbl, DartmouthDarkGreen] at (0.95,1.05) {$B_0$};
    \node[tinylbl, DartmouthGreen] at (-0.62,-0.62) {$B_t$};
  \end{scope}
  \draw[->, line width=0.6pt] (1.6,0) -- (2.7,0) node[midway, above, tinylbl, align=center] {$e=6$};
  %------------------------------------------------ (b) the necklace (e = 6)
  \begin{scope}[shift={(5.3,0.05)}]
    % nodes: n_0 = C_* \cap C_{1/6}, n_a = C_{a/6} \cap C_{(a+1)/6}, n_5 = C_{5/6} \cap C_*
    \foreach \a/\t in {0/-15,1/30,2/70,3/110,4/150,5/195}
      \coordinate (n\a) at ({1.9*cos(\t)},{1.25*sin(\t)});
    % C_* : the normalization of B_0, meeting the chain at the two preimages of o
    \draw[DartmouthDarkGreen, line width=1.2pt] plot[smooth, tension=0.7] coordinates
      {(-2.06,0.0) (n5) (-1.35,-0.95) (-0.45,-1.0) (0.45,-0.82) (1.35,-0.95) (n0) (2.06,0.0)};
    \foreach \a/\lab/\ang in {1/{1/6}/180,2/{1/3}/50,3/{1/2}/90,4/{2/3}/130,5/{5/6}/0} {
      \pgfmathtruncatemacro{\b}{\a-1}
      \ifnum\a=3
        \draw[FigRed, line width=1.5pt] ($(n\b)!-0.2!(n\a)$) -- ($(n\a)!-0.2!(n\b)$);
      \else
        \draw[black!70, line width=0.9pt] ($(n\b)!-0.2!(n\a)$) -- ($(n\a)!-0.2!(n\b)$);
      \fi
      \node[tinylbl] at ($($(n\b)!0.5!(n\a)$)+(\ang:0.36)$) {$C_{\lab}$};
    }
    \foreach \a in {0,...,5} \fill (n\a) circle (0.9pt);
    % coordinate v on C = C_{3/6}: zero at the node with C_{4/6}, pole at the node with C_{2/6}
    \node[tinylbl, FigRed] at ($(n3)+(0.12,-0.2)$) {$0$};
    \node[tinylbl, FigRed] at ($(n2)+(-0.14,-0.2)$) {$\infty$};
    \node[lbl, DartmouthDarkGreen] at (0,-1.25) {$C_*$};
    \node[tinylbl, align=center] at (0,0.1) {special fiber\\[-1pt] of $\pB$};
  \end{scope}
  \draw[->, line width=0.6pt] (7.75,0) -- (8.85,0) node[midway, above, tinylbl] {dual graph};
  %------------------------------------------------ (c) the dual graph Gamma
  \begin{scope}[shift={(10.5,-0.1)}]
    \def\r{1.15}
    \draw[line width=0.8pt] (0,0) circle (\r);
    % orientation: increasing w is counterclockwise, x_w at angle -90 + 360 w
    \draw[->, line width=0.6pt] (100:{\r-0.24}) arc (100:140:{\r-0.24});
    \foreach \a/\lab in {1/{x_{1/6}},2/{x_{1/3}},4/{x_{2/3}},5/{x_{5/6}}}
      \node[tinylbl] at ({-90+60*\a}:{\r+0.38}) {$\lab$};
    \foreach \a in {1,2,4,5} \node[vtx] at ({-90+60*\a}:\r) {};
    \node[circle, fill=DartmouthDarkGreen, inner sep=0pt, minimum size=7pt] at (-90:\r) {};
    \node[tinylbl] at (-90:{\r+0.32}) {$x_*$};
    \node[circle, fill=FigRed, inner sep=0pt, minimum size=6pt] at (90:\r) {};
    \node[tinylbl, FigRed] at (90:{\r+0.3}) {$x_\bullet=x(1/2)$};
  \end{scope}
\end{tikzpicture}
\caption{The construction in Proposition~\ref{prop:necklace} applied to the curve $B_0=\VV(x^2(x+z)-y^2z)\subset \PP^2$. From left to right: the curve $B_0$ and a nearby fiber $B_t$ of $\pB^{(t)}$, the special fiber of $\pB$ with $e=6$, and the dual graph $\Gamma$ with the orientation fixed above. }
\label{fig:necklace}
\end{figure}
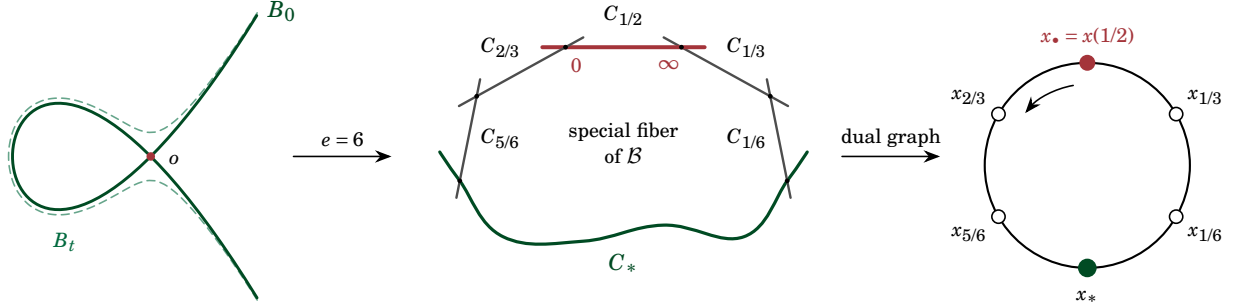

\begin{example}
	Consider the nodal cubic $B_0=\VV(x^2(x+z)-y^2z)\subset \PP^2$, which has a node at $[0:0:1]$. Applying Proposition~\ref{prop:necklace} to this curve with $e=6$ results in a model whose special fiber has 6 components: $C_*$ the normalization of $B_0$, which in this example is isomorphic to $\PP^1_{\KK}$, and five rational curves $C_{1/6}, C_{1/3}, C_{1/2}, C_{2/3}, C_{5/6}$. The latter form a chain whose ends meet $C_*$ at the two points above $o$. Thus, the dual graph $\Gamma$ is a cycle with six edges of length $1/6$; see Figure~\ref{fig:necklace}.
	
	Choose $C=C_{1/2}$, so that $w_{\bullet}=1/2$ and $x_{\bullet}=x(1/2)$. The coordinate $v_{1/2}$ has its zero at $C_{1/2}\cap C_{2/3}$ and its pole at $C_{1/2}\cap C_{1/3}$, and  $\len_{x_\bullet}(x_*)=1/2$. If $\cH$ is the pullback of the class of a line, $\cH$ has degree zero on $C$ with  $\dd(\cH) = 3x_*$ and  $\mom_{x_{\bullet}}(\cH)=-3/2$. Corollary~\ref{cor:support-interval} therefore predicts that every nonzero section of $\cO_{B}(H)$ has initial term supported in $[-3/2,\; 3/2]\cap \ZZ = \{-1,0,1\}$. All of these exponents occur. 
	
	Since $z$ does not vanish at $o=[0:0:1]$, its pullback to $\pB$ trivializes $\cH$ near $C$. In this trivialization $\inn_{x_{\bullet}}(z) = 1$. To compute the initial terms of $x$ and $y$, choose formal coordinates $z_1,z_2$ of the node on $\pB^{(t)}$ such that $z_1z_2=t$. Since the tangent lines of $B_0$ at $o$ are $y=\pm x$ the linear terms of $x/z$ and $y/z$ involve both $z_1$ and $z_2$ with nonzero coefficients. After the base change to $\pB^{(s)}$ the chart near the generic point of $C$ satisfies $z_1=s^3v$ and $z_2=s^3v^{-1}$ where $v$ restricts to $v_{1/2}$ on $C$. Thus, $x$ and $y$ vanish to order $3$ along $C$ and $\inn_{x_{\bullet}}(x) = av+bv^{-1}$ for some $a,b\in \KK^{\times}$ and similarly for $y$. 
\end{example}

%%%%%%%%%%%%%%%%%%%%%%%%%%%%%%%%%%%%%%%%%%%%%%%%%%%%%%%%%%%%%%%%
\subsection{Prescribing the Shear}
%%%%%%%%%%%%%%%%%%%%%%%%%%%%%%%%%%%%%%%%%%%%%%%%%%%%%%%%%%%%%%%%
Set $b=\nu^2h$. We now arrange the shear $\alpha$ from Section~\ref{subsec:tropical-shear} by choosing a rational component of the necklace and twisting an extension of the normal bundle by a vertical divisor. We work directly on $X_{K}$ where $N=N_{B/X_{K}}\cong \cO_{B}(\nu H)$ has degree $b$. The extension of $N$ to $\pB$ will have degree zero on the chosen rational component, and perhaps nonzero degree on other components.

Fix $e\geq3$ and a rational component $C=C_{w_\bullet}$ with $0<w_{\bullet}<1$ and $ew_{\bullet}\in \ZZ$, and write $x_{\bullet}=x(w_\bullet)$ for its corresponding vertex in $\Gamma$. Later in Proposition~\ref{prop:find-the-moment}, both $e$ and $w_{\bullet}$ will be constructed from other prescribed parameters. The vertex $x(0)=x_{*}$ continues to denote the vertex corresponding to the normalization $C_{*}$. Let $\cH$ denote the pullback of $H$ to $\pB$. The restriction of $\cH$ to $C_*$ has degree $\langle H\cdot B_0\rangle = \nu h$, while the restriction of $\cH$ to the other components of the special fiber has degree zero. Thus, for a rational component $C=C_{w_{\bullet}}$ with $0<w_{\bullet}<1$ and $ew_{\bullet}\in \ZZ$, equation \eqref{eq:len} gives
\[
\dd\left(\cH\right) = \nu h x_{*} \quad \quad \text{and} \quad \quad \mom_{x_{\bullet}}(\cH)=-\nu h w_{\bullet}. 
\]
The line bundle $\cH^{\otimes \nu}$ is an extension of  $N\cong \cO_{B}(\nu H)$ and, by the above, has degree zero on $C$. The tropical moment of $\cH^{\otimes \nu}$ is $-bw_{\bullet}$. Lemma~\ref{lem:moving-interval-2} allows us to change this moment by any integer while preserving the degree on $C$. This is crucial because, to verify the hypotheses of Lemma~\ref{lem:polygon-torus}, which we do later in Lemma~\ref{lem:magic-polytope-hypoth}, we need  $\mom_{x_{\bullet}}(\cN)$ to be positive, whereas $-bw_{\bullet}<0$.  We achieve this by first choosing $w_{\bullet}$ to obtain the desired moment modulo $\ZZ$, and then using a vertical twist to obtain the desired value in $\QQ$. 

\begin{proposition}\label{prop:find-the-moment}
    If $b\geq2$  and  $\alpha \in \QQ$, or $b=1$ and $\alpha\in \QQ\setminus\ZZ$, there exist an integer $e\geq3$ and a rational number $w_{\bullet}\in (0,1)$ with $ew_{\bullet}\in \ZZ$ such that the following holds. Construct the model $\pB$ using this value of $e$ as in Section~\ref{subsec:nodal-curve}, and set $C=C_{w_{\bullet}}$ and $x_{\bullet}=x(w_{\bullet})$. There exists a line bundle $\cN$ on $\pB$ whose generic fiber is canonically identified with $N=\cO_{B}(\nu H)$ such that 
        \[
        \deg\left(\cN|_{C}\right)=0 \quad\quad \text{and} \quad \quad \mom_{x_{\bullet}}(\cN)=\alpha. 
        \]  
\end{proposition}

\begin{proof}[Proof of Proposition~\ref{prop:find-the-moment}]
Since $b=\nu^2h$ is a positive integer the interval $(\alpha,\alpha+b)$ contains an integer whenever $b\geq2$ or $\alpha\not\in\ZZ$. Choose an integer $c$ in this interval, and set
\[
w_{\bullet}\coloneqq \frac{c-\alpha}{b} \in \QQ\cap (0,1).
\]
Choose an integer $e\geq 3$ such that $ew_{\bullet}\in \ZZ$, and construct the model $\pB$ following Section~\ref{subsec:nodal-curve} using this $e$. The line bundle $\cH^{\otimes \nu}$ has degree zero on $C$ and tropical moment $-bw_{\bullet}=\alpha-c$. Let $C_+$ and $C_-$ be the components meeting $C$ at $0$ and $\infty$, respectively. Since $e\geq3$ these are distinct. Let $\cV \coloneqq cC_{+} - cC_{-}$ and $\cN=\cH^{\otimes \nu}(\cV)$. The calculation in the proof of part (2) of Lemma~\ref{lem:moving-interval-2} shows that $\deg(\cO_{\pB}(\cV)|_{C}) = 0$ and $\mom_{x_{\bullet}}(\cO_{\pB}(\cV))=c$. Thus, $\deg(\cN|_{C})=0$, and additivity of the tropical moment gives
\[
 \mom_{x_{\bullet}}(\cN) = \nu\mom_{x_{\bullet}}(\cH)+\mom_{x_{\bullet}}\left(\cO_{\pB}(\cV)\right)  = -bw_{\bullet} + c = \alpha.
\]
 Since $\cV$ is a vertical divisor, the canonical rational section of $\cO_{\pB}(\cV)$ restricts to a regular nowhere vanishing section on the generic fiber. This supplies the canonical identification of $\cN|_{B}\cong N$. 
\end{proof}

%%%%%%%%%%%%%%%%%%%%%%%%%%%%%%%%%%%%%%%%%%%%%%%%%%%%%%%%%%%%%%%%
\subsection{The Initial Linear Series and Its Support Diagram}
%%%%%%%%%%%%%%%%%%%%%%%%%%%%%%%%%%%%%%%%%%%%%%%%%%%%%%%%%%%%%%%%

Fix a rational number $\alpha$ satisfying the hypotheses of Proposition~\ref{prop:find-the-moment}, and fix the data produced by that proposition. We now describe the initial linear series on $X_{K}$ and bound its support diagram by a triangle. 

%%%%%%%%%%%%%%%%%%%%%%%%%%%%%%%%%%%%%%%%%%%%%%%%%%%%%%%%%%%%%%%%
\subsubsection{The Initial Linear Series}
%%%%%%%%%%%%%%%%%%%%%%%%%%%%%%%%%%%%%%%%%%%%%%%%%%%%%%%%%%%%%%%%

Continuing our abuse of notation, let $H$ denote the pullback of $H$ to $X_{K}$. Fix integers $d,m\geq1$ such that $\nu m <d$ and set
\[
L= dH, \quad \quad V = H^0(X_K, L), \quad \quad L_{B} = L|_{B}, \quad \text{and} \quad \kappa = \nu h \frac{d}{m}. 
\]
Recall from Section~\ref{sec:specialization-normal} that the initial linear series of $V$ is $V^{\inn}_{B}=\bigoplus_{j\geq0} \ev_{j}(V_{B,j})$, where 
\[
V_{B,j} = \img\left( \begin{tikzcd}[cramped, row sep = .2em, column sep = 3.5em]
H^0\left(X_{K}, \cO_{X_{K}}\left((d-\nu j)H\right)\right) \arrow[r, "\rho_j"] & H^0\left(B, L|_{B} \otimes N^{-j}\right) \cong H^0\left(B, \cO_{B}\left((d-\nu j)H\right)\right)
\end{tikzcd}\right).
\]
The degree of $L_{B}\otimes N^{-j}$ is  $\nu dh  - bj$. If $j>d/\nu$ this is negative, so the target of the above map vanishes. Thus, $V_{B,j}=0$ for $j>d/\nu$. 

%%%%%%%%%%%%%%%%%%%%%%%%%%%%%%%%%%%%%%%%%%%%%%%%%%%%%%%%%%%%%%%%%
\subsubsection{Extending the Line Bundles}\label{subsubsec:model-extensions}
%%%%%%%%%%%%%%%%%%%%%%%%%%%%%%%%%%%%%%%%%%%%%%%%%%%%%%%%%%%%%%%%

We now fix extensions of $L_B$ and $N$ to the model $\pB$. Proposition~\ref{prop:find-the-moment} supplies the extension $\cN$ of $N_{B/X_K}$. Continuing to let  $\cH$ denote the pullback of $H$ to $\pB$, set $\cL_{B} \coloneqq \cH^{\otimes d}$. This extends $L_{B}$ since $\cL_{B}|_B \cong \cO_{B}(dH)$. By construction, the component $C$ is contracted to the node $o$ under the map $\pB\to X$, so $\cH|_{C}$ is trivial. Thus, both $\cL_B$ and $\cN$ have degree zero on $C$.  Choose the compatible trivializations from Section~\ref{sec:inital-terms}.

We will apply Proposition~\ref{prop:shear-set} when $Q_j=0$ for every $j$. Set
\[
\psi \coloneqq \mom_{x_{\bullet}}(\cL_{B}) = d\mom_{x_{\bullet}}(\cH) = -d\nu h w_{\bullet}.
\]
Then $\beta_j = \psi - j\alpha$ and $\deg(\sD_j) = \nu hd - jb$. Combining these calculations with Proposition~\ref{prop:shear-set} gives $\SI(V)\subset \sS$ for the following finite set
\begin{equation}\label{eq:magic-set}
    \sS = \left\{
(j,k)\in \ZZ^{2} \;\; \bigg| \;\; \begin{gathered}
    0 \leq j \leq \frac{d}{\nu} \\
    \psi-j\alpha \; \leq\; k \; \leq \; \psi -j\alpha +  \nu hd -jb
\end{gathered}
\right\}.
\end{equation}

%%%%%%%%%%%%%%%%%%%%%%%%%%%%%%%%%%%%%%%%%%%%%%%%%%%%%%%%%%%%%%%%
\subsubsection{The Row Profile and Triangle Containment}\label{subsec:polygon}
%%%%%%%%%%%%%%%%%%%%%%%%%%%%%%%%%%%%%%%%%%%%%%%%%%%%%%%%%%%%%%%%

From now on we look for sections of $L=dH$ that vanish to order $m$ at $r$ general points of $X_{K}$. In order to apply Corollary~\ref{cor:trop-vanishing} we must check that there exist distinct points $\xi_1,\ldots,\xi_r\in \KK^{\times}$ such that the evaluation map
    \[
    \begin{tikzcd}[row sep = .2em, column sep = 3.5em]
    W_{\sS} \arrow[r] & \displaystyle \bigoplus_{i=1}^{r}\left(\cO_{\TT,(1,\xi_i)}\big/\mfm_{(1,\xi_i)}^{m}\right)
    \end{tikzcd}
    \]
is injective. We do not do this directly; instead, we check the stronger, but combinatorial, condition that $\sL(\sS) \subset \sE(m^{r})$ in order to apply Corollary~\ref{cor:the-vanishing-we-use}. To do this we will place $\sS$ inside a dilated triangle and apply Lemma~\ref{lem:polygon-torus} with $n=r$. Consider the following polygon:
\[
\sP_{\kappa,\alpha} \coloneqq\left\{
(x,y)\in \RR^{2}\;\; \mid \;\;
\begin{gathered}
 x\geq0, \\
 -\alpha x \leq y \leq \kappa-(\alpha+b)x
 \end{gathered}
\right\}.
\]
To see where this polygon comes from, apply the affine map $(j,k) \mapsto (j/m, (k-\psi)/m)$ to the exponent set $\sS$. Under this transformation, the $j$-th column lies in the vertical interval at $x=j/m$ with lower endpoint $y=-\alpha x$ and length $\kappa - bx$. This is made precise in the following lemma.

\begin{lemma}\label{lem:weird-containment}
	With $\kappa=\nu hd/m$ and $\psi$ as above,
	\[
	\sS \subset (0,\psi)+m\sP_{\kappa,\alpha}.
	\]
\end{lemma}

\begin{proof}
	Let $(j,k) \in \sS$ and let $x=j/m$ and $y=(k-\psi)/m$. The inequalities defining $\sS$ give that $x\geq0$ and  $-\alpha x \leq y \leq \frac{\nu hd}{m} - (\alpha+b)x$. Thus, $(x,y) \in \sP_{\kappa,\alpha}$ proving the claim since $(j,k)=(0,\psi)+m(x,y)$. 
\end{proof}

\begin{remark}\label{rem:NO}
	Consider $D=(d/m)H$ as an $\RR$-divisor on $X_{\overline{K}}$, and fix a flag $\{p\}\subset B_{\overline{K}} \subset X_{\overline{K}}$. Since $B\sim \nu H$ and $\kappa/b=d/(\nu m)$ we have  $D-xB\sim (d/m-\nu x)H$, which is ample for $0\leq x < \kappa/b$. By \cite[Thm.~6.4]{lazarsfeldMustata09} the Newton--Okounkov body of $D$ with respect to this flag is
	\[
	\Delta(D) = \left\{ (x,y)\in \RR^2 \;\; \big| \;\; 0 \leq x\leq \frac{\kappa}{b}, \quad 0\leq y \leq \kappa-bx \right\}.
	\]
	Thus, $\sP_{\kappa,\alpha}$ is the image of $\Delta(D)$ under the shear $(x,y)\mapsto (x,y-\alpha x)$. The shape of the triangle is therefore classical. The additional input here is a bound on the actual support of the initial terms in a translate of $m\sP_{\kappa,\alpha}$, together with the freedom to prescribe $\alpha$. It would be interesting to know if a more direct proof via the classical theory of Newton--Okounkov bodies is possible. 
\end{remark}

\begin{lemma}\label{lem:poly-widths}
	Let $\kappa>0$ and $\alpha>0$. The polygon $\sP_{\kappa,\alpha}$ is a triangle with vertices
	\[
	\left(0,0\right), \quad \quad\quad \left(0,\kappa\right), \quad \quad \quad\left(\frac{\kappa}{b}, \frac{-\alpha \kappa}{b}\right),
	\]
	with the following properties:
	\begin{enumerate}
	\item $\sI_{\sP_{\kappa,\alpha}} = \left[\frac{-\alpha\kappa}{b},\; \kappa\right] $, 
	\item $\omega_{\max}(\sP_{\kappa,\alpha}) = \frac{\kappa}{\alpha+b}$, which occurs uniquely when $y=0$, and
	\item $N_{\sP_{\kappa,\alpha}}(\sigma) = \frac{\kappa(\alpha+b)}{b}-\frac{(\alpha+b)^2}{b}\sigma$ for $0 \leq \sigma \leq \frac{\kappa}{\alpha+b}$.
	\end{enumerate}
\end{lemma}	

\begin{proof}
	Intersecting the three lines defining the boundary of $\sP_{\kappa,\alpha}$ gives the claimed vertices. The minimum and maximum heights occur at these vertices giving the claimed interval for $\sI_{\sP_{\kappa,\alpha}}$. For part (2) note that at any height $y \in  \sI_{\sP_{\kappa,\alpha}}$ the inequalities $x\geq0$ and $-\alpha x \leq y$ are equivalent to $x\geq (-y/\alpha)_{+}$. Thus, this gives the left endpoint of the horizontal slice $\sP_{\kappa,\alpha} \cap (\RR\times \{y\})$. Similarly, since $\alpha+b>0$ the remaining inequality $y\leq \kappa-(\alpha+b)x$ is equivalent to $x\leq (\kappa-y)/(\alpha+b)$. This gives the right endpoint of the horizontal slice at height $y$. Subtracting these two gives
\[
\omega_{\sP_{\kappa,\alpha}}(y) = 
\begin{cases}
	\frac{\kappa-y}{\alpha+b} & 0\leq y \leq \kappa, \\
	\frac{\kappa}{\alpha+b} + \frac{by}{\alpha(\alpha+b)}  & -\alpha\kappa/b \leq y \leq 0
\end{cases}.
\]
Maximizing this function gives the claimed $\omega_{\max}(\sP_{\kappa,\alpha})$ at $y=0$. By the above formula, if $y\ge0$ then the slice at height $y$ has width at least $\sigma$ if and only if $y \leq \kappa - (\alpha+b)\sigma$. When $y\leq0$ the same condition becomes $y\geq \frac{-\alpha}{b}(\kappa-(\alpha+b)\sigma)$. Thus, the set of heights whose slice has length at least $\sigma$ is an interval. Subtracting these two values and simplifying gives the claimed formula for $N_{\sP_{\kappa,\alpha}}(\sigma)$. 
\end{proof}

\begin{lemma}\label{lem:magic-polytope-hypoth}
	Let $\kappa\in \RR$ satisfy $b<\kappa \leq \kappa^2/b<r$. Then the interval
	\[
	0<\kappa-b<\alpha <\frac{br}{\kappa}-b
	\]
	is nonempty. For each such $\alpha$, conditions (1) and (2) in Lemma~\ref{lem:polygon-torus} hold for $\sP_{\kappa,\alpha}$ with
	\[
	\Gap = \min\left\{ r- \frac{\kappa(\alpha+b)}{b}, \; \; r\left(1-\frac{\kappa}{\alpha+b}\right)\right\}>0.
	\]
\end{lemma}

\begin{proof}
	Since $\kappa^2/b<r$ we have $\kappa < br/\kappa$ implying that the interval is nonempty. Its left endpoint is positive since $b<\kappa$. For $\alpha$ in this interval, Lemma~\ref{lem:poly-widths} gives 
	\[
	\omega_{\max}(\sP_{\kappa,\alpha}) = \frac{\kappa}{\alpha+b} < 1
	\]
	where the strict inequality follows from $\alpha>\kappa-b$.  By part (3) of the same lemma the difference $r(1-\sigma)-N_{\sP_{\kappa,\alpha}}(\sigma)$ is an affine function on $[0,\kappa/(\alpha+b)]$.  The values of this function at the endpoints of $[0,\kappa/(\alpha+b)]$ are the two terms defining $\Gap$. Since an affine function is greater than or equal to the minimum of its values at the endpoints this verifies condition (2) of Lemma~\ref{lem:polygon-torus}. Finally, $\alpha < br/\kappa - b$ implies that $\kappa(\alpha+b)/b < r$ implying the first term defining $\Gap$ is positive. The second term defining $\Gap$ is positive since $r>0$ and $\alpha+b>\kappa$. Thus, $\Gap>0$.
\end{proof}

\begin{remark}
Our goal is to show that $\sL(\sS) \subset \sE(m^r)$ in order to apply Corollary~\ref{cor:the-vanishing-we-use}. This amounts to showing that $N_{\sS}(p)\leq r(m-p)_{+}$ for all $p\in\ZZ_{\geq0}$. Lemma~\ref{lem:polygon-torus} shows that if the continuous row profile satisfies  $\omega_{\max}(\sP_{\kappa,\alpha})<1$ and there is a uniform gap $\Gap$ such that $N_{\sP_{\kappa,\alpha}}(\sigma) \leq r(1-\sigma)-\Gap$ for $0\leq \sigma \leq \omega_{\max}(\sP_{\kappa,\alpha})$ then the desired inequality on the discrete row profile holds provided $m\Gap\geq1$.  The shear $\alpha$ must satisfy two competing constraints. First, $\sE(m^r)$ has no columns for $p\geq m$, and so no row of $\sS$ may have more than $m$ points. Thus, it suffices to require $\omega_{\max}(\sP_{\kappa,\alpha}) = \kappa/(\alpha+b) <1$, i.e. $\alpha>\kappa-b$. Notice that this fails if we set the shear $\alpha=0$ since $\omega_{\max}(\sP_{\kappa,0}) = d/(\nu m)>1$. Second, since the longest column of $\sE(m^r)$, i.e., the first column, has $rm$ lattice points, $\sS$ can have at most $rm$ nonempty rows. To obtain a positive gap when $\sigma=0$, we require $N_{\sP_{\kappa,\alpha}}(0) = \kappa(\alpha+b)/b < r$, i.e., $\alpha < (br/\kappa)-b$. At the boundary value $\alpha = \kappa-b$ Lemma~\ref{lem:poly-widths} gives
\[
N_{\sP_{\kappa,\kappa-b}}(\sigma) = \frac{\kappa^2}{b}\left(1-\sigma\right) = h\left(\frac{d}{m}\right)^2\left(1-\sigma\right) \quad \quad \text{for $0\leq\sigma \leq 1$}.
\]
This has the same overall shape as the desired upper bound. Thus, the comparison reduces to $h(d/m)^2<r$, i.e., $hd^2<rm^2$. Notice this is precisely the comparison between the leading order terms of  $\dim H^0(X,dH)\approx hd^2/2$ and the $r\binom{m+1}{2} \approx rm^2/2$ conditions imposed by the points. 
\end{remark}

\begin{example}\label{exm:cubic}
	Let $X\subset \PP^3$ be a smooth cubic surface and let $H$ be the class of a hyperplane section, so that $h=3$. With $\nu=1$ we have that $b=3$; consider $r=7$ points. The pair $(d,m)=(3,2)$ lies in the crucial interval $1<3/2<\sqrt{7/3}\approx 1.527\ldots$. Then $\kappa=\nu hd/m = 9/2$ and Lemma~\ref{lem:magic-polytope-hypoth} lets us take any shear in the interval $3/2<\alpha<5/3$. Taking $\alpha=8/5$, we may choose $w_\bullet=7/15$, $e=15$, and $c=3$ in the proof of Proposition~\ref{prop:find-the-moment}. Thus, $\psi = -21/5$, and the set $\sS$ consists of 19 lattice points occurring in the columns $j=0,1,2,3$. Figure~\ref{fig:support-diagram} shows the support diagram inside $(0,\psi)+m\sP_{\kappa,\alpha}$ and its lex staircase together with a comparison of what we get when $\alpha=0$.
\end{example}

\begin{figure}
\centering
% Requires tikz and the paper's macros \sS, \sL, \sE.
% Existing paper colors take precedence; the fallbacks match its palette.
\begingroup
\providecolor{DartmouthGreen}{HTML}{00693E}
\providecolor{FigRed}{HTML}{A4343A}
\providecolor{FigGold}{HTML}{C28E0E}
\providecolor{FigGray}{HTML}{7A7A7A}

\begin{tikzpicture}[
  x=0.68cm, y=0.235cm, line cap=round, line join=round,
  every node/.style={text=black, inner sep=1.5pt},
  support subtitle/.style={font=\footnotesize},
  support tick/.style={font=\scriptsize},
  support grid/.style={draw=FigGray!22, line width=0.25pt},
  support edge/.style={draw=DartmouthGreen!70!black, line width=0.75pt},
  support point/.style={circle, fill=black, draw=white,
    line width=0.35pt, inner sep=0pt, minimum size=3.6pt},
  support axis/.style={draw=FigGray, line width=0.5pt, ->},
  ambient boundary/.style={draw=FigGray!35!black, line width=1.3pt,
    preaction={draw=white, line width=2.4pt}}
]

% Supports: positions 1 and 3, on identical coordinate scales.
\foreach \panelshift in {0,8.3} {
  \begin{scope}[xshift=\panelshift cm]
    \foreach \k in {-9,...,6}
      \draw[support grid] (-0.22,\k) -- (3.22,\k);
    \foreach \j in {0,...,3} {
      \draw[support grid] (\j,-9.45) -- (\j,6.45);
      \node[support tick, anchor=north] at (\j,-10.1) {$\j$};
    }
    \foreach \k in {-9,-5,0,5}
      \node[support tick, anchor=east] at (-0.36,\k) {$\k$};
    \node[support subtitle] at (-0.55,6.95) {$k$};
    \node[support subtitle] at (3.55,-10.35) {$j$};
  \end{scope}
}

% (a) Without shear: the four red rows have more than two points.
\begin{scope}
  
  \fill[DartmouthGreen!8] (0,-3) -- (0,6) -- (3,-3) -- cycle;
  \foreach \k in {-3,...,0}
    \fill[FigRed!18] (-0.22,\k-0.27) rectangle (3.22,\k+0.27);
  \draw[support edge] (0,-3) -- (0,6) -- (3,-3) -- cycle;
  \foreach \j/\lo/\hi in {0/-3/6,1/-3/3,2/-3/0,3/-3/-3}
    \foreach \k in {\lo,...,\hi}
      \node[support point] at (\j,\k) {};
\end{scope}

% (b) With shear: the six gold rows have exactly two points.
\begin{scope}[xshift=8.3cm]
  
  \fill[DartmouthGreen!8] (0,-4.2) -- (0,4.8) -- (3,-9) -- cycle;
  \foreach \k in {-5,...,0}
    \fill[FigGold!27] (-0.22,\k-0.27) rectangle (3.22,\k+0.27);
  \draw[support edge] (0,-4.2) -- (0,4.8) -- (3,-9) -- cycle;
  \foreach \j/\lo/\hi in {0/-4/4,1/-5/0,2/-7/-5,3/-9/-9}
    \foreach \k in {\lo,...,\hi}
      \node[support point] at (\j,\k) {};
\end{scope}

% Lex staircases: positions 2 and 4, on identical coordinate scales.
% Each cell represents (p,q), with its lower-left corner at that point.
% The ambient staircase E(2^7) has heights 14 and 7.
\foreach \panelshift in {4,12.3} {
  \begin{scope}[shift={(\panelshift cm,-2.35cm)}, x=0.45cm, y=0.25cm]
    \foreach \p/\h in {0/14,1/7}
      \foreach \q in {1,...,\h}
        \filldraw[fill=white, draw=FigGray!65, line width=0.35pt]
          (\p,\q-1) rectangle (\p+1,\q);
  \end{scope}
}

% Lex staircase of (a): heights 10,7,4,1.
% Four cells in p=2 and one in p=3 lie outside E(2^7).
\begin{scope}[shift={(4cm,-2.35cm)}, x=0.45cm, y=0.25cm]
  \foreach \p/\h in {0/10,1/7}
    \foreach \q in {1,...,\h}
      \filldraw[fill=DartmouthGreen!28, draw=white, line width=0.45pt]
        (\p,\q-1) rectangle (\p+1,\q);
  \foreach \p/\h in {2/4,3/1}
    \foreach \q in {1,...,\h}
      \filldraw[fill=FigRed!40, draw=white, line width=0.45pt]
        (\p,\q-1) rectangle (\p+1,\q);
  \draw[draw=DartmouthGreen!70!black, line width=0.65pt]
    (0,0) -- (2,0) -- (2,7) -- (1,7) -- (1,10) -- (0,10) -- cycle;
  \draw[draw=FigRed, line width=0.8pt]
    (2,0) -- (4,0) -- (4,1) -- (3,1) -- (3,4) -- (2,4) -- cycle;
  \foreach \x/\h in {1/10,2/7}
    \node[support tick, text=DartmouthGreen!70!black, anchor=west]
      at (\x+0.10,\h) {$\h$};
  \foreach \x/\h in {3/4,4/1}
    \node[support tick, text=FigRed, anchor=west]
      at (\x+0.10,\h) {$\h$};
\end{scope}

% Lex staircase of (b): heights 13,6, with two unused ambient cells.
\begin{scope}[shift={(12.3cm,-2.35cm)}, x=0.45cm, y=0.25cm]
  \foreach \p/\h in {0/13,1/6}
    \foreach \q in {1,...,\h}
      \filldraw[fill=DartmouthGreen!28, draw=white, line width=0.45pt]
        (\p,\q-1) rectangle (\p+1,\q);
  \draw[draw=DartmouthGreen!70!black, line width=0.65pt]
    (0,0) -- (2,0) -- (2,6) -- (1,6) -- (1,13) -- (0,13) -- cycle;
  \foreach \x/\h in {1/13,2/6}
    \node[support tick, text=DartmouthGreen!70!black, anchor=west]
      at (\x+0.10,\h) {$\h$};
\end{scope}

% Axes first, then a heavy closed ambient boundary with a white halo.
\foreach \panelshift in {4,12.3} {
  \begin{scope}[shift={(\panelshift cm,-2.35cm)}, x=0.45cm, y=0.25cm]
    \draw[support axis] (-0.13,0) -- (4.85,0)
      node[right, support subtitle] {$p$};
    \draw[support axis] (0,-0.15) -- (0,15.55)
      node[above, support subtitle] {$q$};
    \foreach \p in {0,1,2,3}
      \node[support tick, anchor=north] at (\p+0.5,-0.1) {$\p$};
    \foreach \q in {7,14} {
      \draw[draw=FigGray, line width=0.5pt] (-0.08,\q) -- (0,\q);
      \node[support tick, anchor=east] at (-0.12,\q) {$\q$};
    }
    % The complete perimeter is intentionally stronger than the cell edges.
    \draw[ambient boundary]
      (0,0) -- (0,14) -- (1,14) -- (1,7) -- (2,7) -- (2,0) -- cycle;
%    \draw[draw=FigGray!35!black, line width=0.5pt]
%      (1,11.5) -- (1.55,12) -- (1.8,12);
    \node[support tick, anchor=west] at (1.85,12) {$\sE(2^7)$};
  \end{scope}
}

% One parameter label centered below each support--staircase pair.
\node[support subtitle] at (3.01625cm,-2.95cm)
  {(a) $\alpha=0,\quad \psi=-3$};
\node[support subtitle] at (11.31625cm,-2.95cm)
  {(b) $\alpha=8/5,\quad \psi=-21/5$};
\end{tikzpicture}
\endgroup

\caption{Support diagrams and their lex staircases for Example~\ref{exm:cubic} with $\kappa = 9/2$ and $m=2$. Left:  for $\alpha=0$ and $\psi=-3$, the staircase has columns of height $10,7,4,1$ and is not contained in $\sE(2^7)$. Right: for $\alpha=8/5$ and $\psi=-21/5$ the staircase has columns of height $13$ and $6$ and fits into $\sE(2^7)$. }
\label{fig:support-diagram}
\end{figure}
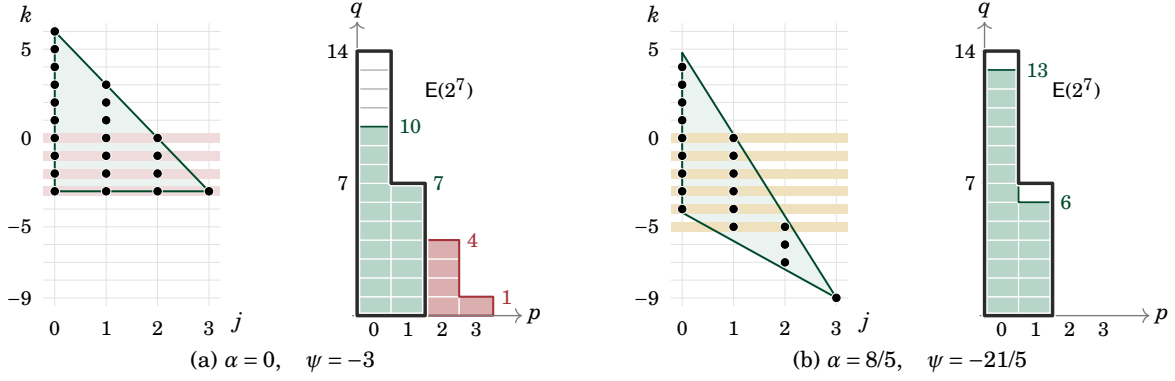
%%%%%%%%%%%%%%%%%%%%%%%%%%%%%%%%%%%%%%%%%%%%%%%%%%%%%%%%%%%%%%%%
\subsection{Vanishing in the Crucial Interval}
%%%%%%%%%%%%%%%%%%%%%%%%%%%%%%%%%%%%%%%%%%%%%%%%%%%%%%%%%%%%%%%%

We now prove the vanishing statement originally labeled as \textbf{(V0)} for all pairs of positive integers $(d,m)$ in the crucial interval. 

\begin{proposition}\label{prop:vanishing-in-interval}
	Under the hypotheses of Theorem~\ref{main-thm:technical-nagata}, let $r>b\coloneqq \nu^2h$. If $d,m$ are positive integers with
	\[
	\nu < \frac{d}{m} < \sqrt{\frac{r}{h}}
	\]
	then there exists a proper Zariski closed subset $\Bad_{d,m}\subset \Conf_{r}(X)$ such that 
	\[
	H^0\left(X_{\pp}, \cO_{X_{\pp}}(dH-mE)\right)=0
	\]
	for all $\pp\in \Conf_{r}(X)\setminus \Bad_{d,m}$.
\end{proposition}

\begin{proof}
    \textit{Choosing the Shear.} Fix positive integers $d,m$ such that $d/m$ is in the crucial interval. Set $\kappa = \nu h(d/m)>0$. Since $b=\nu^2h$ and $\nu < d/m$, we have $b < \kappa$. Similarly, since $d/m < \sqrt{r/h}$ we have $\kappa^2/b < r$. Applying Lemma~\ref{lem:magic-polytope-hypoth}, choose a non-integer rational number
    \[
   \kappa - b < \alpha < \frac{br}{\kappa}-b
    \]   
    and let $\Gap$ be the resulting gap for $\sP_{\kappa, \alpha}$. 
    
    \textit{Scaling.} Choose a positive integer $k$ such that $km\Gap\geq1$. It suffices to prove the vanishing for the pair $(kd,km)$;  taking the $k$-th power of a nonzero section of $dH-mE$ produces a nonzero section of $kdH-kmE$. Thus if we construct $\Bad_{kd,km}$ we may define $\Bad_{d,m}\coloneqq \Bad_{kd,km}$.  Further, the ratio $d/m$, and hence $\kappa$, is unchanged by this rescaling, so we may keep the same $\alpha$ and $\Gap$. Replacing $(d,m)$ with $(kd,km)$, we may therefore assume that $m\Gap\geq1$.  
        
    \textit{Vanishing over $\overline{K}$.} Construct $e$, $\pB$, $C$, and $\cN$ by Proposition~\ref{prop:find-the-moment} for our chosen parameter $\alpha$. Let $L=dH$ on $X_K$ and $V=H^0(X_K, L)$. Use the extensions $\cL_{B}$ and $\cN$ constructed in Section~\ref{subsubsec:model-extensions} with trivial base divisors.  With these choices Proposition~\ref{prop:shear-set} shows that the set $\sS$ of \eqref{eq:magic-set} contains $\SI(V)$. By Lemma~\ref{lem:weird-containment} we have the following containment:
    \[
    \sS \subset (0,\psi)+m\sP_{\kappa,\alpha}.
    \]
   By Lemma~\ref{lem:magic-polytope-hypoth}, the triangle $\sP_{\kappa,\alpha}$ satisfies conditions (1) and (2) of  Lemma~\ref{lem:polygon-torus} with $n=r$. Our scaling choice of $k$ ensures that condition (3),  $m\Gap\geq1$, holds, and so Lemma~\ref{lem:polygon-torus} gives $\sL(\sS)\subset \sE(m^r)$. Corollary~\ref{cor:the-vanishing-we-use}  then gives a nonempty Zariski open subset $U\subset \Conf_{r}(X_{\overline{K}})$ such that for every $\pp \in U$
    \[
    H^0\left(X_{\overline{K}}, \cO_{X_{\overline{K}}}(dH)\otimes \cI_{Z_{\pp}(m^r)}\right)=0.
    \]

      \textit{Descending to $\KK$.} Now apply Lemma~\ref{lem:field-descent} to the surface $X$, the line bundle $dH$, the multiplicities $m^r$, and the field extension $\KK\subset \overline{K}$. This gives a proper Zariski closed subset $\Bad_{d,m} \subset \Conf_{r}(X)$ such that $H^0(X, \cO_{X}(dH)\otimes \cI_{Z_{\pp}(m^r)})=0$ for all $\pp\in \Conf_r(X)\setminus \Bad_{d,m}$. The claim follows from the standard isomorphism $H^0(X_{\pp},dH-mE) \cong H^0(X, \cO_{X}(dH)\otimes \cI_{Z_{\pp}(m^r)})$.
\end{proof}

%%%%%%%%%%%%%%%%%%%%%%%%%%%%%%%%%%%%%%%%%%%%%%%%%%%%%%%%%%%%%%%%
\subsection{Completing the Proofs}
%%%%%%%%%%%%%%%%%%%%%%%%%%%%%%%%%%%%%%%%%%%%%%%%%%%%%%%%%%%%%%%%

\begin{lemma}\label{lem:nagata-to-seshadri}
    Let $\pp\in \Conf_{r}(X)$ and $\lambda \in \RR_{>0}$. The following are equivalent:
    \begin{enumerate}
        \item The $\RR$-divisor $H-\lambda E$ is nef on $X_{\pp}$.
        \item For every curve $C\subset X$ we have $\sum_{i=1}^{r}\mult_{p_i}(C) \leq \frac{1}{\lambda} \langle H\cdot C\rangle$.
    \end{enumerate}
\end{lemma}

\begin{proof}
    That (1) $\implies$ (2) follows from the fact that if  $C\subset X$ is a curve with strict transform $\tilde{C}\subset X_{\pp}$ then
    \[
    \left\langle (H-\lambda E) \cdot \tilde{C}\right\rangle = \langle H\cdot C\rangle -\lambda \sum_{i=1}^{r} \mult_{p_i}(C).
    \]
    See \cite[Prop.~V.3.6]{hartshorne77}. Thus, if $H-\lambda E$ is nef the right-hand side is nonnegative, and rearranging gives (2).
    For (2) $\implies$ (1) note that every integral curve on $X_{\pp}$ is either $E_{i}$ for some $i=1,\ldots,r$ or the strict transform $\tilde{C}$ of an integral curve $C\subset X$. Thus, it suffices to check the intersection of $H-\lambda E$ with curves of each of these types. For strict transforms this follows from the displayed equation above. For $E_i$ we have that $\langle (H-\lambda E)\cdot E_i\rangle = \lambda>0$.
\end{proof}

Before proving Theorem~\ref{main-thm:technical-nagata} we need the following lemma, which handles the endpoint $r=\nu^2h$ not covered by our tropical methods. This argument follows the classical idea behind the non-strict square case of Nagata's theorem \cite{nagata59}. Namely, we specialize points onto an integral curve in $|\nu H|$ and then use the parameter spaces for divisors to pass to a locus of very general points. This lemma is proved in the case $X=\PP^2$ in \cite[p.~417]{strycharzSzembergSzemberg04} and \cite[Lem.~II.1]{harbourne01}, and when $H$ is very ample in \cite[\S~II]{harbourne03}. The proof below follows the same idea.

\begin{lemma}\label{lem:endpoints}
 	Suppose $|\nu H|$ contains an integral curve and $r=\nu^2 h$. For a very general point $\pp \in \Conf_{r}(X)$
	\[
	\varepsilon(X,H;\pp)=\sqrt{\frac{h}{r}}.
	\] 
\end{lemma}

\begin{proof}
	Set $\lambda =1/\nu= \sqrt{h/r}$. It is enough to show that $H-\lambda E$ is nef on $X_{\pp}$ for a very general point $\pp \in \Conf_{r}(X)$. Choose an integral curve $B_0 \in |\nu H|$ and let $\qq=(q_1,\ldots,q_r)$ be $r$ distinct smooth points on $B_0$. Consider the blow-up $X_{\qq}\to X$ with $\tilde{B}_{0}\in |\nu H-E|$ the strict transform of $B_0$. Since $B_0$ has multiplicity one at each $q_i$ we have that $\langle \tilde{B}_0^2\rangle = \nu^2h-r= 0$. Every other integral curve meets $\tilde{B}_0$ nonnegatively implying that $\tilde{B}_{0}\sim \nu H-E$ is nef. This in turn implies that $H-\lambda E$ is nef on $X_{\qq}$.
	
	For an integer $k\geq1$ and $\mm\in \ZZ^r_{\geq0}$, let $\Sigma_{k,\mm} \subset \Conf_{r}(X)$ be the set of $\pp \in \Conf_{r}(X)$ for which there exists an effective divisor $D\subset X$ with $\langle H\cdot D\rangle = k$ and $\mult_{p_i}(D)\geq m_i$ for all $i$. We claim that $\Sigma_{k,\mm}$ is closed. This is presumably well known, but I could not find a precise reference, and so we include a sketch of the proof for completeness. 
	
	The set of effective divisors $D\subset X$ with $\langle H\cdot D\rangle  = k$ is bounded, i.e., only finitely many Hilbert polynomials occur among all such $D$.  In particular, since $H$ is ample, Kleiman's criterion \cite{kleiman66}, see also \cite[Thm.~1.2.23]{lazarsfeld04-I}, implies that $\{\alpha \in \overline{\NE}(X) \; | \; \langle H\cdot\alpha\rangle = k\}$ is compact. Thus, it contains only finitely many integral numerical classes. By Riemann--Roch
	\[
	\chi\left(X, \cO_{D}(nH)\right) = n\langle H\cdot D\rangle - \frac{1}{2}\langle D\cdot (D+K_X)\rangle,
	\]
	meaning the Hilbert polynomial of $D$ only depends on the numerical class of $D$. For each such polynomial $P$ there is an open subscheme $\Div_{X}^{P} \subset \Hilb_{X}^{P}$ representing effective divisors on $X$ with Hilbert polynomial $P$ \cite[Thm.~9.3.7]{foag}. Since $X$ is smooth, if $R$ is a discrete valuation ring with fraction field $\FF$, taking scheme-theoretic closures uniquely extends effective divisors on $X_{\FF}$ to $X_{R}$. The resulting divisors are effective Cartier divisors that are flat over $R$, and so their Hilbert polynomial is unchanged. Thus, the valuative criterion shows that $\Div_{X}^{P}$ is proper. So $\Div_{X}^{P}$ is projective since $\Hilb_{X}^{P}$ is projective \cite[Thm.~5.14]{foag}. Let $\pD_{k}$ be the finite disjoint union of these schemes, with $\tilde{D} \subset \pD_k \times X$ the universal divisor. 
	Let $\tilde{\Sigma}_{k,\mm} \subset \pD_{k}\times \Conf_{r}(X)$ be the locus of pairs $(D,\pp)$ with $\mult_{p_i}(D) \geq m_i$ for all $i$. Using the same construction as in Lemma~\ref{lem:uni-fat-bundle}, the jets of the defining section of the universal divisor give a section of a vector bundle on $\pD_{k}\times \Conf_{r}(X)$. Thus, the zero locus of this section, $\tilde{\Sigma}_{k,\mm}$, is closed. Since $\pD_{k}$ is proper, the projection $\pD_{k}\times \Conf_{r}(X)\to \Conf_{r}(X)$ is proper. This proves that $\Sigma_{k,\mm}$ is closed since it is the image of the closed set $\tilde{\Sigma}_{k,\mm}$. 
	 Let $W$ be the union of those $\Sigma_{k,\mm}$ such that $\sum_{i} m_i >\nu k$ and $\Sigma_{k,\mm}$ is a proper subset of $\Conf_{r}(X)$. Let $\pp \not \in W$. Towards a contradiction suppose that $H-\lambda E$ is not nef on $X_{\pp}$. By Lemma~\ref{lem:nagata-to-seshadri} there exists a curve $C \subset X$ with $\sum_{i} \mult_{p_i}(C) > \nu \langle H\cdot C\rangle$. Setting $k=\langle H \cdot C
	\rangle$ and $m_i = \mult_{p_i}(C)$ we have that $\pp \in \Sigma_{k,\mm}$. Since  $\pp \not \in W$ this forces $\Sigma_{k,\mm}=\Conf_{r}(X)$. Thus, $\qq\in \Sigma_{k,\mm}$ and so there exists an effective divisor $D$ on $X$ such that $\langle H\cdot D \rangle = k$ and $\mult_{q_i}(D)\geq m_i$. Considering the strict transform $\tilde{D}\subset X_{\qq}$ of $D$ we have that
	\[
	\left\langle (H-\lambda E) \cdot \tilde{D}\right\rangle =k-\lambda\sum_{i=1}^{r}\mult_{q_i}(D)\leq k-\lambda \sum_{i=1}^r m_i < 0.
	\]
	 The last inequality comes from the fact that $\sum_{i} m_i >\nu k$ and $\lambda=1/\nu$. Since $\tilde{D}$ is effective this  contradicts the fact that $H-\lambda E$ is nef on $X_{\qq}$. 
	
	Thus, for every $\pp \in \Conf_{r}(X)\setminus W$ we have that $H-\lambda E$ is nef on $X_{\pp}$. As $W$ is a countable union of proper closed sets this, combined with the standard upper bound, gives the desired equality for very general points. 
\end{proof}

\begin{proof}[Proof of Theorem~\ref{main-thm:technical-nagata}]
Lemma~\ref{lem:endpoints} proves the claim when $r=\nu^2h$, so assume $r>\nu^2h$. For each pair of positive integers $(d,m)$ with  $\nu < d/m < \sqrt{r/h}$ Proposition~\ref{prop:vanishing-in-interval} produces a proper closed subset $\Bad_{d,m}\subset \Conf_{r}(X)$ such that  $H^0(X_{\pp}, dH-mE)=0$ for all $\pp \in\Conf_{r}(X)\setminus \Bad_{d,m}$. Let $\Bad_{r}\coloneqq \bigcup_{(d,m)} \Bad_{d,m} \subset \Conf_{r}(X)$ where the union runs over all pairs as above. Since $\Conf_{r}(X)$ is irreducible and $\KK$ is uncountable, the complement $\Conf_{r}(X)\setminus \Bad_r$ has $\KK$-points. Every point $\pp$ in this complement satisfies the hypotheses of Lemma~\ref{lem:vanishing-to-nefness}, and Corollary~\ref{cor:vanishing-to-nagata} gives the desired claim.
\end{proof}

\begin{proof}[Proof of Theorem~\ref{main-thm:general-nagata}]
    Set $A=kH$. Since $A$ is very ample and  $(X,A)\neq (\PP^2,\cO_{\PP^2}(1))$ we have that $\langle A^2\rangle = k^2h\geq2$ (see Remark~\ref{rem:degree-one}). Thus, Lemma~\ref{lem:2H-nodal-curve-exists}, applied to $A$, shows there exists an integral curve $B_0\in |2A|$ such that $B_0$ has exactly one ordinary node. The linear series $|2A|$ is basepoint free, and $(2k)^2h=4k^2h$. Applying Theorem~\ref{main-thm:technical-nagata} to $(X,H)$ with $\nu=2k$ gives the claim for $r\geq 4k^2 h$.
\end{proof}

In order to prove the strict inequality for Theorem~\ref{main-thm:nagata-I} we will rely on Nagata's original theorem, which concerns the case when $r$ is a perfect square and the multiplicities are uniform at each point. The following lemma, which has appeared in \cite[Rem.~2.6]{strycharzSzembergSzemberg04} and \cite[\S~1.1]{nivoche21}, shows that such a bound in this uniform case implies the corresponding bound for all multiplicities. 

\begin{lemma}\label{lem:sym}
	Fix an integer $r\geq 1$ and a real number $\lambda>0$.  Suppose that, for a very general point $\pp = (p_1,\ldots,p_r)\in \Conf_{r}(\PP^2)$, every plane curve $C\subset \PP^2$ satisfies:
	\begin{equation}\label{eq:sym-gen}
	\lambda \min\left\{\mult_{p_i}(C) \;\;\big|\;\; 1\leq i \leq r\right\} <\deg(C).
	\end{equation}
	Then for a very general point $(q_1,\ldots,q_r)\in \Conf_{r}(\PP^2)$  every plane curve $C\subset \PP^2$  of positive degree satisfies:
	\[
	\sum_{i=1}^r \mult_{q_i}(C) < \frac{r}{\lambda} \deg(C).
	\]
\end{lemma}

\begin{proof}
Let $\Sigma_0\subset \Conf_{r}(\PP^2)$ be a countable union of proper closed subsets such that every point outside of $\Sigma_{0}$ satisfies \eqref{eq:sym-gen} for all plane curves. For $d\geq1$ and $\mm \in \ZZ^{r}_{\geq0}$ let $\Sigma_{d,\mm} \subset \Conf_{r}(\PP^2)$ be the locus where there is a curve of degree $d$ vanishing at $p_i$ with multiplicity at least $m_i$ for all $i$. The jet evaluation argument in the proof of Lemma~\ref{lem:field-descent} shows that $\Sigma_{d,\mm}$ is closed, although it need not be a proper subset of $\Conf_{r}(\PP^2)$. Let $\Sigma$ be the union of $\Sigma_0$ and the $\Sigma_{d,\mm}$ that are proper subsets of $\Conf_{r}(\PP^2)$. This is a countable union of proper closed subsets. 

    Fix a point $\pp \in \Conf_{r}(\PP^2)\setminus \Sigma$ and suppose that there exists a curve $C\subset \PP^2$ of degree $d$ such that $\mult_{p_i}(C)=m_i$. By definition this means $\pp\in \Sigma_{d,\mm}$, but since $\pp\not\in \Sigma$ this means $\Sigma_{d,\mm} = \Conf_{r}(\PP^2)$. For a permutation $\sigma\in S_{r}$ let $\sigma(\mm) = (m_{\sigma(1)},\ldots,m_{\sigma(r)})$. The automorphism of $\Conf_{r}(\PP^2)$ induced by the permutation $\sigma$ sends $\Sigma_{d,\mm}$ to $\Sigma_{d,\sigma(\mm)}$. Since $\Sigma_{d,\mm} = \Conf_{r}(\PP^2)$ this means that $\Sigma_{d,\sigma(\mm)}=\Conf_{r}(\PP^2)$ for all $\sigma \in S_r$. Hence for each $\sigma \in S_r$ there exists a nonzero homogeneous polynomial $f_{\sigma}\in \KK[x,y,z]$ of degree $d$ with $\mult_{p_i}(f_{\sigma}) \geq m_{\sigma(i)}$ for all $i$. Let $F \coloneqq \prod_{\sigma \in S_{r}} f_{\sigma}$, which is a nonzero homogeneous polynomial of degree $(r!)d$. Since multiplicities are additive for each point $p_i$ we have that
    \[
    (r-1)!\sum_{j=1}^{r}m_j = \sum_{\sigma\in S_r} m_{\sigma(i)}\leq \mult_{p_i}(\VV(F)).
    \]
    Since $\pp\not\in \Sigma_{0}$ by combining this with \eqref{eq:sym-gen} we get that
    \[
  \lambda (r-1)!\sum_{j=1}^{r}m_j  \leq \lambda \min\left\{\mult_{p_i}(\VV(F)) \;\;\big|\;\; 1\leq i \leq r\right\} <\deg(\VV(F)) = (r!)d.
    \]
    Dividing by $\lambda(r-1)!$ gives the desired inequality.
\end{proof}

\begin{proof}[Proof of Theorem~\ref{main-thm:nagata-I}]
    Let $H$ be the class of a line in $\PP^2$, so that $h=1$, and set $\nu=3$. Then $\nu^2h=9$ and $|3H|$ is basepoint free. Let $B_0 \in |3H|$ be the curve given by $x^2(x+z)-y^2z=0$. Direct computation shows that $B_0$ is integral, and its only singular point is  $[0:0:1]$ where it has an ordinary node. Theorem~\ref{main-thm:technical-nagata} proves the non-strict inequality for $r\geq10$.

    It remains to show the inequality is strict for $r\geq10$. If $r$ is not a perfect square this follows from the fact that the left-hand side is an integer, but $d\sqrt{r}$ is irrational. Assume $r=n^2$ for an integer $n\geq4$. Nagata's equal-multiplicity theorem \cite{nagata59} shows that outside of a subset $\Sigma_{\mathrm{Nag}}\subset \Conf_{n^2}(\PP^2)$, which is a countable union of proper closed subsets, no plane curve of degree $d$ has multiplicity at least $m$ at every point when $d \leq nm$. Thus, every plane curve $C$ satisfies
    \[
    n\min\{ \mult_{p_i}(C) \; | \; 1\leq i \leq n^2\} < \deg(C)
    \]
    for $\pp \in \Conf_{n^2}(\PP^2)\setminus \Sigma_{\mathrm{Nag}}$. Applying Lemma~\ref{lem:sym} with $\lambda = n$ gives the stated strict inequality outside a countable union of proper closed sets. Taking its union with the bad set already used for the non-strict inequality gives the single bad set in the statement. 
\end{proof}

\begin{proof}[Proof of Corollary~\ref{main-cor:examples}]
We check each of the three cases in order.
\begin{enumerate}
    \item For $X=\PP^2$ and $H=\cO_{\PP^2}(1)$, the case $r\geq10$ follows from Theorem~\ref{main-thm:nagata-I} and Lemma~\ref{lem:nagata-to-seshadri}, together with the standard upper bound on Seshadri constants. The $r=9$ case follows from Lemma~\ref{lem:endpoints} using an integral cubic.    
     \item For $X=\PP^1\times \PP^1$ and $H=\cO_{\PP^1\times\PP^1}(1,1)$ we have $h=2$. Thus, Theorem~\ref{main-thm:general-nagata} gives the required conclusion for $r\geq 4h=8$. Alternatively one may use Theorem~\ref{main-thm:technical-nagata} with $\nu=2$.
    \item Let $X\subset \PP^3$ be a smooth surface of degree $d\geq3$ and $H=\cO_{X}(1)$. Letting $g:X\to (\PP^3)^{\vee}$ be the Gauss map, we have that $g^{*}\cO_{(\PP^3)^{\vee}}(1)\cong \cO_{X}(d-1)$, and so $g$ is finite onto its image. In characteristic zero, it is generically unramified \cite[\href{https://stacks.math.columbia.edu/tag/02GF}{Tag 02GF}]{stacks-project}. Thus, by Katz \cite[Prop.~3.5]{katz73}, a general tangent plane section of $X$ has exactly one singular point, an ordinary node. Such a curve is irreducible since by B\'{e}zout's theorem a reducible plane curve of degree $d\geq3$ cannot have only one singular point that is an ordinary node. Since $H$ is very ample with $h=d\geq3$, Theorem~\ref{main-thm:technical-nagata}, applied with $\nu=1$, gives the claim for $r\geq d$. 
\end{enumerate}
\end{proof}

%%%%%%%%%%%%%%%%%%%%%%%%%%%%%%%%%%%%%%%%%%%%%%%%%%%%%%%%%%%%%%%%
\bibliographystyle{alphaurl}
\bibliography{references.bib}
%%%%%%%%%%%%%%%%%%%%%%%%%%%%%%%%%%%%%%%%%%%%%%%%%%%%%%%%%%%%%%%%

\end{document}